\documentclass[english,a4paper,12pt,oneside]{book}

\usepackage{hyperref}
\usepackage{graphicx}
\usepackage[latin1]{inputenc}
\usepackage{amsmath, amsthm, amssymb}
\usepackage{bbm}
\usepackage[english]{babel}
\usepackage{marvosym}
\usepackage{autonum}
\usepackage{color}
\usepackage{varwidth}
\usepackage{svg}
\usepackage{natbib}

\definecolor{leichtgrau}{gray}{.80}
\def\comment#1{\colorbox{leichtgrau}{{\begin{varwidth}{\dimexpr\linewidth-3\fboxsep}#1\end{varwidth}}}}

\newtheorem{theorem}{Theorem}[chapter]
\theoremstyle{definition} 
\newtheorem{definition}[theorem]{Definition} 
\theoremstyle{definition} 
\newtheorem{lemma}[theorem]{Lemma} 
\theoremstyle{definition} 

\theoremstyle{definition} 
\newtheorem{corollary}[theorem]{Corollary} 
\theoremstyle{definition}
 
\theoremstyle{definition} 
\newtheorem{assumption}[theorem]{Assumption} 
\theoremstyle{definition}

\begin{document}
\pagestyle{empty}
%%%% Title page
\begin{titlepage}
\begin{center}
\includegraphics{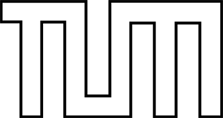}\\[3mm]
\sf
{\Large
  Technische Universit\"at M\"unchen\\[5mm]
  Department of Mathematics\\[8mm]
}
\normalsize
\includegraphics{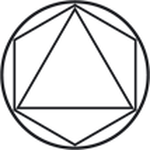}\\[15mm]

Master's Thesis\\[15mm]
{\Huge
  Mathematical foundation of
  \bigskip
  Information Field Dynamics
}

\bigskip

\normalsize

Christian M\"unch
\end{center}
\vspace*{75mm}

Supervisor: Prof. Dr. Simone Warzel
\medskip

Advisor: PD Dr. Torsten En\ss lin
\medskip

Submission Date: 14th November 2014

\end{titlepage}
%%%% The following has to be signed by hand!

\vspace*{150mm}

I assure the single handed composition of this master's thesis only supported by declared resources.
\bigskip

Garching, 
\newpage
\pagenumbering{roman}
\pagestyle{headings}
%%%% Zusammenfassung in deutscher Sprache
\section*{Zusammenfassung}
Die Vorhersage bestimmter gemessener Eigenschaften eines physikalischen Prozesses, den Daten $d(0)$, f\"ur einen zuk\"unftigen Zeitpunkt ist in vielen F\"allen w\"unschenswert. Prozesse k\"onnen h\"aufig durch eine vektorwertige Funktion $\phi$, das Signal, und eine Evolutionsgleichung modelliert werden, die jedoch oft nicht analytisch zu l\"osen ist.\\
Ein wachsendes Leistungsverm\"ogen in der Computertechnologie erm\"oglicht jedoch die immer genauere numerische Simulation von Evolutionen.\\
F\"ur Probleme sorgen hierbei einerseits die wenigen Freiheitsgrade der Daten im Vergleich zum Signal, die einen eindeutigen R\"uckschluss unm\"oglich machen, und andererseits eine h\"aufig sehr komplizierte Evolutionsgleichung.\\
Torsten En\ss lin verwendet in seiner Arbeit \citep{enss} \"uber Informations Feld Dynamik (IFD) deshalb einen wahrscheinlichkeitstheoretischen Ansatz und verfolgt dann die Idee, Evolutionsgleichungen zu diskretisieren und dadurch vereinfacht nachzumodellieren.\\
Hierbei nimmt er zu jedem Zeitpunkt $t$ zwischen der Startzeit und der Zeit $T$, zu welcher die Daten simuliert werden sollen, eine lineare Relation zwischen den Daten $d(t)$ und dem Signal $\phi(t)$ durch einen Operator $R(t)$ an. Messabweichungen werden durch einen Zufallsvektor $n(t)$ dargestellt. Auch das Signal wird zu jedem Zeitpunkt als Zufallsvektor interpretiert, dessen Wahrscheinlichkeitsverteilung diejenigen Signale gewichtet, die m\"oglicherweise zu den gemessenen Daten gef\"uhrt haben. Die Daten k\"onnen so als eine Realisierung des Zufallsvektors
\[
R(t)\phi(t) + n(t)
\]
aufgefasst werden.\\
IFD partitioniert das Zeitintervall $(0,T]$ durch $\left((t_{i},t_{i+1}]\right)_{i=0}^{2^N}$ und simuliert induktiv f\"ur jedes $i$ aus den Daten $d(t_{i})$ des vorherigen Zeitschrittes die Daten $d(t_{i+1})$, bis $d(T)$ berechnet ist. Dabei wird die bedingte Wahrscheinlichkeitsdichte $\mathcal{P}_{\phi(t_{i}) \vert d(t_{i})}$ aus Annahmen \"uber $\mathcal{P}_{\phi(t_{i})}$ und $\mathcal{P}_{n(t_{i})}$ berechnet und gem\"a\ss\ der linearisierten Evolutionsgleichung weiterentwickelt. Die so entstehende Dichte $\mathcal{P}_{\phi(t_{i}) \vert d(t_{i+1})}$ wird dann mit der bedingten Wahrscheinlichkeitsdichte $\mathcal{P}_{\phi(t_{i+1}) \vert d(t_{i+1})}$ abgeglichen. Dies geschieht durch Minimierung der relativen Entropie zwischen den Dichten bez\"uglich $d(t_{i+1})$ und legt so den simulierten Datenvektor f\"ur den Zeitpunkt $t_{i+1}$ fest. Dabei entsteht $d(t_{i+1})$ durch Matrixmultiplikation aus $d(t_{i})$. Es ergibt sich ein iterativer Algorithmus zur Berechnung von $d(T)$ aus $d(0)$.\\
In dieser Arbeit wird der Ansatz in \citep{enss} auf mathematische Weise pr\"azisiert und einzelne Teile werden korrigiert. Approximationsschritte werden dabei durch das Verschwinden des informationstechnischen Fehlers im Grenzfall $\delta t= \frac{T}{2^N}\rightarrow 0$ gerechtfertigt. Am Beispiel des eindimensionalen Klein-Gordon Feldes, periodisch \"uber $[0,2\pi)$, wird die Simulation mittels IFD illustriert. Dabei ensteht durch den Grenz\"ubergang $\delta t \rightarrow 0$ eine gew\"ohnliche Differentialgleichung f\"ur die Daten, deren L\"osung eine direkte Simulation von $d(T)$ durch $d(0)$ durch Matrixmultiplikation zul\"asst. Dadurch kann der iterative Algorithmus von IFD durch die Berechnung einer einzelnen Matrix weiter vereinfacht werden.

\newpage
\tableofcontents
\newpage

%%%% Page numbering restarts here
\pagenumbering{arabic}
\pagestyle{headings}

\pagenumbering{arabic}  
\setcounter{page}{1}
\chapter{Introduction}

\section{Necessity of numerical simulations}
Many physical processes evoke the wish in us to predict several features of them for a prospective time.  Those processes are described by the state of a system, mostly taking the form of a vector valued function, i.e. a field, which we call the signal of interest $\phi$.
The features of interest at different times $t$ are imprinted in data vectors $d(t)$. Most often, these vectors contain only a restricted amount of information about the signal $\phi(t)$ at the corresponding time. Therefore, they usually have less degrees of freedom than the signal.
Furthermore, the underlying evolution equation of the process is often not analytically solvable.
Additionally, due to the overwhelming amount of information, needed to fully specify a continuous signal, and due to the limited information in a finite data vector $d(0)$ for the present time, the initial state of the system $\phi(0)$ is usually not known precisely.\\
The rapidly growing possibilities of computer technology permit us to simulate the physical process numerically until the time of interest, starting with the measured data $d(0)$ at present, with increasing accuracy. Nevertheless, the problem that a continuous field is imperfectly represented by finite data, does not disappear.

\section{Simplification of processes and subgrid structures}
Despite that the access to computer storage already expanded a lot, it is still finite. This leads to the necessity of a discretized signal instead of the real one within simulations, because only finitely many of the degrees of freedom that determine the real process can be stored. For this reason, for a numerical simulation, the implicit assumption about some subgrid structure of the signal is generally needed. This structure allows to specify a discretized evolution equation, which models the real process of interest. By this, the real signal can be sufficiently close approximated with a function that is discrete in time and from a finite dimensional space for every time step.\\
Many physical processes are stable enough in time. And besides this, quite often already an approximation of the signal at every time step of the time discretization by a function from a finite dimensional space has a determining effect on the real signal's parameters of interest, which are contained in the data vectors at those times. Such subgrid structures, like constant, linear or polynomial interpolation, are therefore present in many cases. Examples for subgrid scale modeling approaches can be found in \citep[p.16 f.]{enss}

\section{Simulation despite of limited information}
Analogous to the reconstruction of a signal from measurements, the idea of field reconstruction from a data vector $d(0)$ at present time in computer memory came up. The reconstructed field $\phi(0)$ is then evolved to a later time $T$ and the features of interest $d(T)$ for this time can be computed from  or just read out of $\phi(T)$.\\
The difficulty here is not only a possibly very complicated evolution equation of the process, but also the limited amount of information about the signal, which comes with the initial data. In order to deal with this problem, algorithms which simulate features of the process at a prospective time have to incorporate all the present knowledge about the evolution process of the signal and about its connection to the data vector for different times.\\
Most of those simulation algorithms need explicit assumptions about the subgrid structure of the signal and by this, also about the smoothness of the physical process.\\
This is why specific algorithms are in general only suitable to simulate one or a few number of processes for which they were designed. Also the optimal algorithm for the simulation of a specific kind of problem can change with the state the simulated field is in, or if the type of features one is interested in varies.

\section{Information Field Dynamics}
Information Field Dynamics (IFD), as introduced in \cite{enss}, provides a tool to construct simulation schemes for various physical processes, without any concrete assumptions about the subgrid structure of the problem.\\
The setting of IFD assumes a linear connection between the data $d(t)$ and the signal $\phi(t)$ by a (maybe time dependent) response operator $R(t)$ for every time $t$ in between the initial time and the time $T$ for which the data $d(T)$ is simulated. The response describes a measurement process and can be generalized to an arbitrary measurable operator. Also inaccuracies of the measurement process are taken into account by adding a noise term $n(t)$, which is the realization of a random vector with some distribution $\mathbb{P}_{n(t)}$.\\ 
In a first step, IFD discretizes the evolution equation in time by partitioning the evolution interval $[0,T]$. It then aims to update the initially measured data $d(0)$ inductively, time step for time step. This is done by simulating the data at every of those intermediate times $t_{i}$, where $\left((t_{i},t_{i+1}]\right)_{i=0}^{2^N}$ is the partition of $(0,T]$. The simulated data vector $d(t_{i})$ for time step $t_{i}$ is used here to determine the data for the next time $t_{i+1}$. This way, one finally obtains a simulation of the data vector $d(T)$ for the time of interest.\\   

Picture \ref{fig1} by Torsten En\ss lin illustrates one updating step from $t_{i}$ to $t_{i+1}$ with IFD.\\
By the knowledge of the data vector $d(t_{i})$ and assumptions on the prior statistics of the underlying process $\mathcal{P}_{\phi(t_{i})}$ and the noise distribution $\mathbb{P}_{n(t_{i})}$ for the current time $t_{i}$, IFD first constructs the posterior $\mathcal{P}_{\phi(t_{i}) \vert d(t_{i})}$, i.e. the probability density over all possible configurations $\phi(t_{i})$ of the physical field, which might have resulted in $d(t_{i})$. In Picture \ref{fig1}, this is illustrated by an arrow with the label "signal inference", connecting the "data in computer memory" and the "configuration space" of the signal, the field to be simulated.\\

In a second step, this density is evolved to the next time step according to the differential operator which describes the field dynamics of the signal. The transformation formula for Lebesgue integrable functions is used here. This leads to an evolved posterior $\mathcal{P}_{\phi(t_{i+1}) \vert d(t_{i})}$. Picture \ref{fig1} represents this step with an arrow from the signal's "configuration space" at the starting time $t_{i}$ into its "configuration space" at the later time $t_{i+1}$, marked with "time evolution".\\

The simulated data vector $d(t_{i+1})$ gives us several degrees of freedom for the posterior density $\mathcal{P}_{\phi(t_{i+1}) \vert d(t_{i+1})}$ of the signal at the later time $t_{i+1}$ by its relation to the physical process. This way, $\mathcal{P}_{\phi(t_{i+1}) \vert d(t_{i+1})}$ takes the form of a function depending on $d(t_{i+1})$. Entropic matching between the posterior density at time $t_{i+1}$ and the evolved posterior is used to determine those degrees of freedom and thereby the simulation of the data $d(t_{i+1})$. This means that the relative entropy $D\left( \mathcal{P}_{\phi(t_{i+1}) \vert d(t_{i+1})} \Vert \mathcal{P}_{\phi(t_{i+1}) \vert d(t_{i})}\right)$ between the two densities is minimized w.r.t. the data $d(t_{i+1})$ at the posterior time. Because relative entropy is a measure for information discrepancy, that way the least possible amount of spurious information is added within the update from $t_{i}$ to $t_{i+1}$.\\
In Picture \ref{fig1}, this inference step is marked with an arrow from the "configuration space" at time $t_{i+1}$ to the new "data in computer memory", labeled with "entropic matching".\\

The result of the described procedure is a concrete connection in form of a matrix application between the initial data $d(t_{i})$ and the simulated vector $d(t_{i+1})$, as illustrated by the arrow "simulation scheme" in Picture \ref{fig1}, connecting the old "data in computer memory" and the new.

\section{Content of this work}
In this work, the measure theoretical fundament and the necessary probabilistic framework for IFD are introduced in Chapter \ref{chap:Background on measure theory}.
The notation for the rest of the work is established here and essential properties are developed.\\
Amongst others, this includes a definition for the prior $\mathcal{P}_{\phi}$, the posterior $\mathcal{P}_{\phi \vert d}$, the evidence $\mathcal{P}_{d}$ and the likelihood $\mathcal{P}_{d \vert \phi}$ in case $\phi$ is the unknown quantity (signal) and $d$ the known one (data). Also Bayes's Theorem for the setting of IFD is derived.\\
In Chapter \ref{chap:setting}, the general setting for IFD is described, consisting of the signal with its evolution equation and its linear connection to the data by a response operator.\\
The developed environment is then used in Chapter \ref{chap:Updating the data}, to step by step imbed the physical language used in \citep{enss} into a mathematical framework. Also a general approach for the approximation of an evolution equation is given, serving as a base to describe the construction of a simulation scheme with IFD. Simulation errors in the various steps are pointed out, to give an idea where inaccuracies come in and which steps therefore lead to the necessity of many degrees of freedom of the signal and of many time steps within the simulation. In Chapter \ref{chap:Maximum_Entropy_Principle}, an interpretation of differential entropy as a measure of information is given and the idea of entropic matching is explained. Chapter \ref{chap:Example: Klein-Gordon field} illustrates IFD in an example scenario. In this chapter, also the mentioned matrix relation for the update steps of IFD is derived. It allows to simulate a data vector $d(T)$ which averages a Klein-Gordon field with one dimension in space and periodic over $[0,2\pi)$.\\
In the end of the example, also a non-iterative equation for the direct computation of $d(T)$ from $d(0)$ is constructed. This is reached by the fact that the original problem converts into an ordinary differential equation for the data, if the simulation time steps get infinitesimally small.\\

The work closes with a conclusion in Chapter \ref{chap:Conclusion}, summarizing the derived results.\\

Chapter \ref{chap:Dictionary for physicists} contains a small dictionary, serving as a reference work to look up the notation used in this work. All expressions here are translated into the corresponding terms which are typically used in physical literature.\\

\newpage
The following picture by Torsten En\ss lin illustrates the idea of IFD.
\begin{figure}[ht]
	\centering
  \includegraphics[scale=1]{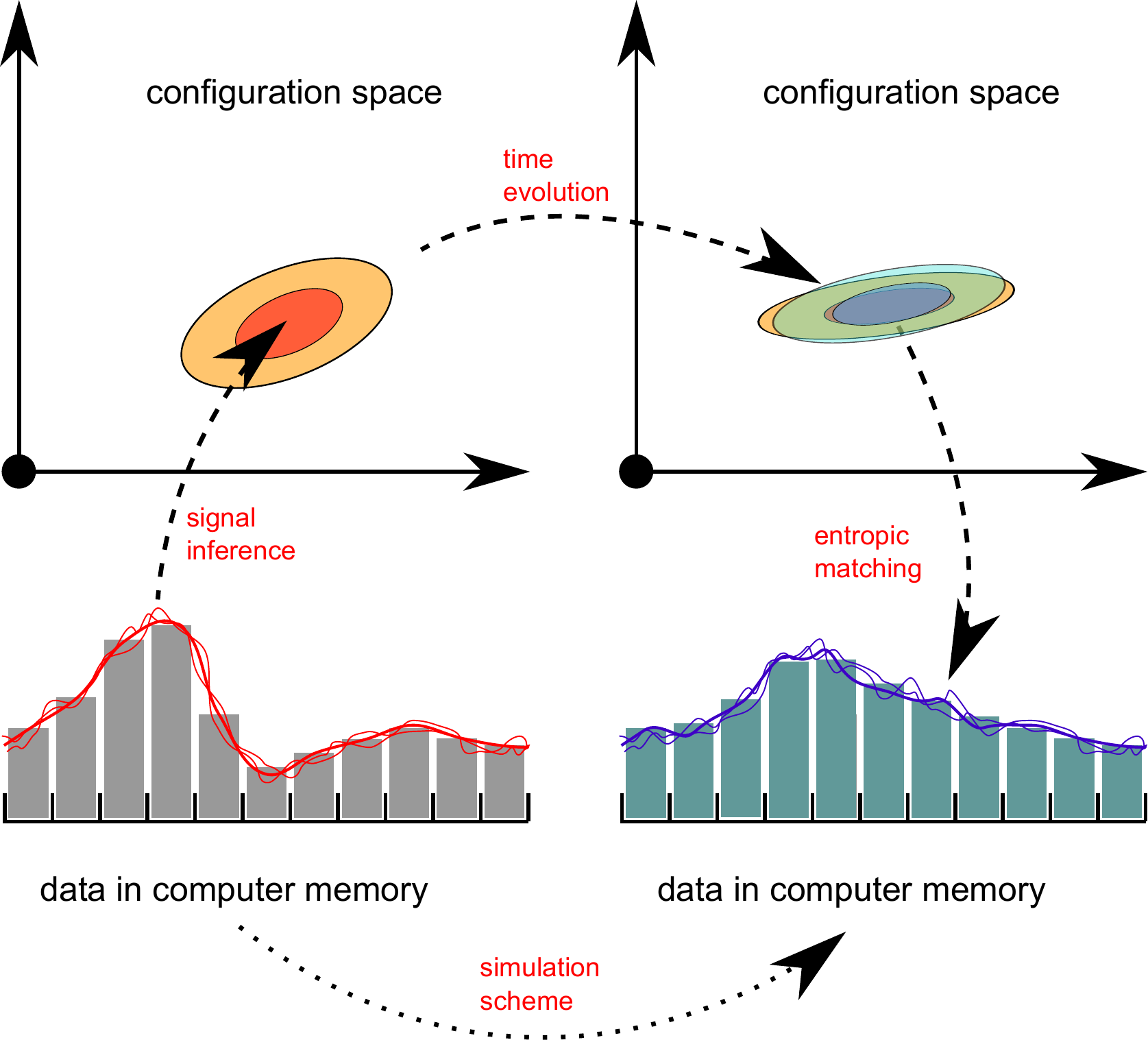} 
	\caption{IFD}
	\label{fig1}
\end{figure}

\chapter{Background on measure theory} 
\label{chap:Background on measure theory} 
\setcounter{section}{-1}
\section{Motivation}
There are many quantities we are interested to know about the world. In the following, such a quantity is called the signal of interest. Some data vector which comes from measuring the signal, constrains the degrees of freedom of the latter within the space of all possible signal configurations. Measurement noise and a large number of signal degrees of freedom often inhibit invertibility of the measurement process, in that several signal configurations could have resulted in the measured data. Therefore, one might consider to work with probability distributions for the signal, representing how likely it is that our measured data vector comes from a certain signal configuration. In this context, technical terms like prior distribution, likelihood, evidence and posterior distribution often appear for specific probabilities.\\
Now there are different concepts to imbed such a problem into a probabilistic setting. One might see probability as a frequency of how often some event happens. This is the stochastic approach to probability theory \cite{fisher}.\\
One might also derive a probabilistic setting as an extension of logic, where the probability of an event $A$ is assumed to indicate how likely it is that $A$ is true. In this case, $A$ is a statement, which one assigns some plausibility value to \cite{cox}.\\
However, all these approaches result in a problem of measure theory and can be recovered within this domain. We therefore should spend some effort to properly define our probabilistic setting in terms of measure theory.\\
For the ones who already feel familiar with those topics, this chapter and especially Section \ref{sec:expectation} might be seen as an introduction into the notation that is used throughout this work.

\section{First terms and definitions}
We start with some general terms and definitions from measure theory, which will consistently be needed in this work.
\begin{definition}[Measure theory]
The following definitions can be found in \cite[p.1 ff.]{kall}.
\begin{itemize}
\item For some abstract space $\Omega$, a $\sigma$-field  $\mathcal{A}$ in $\Omega$ is a nonempty collection of subsets of $\Omega$ which is closed under countable unions and intersections and under complementation.\\
\comment{$\Omega$ can be seen as an index set which labels possible events. In a physical setting, the space in which one defines a $\sigma$-field might for example be a phase space or the configuration space of some field.}

\item A set $A\in \mathcal{A}$, where $\mathcal{A}$ is a $\sigma$-field in $\Omega$, is called a measurable set.\\
\comment{Measurable sets are (classes of) events in our probabilistic setting and we want to assign probabilities to them. More general, one could also define an event, if totally specified, as any $x\in \Omega$. Physical sizes, though, often are adhered some inaccuracy to, so that in our case it makes more sense to define probabilities for classes of events, i.e. for measurable sets $A\in \mathcal{A}$ which embrace single elements of $\Omega$.}

\item If $\mathcal{C}$ is an arbitrary class of subsets of $\Omega$, there is a smallest $\sigma$-field containing $\mathcal{C}$. We call this $\sigma$-field $\sigma(\mathcal{C})$.\\
In case $\Omega = \mathbb{R}^{n}$ for some $n\in \mathbb{N} $, we call $\mathcal{B}(\mathbb{R}^{n})$ the Borel $\sigma$-field which is generated by the set of open sets in $\mathbb{R}^{n}$. If not especially specified, we always mean this $\sigma$-field if we talk about $\mathbb{R}^{n}$.
In general, we define the Borel $\sigma$-field $\mathcal{B}(S)$  in a topological space $S$, where the topology on $S$ is given by the specification of open sets in $S$. $\mathcal{B}(S)$ then means the $\sigma$-field which is generated by the set of open sets in $S$ w.r.t. this topology.\\
\comment{
In physics, where one mostly works with topological spaces, the Borel $\sigma$-field is very often the natural choice.
}

\item A measurable space is a pair $\left( \Omega, \mathcal{A}\right)$, where $\Omega$ is a space and $\mathcal{A}$ is a $\sigma$-field in $\Omega$.

\item A function between two measurable spaces $\ \left(\Omega, \mathcal{A}\right)$ and $\left(S, \mathcal{S}\right)$ is called $\mathcal{A}$-measurable, or just measurable, if the preimage of any set in $\mathcal{S}$ is contained in $\mathcal{A}$.\\
\comment{Amongst others, to consider the Borel $\sigma$-field in a topological space is the natural choice, because this way in particular every continuous mapping between metric spaces - like for example normed spaces - is measurable.\\
The concept of measurable functions allows us to compare measurable sets, i.e. events, in different measurable spaces, by taking the image or preimage of a set under a measurable map between those spaces.}

\item Two measurable spaces $(S,\mathcal{S})$ and $(U,\mathcal{U})$ are Borel isomorphic, if there is a bijection $f:S \rightarrow U$ such that $f$ and $f^{-1}$ are measurable.

\item A space $(S,\mathcal{S})$ which is Borel isomorphic to a Borel subset of $\left[0,1\right]$ is called a Borel space.\\
\comment{We will need this property for several theorems within the derivation of our probabilistic setting.}

\item Given a measurable space $\left( \Omega, \mathcal{A}\right)$, a measure is a function $\mu: \mathcal{A} \rightarrow \mathbb{R}_{+}\cup\lbrace \infty \rbrace$ with $\mu(\emptyset)=0$ and which is countably additive, i.e.:

\[
	\mu \left( \cup_{k \geq 1}A_{k} \right) = \sum_{k \geq 1} \mu (A_{k}) \text{ for } A_{1},A_{2},\cdots\in \mathcal{A}\ \text{disjoint}.
\]
$\mu$ is called $\sigma$-finite, if there is a partition of disjoint sets $A_{1},A_{2},... \in \mathcal{A}$ of $\Omega$ , such that $\mu (A_{n})< \infty$ for all $n$.\\
It is further called finite, if $\mu (\Omega) < \infty$, and a probability measure, if $\mu (\Omega) = 1$.

\item A triple $\left(\Omega, \mathcal{A}, \mu\right)$ is then called a measure space. It is called a probability space, if $\mu(\Omega) = 1$.\\
\comment{This way, one can assign a probability to an event $A\in \mathcal{A}$, which lies in the interval $[0,1]$. $\mu(A)=0$ means that $A$ is impossible, whereas $\mu(A)=1$ indicates that it is certain.\\
In our case, probability measures will mostly have the form of integral measures $f\cdot \lambda$, where $f$ will be a density function and $\lambda$ the Lebesgue measure on $\mathbb{R}^{n}$ for some $n\in \mathbb{N}$. The latter terms will be defined soon.}

\item For two measurable functions $f$ and $g$ from a measure space $\left(\Omega, \mathcal{A}, \mu\right)$ into a measurable space $\left(S, \mathcal{S}\right)$, we say that $f$ and $g$ are equal (or also agree) $\mu$-almost sure ($\mu$-a.s.) or $\mu$-almost everywhere ($\mu$-a.e.), if the set $N = \lbrace \omega\in \Omega: f(\omega)\neq g(\omega) \rbrace$ is of measure zero w.r.t. $\mu$, i.e. $\mu(N)=0$. One also calls a set $N$ with the latter property a $\mu$-null set. If it is clear which measure we refer to, we will also sometimes skip the $\mu$. In general, one says that a property holds a.e. or a.s., if it holds everywhere except for a null set w.r.t. the underlying measure.\\
In this work, we always assume measure spaces $\left(\Omega, \mathcal{A}, \mu\right)$ to be complete, in the sense that we have $\mathcal{A} = \mathcal{A}_{\mu} = \sigma(\mathcal{A},\mathcal{N}_{\mu})$, where $\mathcal{N}_{\mu}$ is the class of all subsets of $\mu$-null sets of $\mathcal{A}$. As we always consider Borel spaces, this assumption can be made w.l.o.g. by \cite[Lemma 1.25]{kall}, which states that a function $f$ from $\left(\Omega, \mathcal{A}, \mu\right)$ into a measurable space $\left(S, \mathcal{S}\right)$ is $\mathcal{A}_{\mu}$-measurable, if and only if there is some $\mathcal{A}$-measurable function $g$ with $f=g$ $\mu$-a.e. The remark after this lemma also assures the existence of a unique extension of $\mu$ from $\mathcal{A}$ to $\mathcal{A}_{\mu}$.\\
\comment{ Fields in physics, e.g. functions over physical spaces like the density of matter in the cosmos or the wave function of an electron, are most often only determined up to null sets with respect to some measure on a measurable space. This shall mean that a function is only a representative out of the equivalence class of functions which agree almost everywhere. This is not a problem though, because the influence of physical fields is most often measured by integration over the space they are defined on, with respect to the underlying measure. Ambiguities of the field values on null sets do not change the value of an integral, as we will see, and are therefore physically irrelevant. One might therefore also argue, that physics in some cases should better be described in terms of equivalence classes of a.e. equal functions, where the underlying measure will mostly be the Lebesgue measure.
}

\item For a measure space $\left(\Omega,\mathcal{A},\mu\right)$ and a simple, nonnegative function $f:\Omega \rightarrow \mathbb{R}$, i.e. 
\[
f = c_{1}1_{A_{1}}+...+c_{n}1_{A_{n}},
\]
with $n\in \mathbb{N}, c_{1},...,c_{n}\in \mathbb{R}_{+}, A_{1},...,A_{n}\in \mathcal{A}$, $\cup_{k=1}^{n}A_{k} = \Omega$ and 
\[
1_{A_{k}}(\omega) = 
\begin{cases}
1,\text{ if }\omega\in A_{k}\\
0 \text{ else},
\end{cases}
\] one can define the integral 
\[
\mu f = \int\limits_{\Omega} f d\mu = \int\limits_{\Omega} f(\omega) \mu(d\omega) := c_{1}\cdot \mu(A_{1})+...+c_{n}\cdot \mathbb{P}(A_{n})
\]
\cite[cf.][p.10]{kall}.\\
This integral can firstly be generalized to measurable functions $f:\Omega \rightarrow \mathbb{R}_{+}$ by approximation with simple functions $f_{n}:\Omega \rightarrow \mathbb{R}_{+}$ from below \cite[see][Lemma 1.11]{kall}:
\[
\lim\limits_{n\rightarrow\infty}f_{n}(\omega) = f(\omega),\ \ \forall\omega\in \Omega,
\]
where
\[
0\leq f_{n}(\omega) \leq f(\omega),\ \ \forall\omega\in \Omega,\ \forall n\in \mathbb{N}.
\]
One can show that the definition of the general integral as a limit of integrals over these simple functions,
\[
\int\limits_{\Omega} f d\mu := \lim\limits_{n\rightarrow\infty} \int\limits_{\Omega} f_{n} d\mu,
\]
is well-defined and unique.\\
A general measurable function $f:\Omega \rightarrow \mathbb{R}$ is called integrable, $f\in L_{1}(\Omega,\mathcal{A},\mu)$, if 
\[
\mu\vert f\vert < \infty,
\]
using the fact that in this case $f=f^{+}-f^{-}$, where $f^{+}=\max\lbrace f, 0\rbrace$ and $f^{-}=\min\lbrace f, 0\rbrace$ are both measurable \cite[cf.][Lemma 1.9]{kall} and integrable and one can define 
\[
\mu f:=\mu f^{+}-\mu f^{-}.
\]
If $f$ is a measurable function into the complex plane $\mathbb{C}$, we define the integral of $f$ as
\[
\mu f:=\mu \mathrm{Re}(f) + i\cdot \mu \mathrm{Im}(f),
\]
in case these integrals exist.
For $A \in \mathcal{A}$, one also defines 
\[
\int\limits_{A} f d\mu  := \int\limits_{\Omega} f\cdot1_{A} d\mu .
\]
For a measure space $\left(\Omega,\mathcal{A},\mu \right)$ and $p>0$, $L_{p}\left(\Omega,\mathcal{A},\mu \right)$ is the space of all equivalence classes of $\mathcal{A}$-measurable functions $f:\Omega \rightarrow \mathbb{R}$ that agree $\mu $-a.e. and satisfy
\[
\Vert f \Vert_{p} = \left( \mu  \vert f \vert^{p} \right)^{1/p} < \infty.
\]
In some cases we also talk about $L_{p}\left(\Omega,\mathcal{A},\mu \right)$ as the space that is constructed from complex valued functions, but this will be mentioned then.\\
For a measurable function $f=(f_{1},...,f_{n})^T:\Omega \rightarrow \mathbb{R}^{n}$, with $n\in \mathbb{N}$, or analogously for $f=(f_{1},...,f_{n})^T:\Omega \rightarrow \mathbb{C}^{n}$, we define
\[
\mu f := (\mu f_{1},...,\mu f_{n})^T,
\]
in case all the integrals exist.\\
\comment{As mentioned above, integrals are needed to calculate a property of a function. We will see that they are also necessary to define the expectation value of a random element and its variance. Those quantities give a hint on which value one has to expect for the random element in the long run and how large the deviation of this value is in average. By \cite[Lemma 1.24]{kall}, if two functions agree a.e., then also their integrals coincide. As already mentioned, this justifies the interpretation of physical fields as equivalence classes of functions that are only a.e. determined and that are defined only up to null sets.}

\item If for two measures $\mu$ and $\nu$, one has the relation 
\begin{align}
\nu = f \cdot \mu\label{density}
\end{align}
for a measurable function $f$, meaning that $\nu(A) = \int\limits_{A} f d\mu$ for all measurable sets $A$, we call $f$ the $\mu$-density or $\mu$-derivative of $\nu$. If $\nu$ is a probability measure, so that $f \geq 0$ $\mu$-a.s. and $\mu f = 1$, we call $f$ the probability density of $\nu$ w.r.t. $\mu$.\\
\comment{In this work, we will mostly deal with probability densities of random elements with respect to the Lebesgue measure.}

\item For two measurable spaces $\left(S,\mathcal{S}\right)$ and $\left(U,\mathcal{U}\right)$, the $\sigma$-field $\mathcal{S}\otimes \mathcal{U}$ on the product space $S\times U$ is defined as the $\sigma$-field that is generated by sets of the form $A\times B$, where $A\in \mathcal{S}$ and $B\in \mathcal{U}$.\\
\comment{This definition is already a preparation for the joint distribution of two random elements.}
\end{itemize}
\end{definition}

We will need Fubini's Theorem \cite[Theorem 1.27]{kall} in many calculations.
\begin{theorem}[Product measures and iterated integrals, Lebesgue, Fubini, Tonelli]\label{fubini}
For any $\sigma$-finite measure spaces $\left(S,\mathcal{S},\mu\right)$ and $\left(U,\mathcal{U},\nu\right)$, there exists a unique measure $\mu\otimes \nu$ on $\left(S\times U,\mathcal{S}\otimes \mathcal{U}\right)$, satisfying
\[
(\mu \otimes \nu)(B\times C) = \mu B \cdot \nu C, \: \forall B\in \mathcal{S}, C\in \mathcal{U}.
\]
Furthermore, for any measurable function $f: S\times U \rightarrow \mathbb{R}_{+}$, it is
\[
(\mu \otimes \nu)f = \int\limits_{S} \int\limits_{U} f(s,u)\nu(du) \mu(ds) = \int\limits_{U} \int\limits_{S} f(s,u)\mu(ds) \nu(du).
\]
The last relation remains valid for any measurable function $f:S\times U \rightarrow \mathbb{R}$ with 
\[(\mu \otimes \nu)\vert f \vert < \infty.\]
\end{theorem}
By \cite[Theorem A.2.4]{werner}, this theorem stays true for measurable functions $f: S\times U \rightarrow \mathbb{C}$.\\
\comment{In our setting, we will use this theorem for example for the derivation of the posterior - the distribution of our signal given some measured data - from their joint distribution.
By Bayes' Theorem, the posterior can then also be computed from the prior - the distribution of the prior knowledge of the signal - the likelihood - the distribution of the data, given the signal - and the evidence - the distribution of the data.}\\

In the setting of Theorem \ref{fubini}, we call $\mu\otimes \nu$ the product measure of $\mu$ and $\nu$ on $\left(S\times U,\mathcal{S}\otimes \mathcal{U}\right)$.\\

\comment{The joint distribution of two random elements $\xi$ and $\eta$, as we will see, is a product measure, if and only if they are independent. In this case, if the distributions of the random elements, as defined in the next section, are densities $\mathcal{P}_{\xi}$ and $\mathcal{P}_{\eta}$ w.r.t. the Lebesgue measure $\lambda$ and if we write $ds$ for $\lambda(ds)$, one can observe the similarity to the notation
\[
\int\limits_{S} \mathcal{P}_{\xi}(s) \int\limits_{U} \mathcal{P}_{\eta}(u) f(s,u)du ds,
\]
which is quite common in physics.
}

\section{Random elements and expectation value}
With this background from measure theory, we are now able to turn our attention towards the concept of random elements, the fundamental component for the rest of this work.\\
To emphasize that we are in a probabilistic setting, we often call the measure in a probability space $\mathbb{P}$ and the measure space $\left(\Omega, \mathcal{A}, \mathbb{P} \right)$. 
\begin{definition}[Random elements and expectation value]
All the definitions can be found in \cite[p.22 ff.]{kall}.
\begin{itemize}
\item A random element $\xi : \Omega \rightarrow S$ is a measurable function between a probability space $\left(\Omega, \mathcal{A}, \mathbb{P} \right)$ and a measurable space $\left(S, \mathcal{S}\right)$. In this case $\mathbb{P}\circ\xi^{-1}$, the so called distribution of $\xi$, which we denote as $\mathbb{P}_{\xi}$, is a probability measure on $\left(S,\mathcal{S}\right)$, so that $\left(S,\mathcal{S},\mathbb{P}_{\xi}\right)$ gets a probability space. We will therefore mostly only consider the latter space if we talk of a random element $\xi$ and assume all random elements to be defined on the same abstract probability space $\left(\Omega, \mathcal{A}, \mathbb{P} \right)$.\\
If $S=\mathbb{R}$, we call $\xi$ a random variable and if $S=\mathbb{R}^{n}$ for some $n\in \mathbb{N}$, we call it a random vector.\\
\comment{
The underlying abstract space $\left(\Omega, \mathcal{A}, \mathbb{P} \right)$ is sometimes needed to compare the probability of events for different random variables and to compute new distributions from existing ones, like the joint distribution, which is defined in the next bullet point.\\
This space is rarely constructed explicitly, its presence is often postulated and its existence is then shown for special cases. In physics, "the wave function of the Universe" might be regarded as an analogous allegory.\\
An element $A\in \mathcal{A}$, which the probability $\mathbb{P}(A)$ is assigned to, can be seen as a label for the event of a random element $\xi$ in a measurable space $(S,\mathcal{S})$ to be found within the set $B\in \mathcal{S}$, with $A = \xi^{-1}(B)$. It is thus a mark which assigns the probability $\mathbb{P}(A)$ to $B$.\\
One wants to know the probability of such an event $B\in \mathcal{S}$. Because $\mathbb{P}(A)$ makes a statement about how likely $A$ is, one then knows the likelihood of $\xi$ to hit the set $B$ within a realization.\\
By this, one can make predictions about the mean value of $\xi$, or of some function $f(\xi)$, within many realizations. Conversely, one often takes estimates for such expectation values from many realizations, to make statements about the distributions of random elements and to model them. By this, one can make more or less reliable predictions for prospective realizations. This is one connection between the measure theoretical approach to probability theory and the frequency dependent one. One can find a similar connection between measure theory and the approach to probabilities as an extension of logic.\\
In this way, our setting includes different concepts of probability theory and forms a generalizing base in this area.
}

\item For a random element $\xi$ from $\left(\Omega, \mathcal{A}, \mathbb{P} \right)$ to $\left(S, \mathcal{S}\right)$, we define the generated $\sigma$-field  as 
\[
\sigma(\xi) := \lbrace \xi^{-1}(B): B \in \mathcal{S} \rbrace.
\]
For a measurable set $A\in \mathcal{S}$, we write $\left\lbrace \xi\in A  \right\rbrace$ or $\xi^{-1}(A)$ for $\left\lbrace \omega\in \Omega: \xi(\omega)\in A  \right\rbrace$.\\
\comment{Events (or event classes) will most often take this form, because as already mentioned, we are interested in the probability of some random element - like the signal - to attain values within some measurable set. }

\item For two random elements $\xi$ and $\eta$ into measurable spaces $\left(S,\mathcal{S}\right)$ and $\left(U,\mathcal{U}\right)$, by \cite[Lemma 1.8]{kall}, the function $\omega \mapsto \left(\xi(\omega),\eta(\omega)\right)$ is measurable  on the product space $S\times U$, endowed with the $\sigma$-field $\mathcal{S}\otimes \mathcal{U}$. Therefore, one can define the joint distribution $\mathbb{P}_{\xi,\eta}$ as the distribution of this random element.

\item Two random elements $\xi$ and $\eta$ are called independent, if for the generated $\sigma$-fields $\sigma(\xi)$ and $\sigma(\eta)$ it holds
\[
\mathbb{P} (A \cap B) = \mathbb{P}(A) \cdot \mathbb{P}(B),\: \forall A\in \sigma(\xi),\: \forall B\in \sigma(\eta).
\]

\item If $\xi$ is a random element in $\left(S,\mathcal{S}\right)$ and $f:S \rightarrow \mathbb{R}$ is in $L_{1}\left( S, \mathcal{S},\mathbb{P}_{\xi} \right)$, one can define the expectation value of $f\circ \xi$ as
\[
\mathbb{E}_{\xi}\left[ f \right]=\mathbb{E}\left[ f(\xi) \right] := \mathbb{P}_{\xi}f.
\]
If $f$ is complex valued, the expectation value is defined the same way, and by the definition of the integral for complex valued functions, it then splits up into
\[
\mathbb{E}_{\xi}\left[ f \right]= \mathbb{P}_{\xi}\mathrm{Re}(f) + i\cdot \mathbb{P}_{\xi}\mathrm{Im}(f).
\]
In the case that $f:S \rightarrow \mathbb{R}^{n}$ for some $n \in \mathbb{N}$, i.e. $f=(f_{1},...,f_{n})^T$, by \cite[Lemma 1.8]{kall}, $f$ is measurable, if and only if all the $f_{i}$ are measurable and in case all the integrals exist, we define the expectation value of $f\circ \xi$ as
\[
\mathbb{E}_{\xi}\left[ f \right]=\mathbb{E}\left[ f(\xi) \right]:=\left(\mathbb{E}_{\xi}\left[f_{1}\right],...,\mathbb{E}_{\xi}\left[f_{n}\right]\right)^T.
\]
The analogous definition is made for complex valued mappings.\\
One also talks of a random variable $\xi$ to be in $L_{p}$ for $p>0$, if it is in $L_{p}\left( \Omega,\mathcal{A}, \mathbb{P} \right)$.\\
For $A \in \mathcal{A}$ and a random element $\xi$ in $\left(S,\mathcal{S}\right)$, one writes
\[
\mathbb{E}\left[\xi; A \right]:= \int\limits_{\Omega} \xi\cdot1_{A} d\mathbb{P}
\]
for the expectation value of $\xi_{A}:=\xi \cdot 1_{A}$, provided that the integral exists. The latter definition often comes in the form
\[
\mathbb{E}\left[\xi; \left\lbrace \xi \in A \right\rbrace \right]= \int\limits_{S} s \cdot1_{A}(s) \mathbb{P}_{\xi}(ds),
\]
for a set $A\in \mathcal{S}$.\\
\comment{Very often, one estimates certain expectation values $\mathbb{E}_{\xi}\left[ f \right]$ with measurements. Together with additional assumptions, one then makes a guess on the distribution of a random element in the statistical interpretation of probabilities or on its distribution, understood as the representation of some likelihood, if probabilities are seen as logic extended uncertainty. This way, amongst other things, expectation values can be seen as a connection between reality, realizations of a random element and probability theory. Physicists often use the notation
\[
\langle f \rangle_{(\xi)}:= \mathbb{E}_{\xi}[f].
\]
}
\end{itemize}
\end{definition}

For the existence issue, we note that in this work all measurable spaces are Borel spaces. \cite[Theorem 2.19,\ (existence,\ Borel)]{kall} states, that for any probability measures $\mu_{1},\mu_{2},...$ on Borel spaces $S_{1},S_{2},...$, there exist independent random elements $\xi_{1}, \xi_{2},...$ on $([0,1],\lambda_{\mathbb{R}})$, where $\lambda_{\mathbb{R}}$ is the Lebesgue measure, with distributions $\mu_{1},\mu_{2},...$.\\
In particular, the existence of all random elements $\xi$, where $\mathbb{P}_{\xi}$ is an arbitrary probability measure on a Borel space, is assured.

\section{Conditional expectation and conditional probability}
\label{sec:expectation}
Now, we introduce the concept of conditional expectation and conditional probability. These concepts will be needed, when the outcome of a certain random element is known or can be restricted to some measurable set and if we want to derive a distribution which represents the outcome of another random element, given this information.\\
\comment{In our case, as an example, we want to make a statement about the distribution of our signal, given the measured data, if we have some prior information on the distribution of the signal and if we know the distribution of the data, given the signal.}\\

Before we come to the main definition of this section, we need to introduce a technical term \cite[cf.][p.19]{kall}.
\begin{definition}[Probability kernel]
Given two measurable spaces $\left( S,\mathcal{S} \right)$ and $\left(U,\mathcal{U} \right)$, a mapping
\[
\mu: S \times U \rightarrow \mathbb{R}_{+}\cup \infty
\]
is called a (probability) kernel from $S$ to $U$, if the function $\mu_{s}B = \mu(s,B)$ is $\mathcal{S}$-measurable in $s\in S$ for fixed $B \in \mathcal{U}$ and a (probability) measure in $B \in \mathcal{U}$ for fixed $s\in S$.\\
A kernel on the basic probability space $\left(\Omega,\mathcal{A},\mathbb{P}\right)$ is called a random measure \cite[p.83]{kall}.
\end{definition}
\comment{Our setting will assure the existence of a probability kernel which agrees a.s. with the conditional probability, as introduced in the next definition. We will further construct our probability kernel as an integral measure w.r.t. the Lebesgue measure for its second argument and this way, derive the more common concept of conditional probability as a density. Conditional means, that we are only interested in a coarser set of possibilities, so that for example we are left with classes of events only in $\mathcal{F}:=\sigma(F)$, for some $F\in \mathcal{A}$, and the restriction of any conditional random element to $\mathcal{F}$. 
One has to consider the whole $\sigma$-field $\mathcal{F}$ and not only $F$ for the concept of conditional probabilities, because a probability measure is defined on a $\sigma$-field and not only on a single measurable set.
}\\

With the previous definition in mind, we now introduce the so called conditional expectation \cite[cf.][p.81]{kall}. Assume first that we have some probability space $\left(\Omega, \mathcal{A},\mathbb{P}\right)$ and an arbitrary sub-$\sigma$-field $\mathcal{F}$ of $\mathcal{A}$, i.e. $\mathcal{F}$ is a $\sigma$-field and each set in $\mathcal{F}$ is also contained in $\mathcal{A}$. One then introduces the closed linear subspace $M$ within the space $L_{2}(\Omega,\mathcal{A},\mathbb{P})$, including all random variables $\eta\in L_{2}(\Omega,\mathcal{A},\mathbb{P})$, that agree a.s. with some random variable in $L_{2}(\Omega,\mathcal{F},\mathbb{P})$. \cite[Theorem 1.34]{kall} states, that for any $\xi \in L_{2}(\Omega,\mathcal{A},\mathbb{P})$, there exists an a.s. unique random variable $\eta \in M$, such that $\xi - \eta \bot M$, i.e. $\mathbb{P}\left[ (\xi - \eta)\cdot \zeta \right] = 0\ \forall\, \zeta\in M$. This leads us to the definition of the conditional expectation $\mathbb{E}\left[ \xi \vert \mathcal{F} \right]$ as an arbitrary $\mathcal{F}$-measurable version of this $\eta$, i.e. one representative out of the equivalence class of random elements in $L_{2}(\Omega,\mathcal{F},\mathbb{P})$ which agree a.e. with $\eta$.\\
\comment{
Note that in this last passage we always talked about random variables, i.e. measurable functions from $\Omega$ to $\mathbb{R}$.
That is, the multiplication under the integral in
\[
\mathbb{P}\left[ (\xi - \eta)\cdot \zeta \right]
\]
from above is just a scalar multiplication, but the whole expression is a scalar product in the Hilbert space $L_{2}(\Omega, \mathcal{A}, \mathbb{P})$.\\
Remember also the definition of $L_{2}(\Omega,\mathcal{A},\mathbb{P})$ as the space of equivalence classes of $\mathbb{P}$-a.e. coinciding and $\mathcal{A}$-measurable functions $f$, for which
\[
\mathbb{P}\vert f \vert^{2}< \infty.
\]
That is, a function $f$ lies in this space, if $\sigma(f)=\lbrace f^{-1}(B): B\in \mathcal{B}(\mathbb{R}) \rbrace \subset \mathcal{A}$, and if the above integral is finite.\\
An $\mathcal{F}$-measurable function is not necessarily $\mathcal{A}$-measurable. The subspace $M$ is the space of functions which are $\mathcal{A}$-measurable and agree to some element out of $L_{2}(\Omega,\mathcal{F},\mathbb{P})$ on every set in $\mathcal{F}$ which is not a null set w.r.t. $\mathbb{P}$. This subspace is needed, precisely because the integral is only defined for one $\sigma$-field at a time, so that one has to imbed $L_{2}(\Omega,\mathcal{F},\mathbb{P})$ into $L_{2}(\Omega,\mathcal{A},\mathbb{P})$ first, before one can compare elements of the two spaces. An $\mathcal{A}$-measurable function is automatically $\mathcal{F}$-measurable. This is not the case the other way around. So if we want to define an element in $L_{2}(\Omega,\mathcal{F},\mathbb{P})$ by some $\eta$ in $L_{2}(\Omega,\mathcal{A},\mathbb{P})$, we need $\eta$ to be contained in $M$.\\
The conditional expectation $\mathbb{E}\left[ \xi \vert \mathcal{F} \right]$ is the $\mathbb{P}$-a.s. determined $\mathcal{F}$-measurable random variable, that equals $\eta$ from above $\mathbb{P}$-a.e. In this sense, one shrinks the $\sigma$-field of possible events $\sigma(\xi)=\lbrace \xi^{-1}(B): B\in \mathcal{B}(\mathbb{R}) \rbrace$ to $\sigma \left( \lbrace \xi^{-1}(B): B\in \mathcal{B}(\mathbb{R}) \rbrace \right) \cap \mathcal{F}$ and allows the new random variable $\mathbb{E}\left[ \xi \vert \mathcal{F} \right]$ to only hit sets whose preimage is contained within this smaller $\sigma$-field, i.e. that are labeled by sets in it. This way, one can incorporate new knowledge about a random variable and reduce uncertainty, by diminishing its $\sigma$-field of possible outcomes.\\
For example, consider the case when we want to condition on some event $F\in \mathcal{A}$ with $0<\mathbb{P}(F)<1$. In this case, we set $\mathcal{F}:= \lbrace F, \Omega\backslash F, \Omega, \emptyset \rbrace$. For $\xi \in L_{2}(\Omega,\mathcal{A},\mathbb{P})$, one then has
\[
\mathbb{E}[\xi \vert \mathcal{F}] = \dfrac{\mathbb{E}[ \xi; F]}{\mathbb{P}(F)} \cdot 1_{F} + \dfrac{\mathbb{E}[ \xi; \Omega\backslash  F]}{\mathbb{P}(\Omega\backslash F)} \cdot 1_{\Omega\backslash F}
\]
$\mathbb{P}$-a.s.\\
$\mathbb{E}[\xi \vert \mathcal{F}]$ can only attain two values and the expectation value of $\xi$, given that $F$ is true, can easily be computed from the conditional expectation, because it is
\[
\mathbb{E}[\mathbb{E}[\xi \vert \mathcal{F}]; F] = \mathbb{E}[\xi; F].
\]
}\\

By \cite[Theorem 5.1]{kall}, the $L_{2}$-projection $\mathbb{E}\left[.\vert \mathcal{F}\right]$ can be extended to an a.s. unique linear operator $\mathbb{E}\left[ . \vert \mathcal{F} \right] : L_{1}(\Omega,\mathcal{A},\mathbb{P}) \rightarrow L_{1}(\Omega,\mathcal{F},\mathbb{P})$, with the property 
\[
\mathbb{E}\left[\mathbb{E}\left[\xi \vert \mathcal{F} \right]; A \right] = \mathbb{E}\left[\xi; A \right],\: \xi \in L_{1}(\Omega,\mathcal{A},\mathbb{P}),\: A\in \mathcal{F}. \label{condexpec}
\]
In most of the cases, $\mathcal{F}$ will be the $\sigma$-field $\sigma(\eta)$, generated by some other random element $\eta$ which maps into a measurable space $\left(U,\mathcal{U} \right)$.
We then write $\mathbb{E}\left[ \xi \vert \eta \right]$ for $\mathbb{E}\left[ \xi \vert \sigma(\eta) \right]$.\\
Now we have enough background to define the conditional probability \cite[cf.][p.83]{kall}.
\begin{definition}[Conditional probability]
The conditional probability of an event $A\in \mathcal{A}$, given a $\sigma$-field $\mathcal{F} \subset \mathcal{A}$, we define as
\[
\mathbb{P}\left[A \vert \mathcal{F}\right] := \mathbb{E}\left[ 1_{A} \vert \mathcal{F} \right].
\]
If $\eta$ is a random element in some measurable space $\left(U,\mathcal{U}\right)$, this definition reads as
\[
\mathbb{P}\left[A \vert \eta\right] := \mathbb{E}\left[ 1_{A} \vert \eta \right],\: A\in \mathcal{A}.
\]
\end{definition}
The conditional probability, by \cite[Theorem 5.1]{kall}, is the a.s. unique random variable in $L_{1}\left(\Omega,\mathcal{F},\mathbb{P}\right)$ (or in the second case in $L_{1}\left(\Omega,\sigma(\eta),\mathbb{P}\right)$), satisfying
\[
\mathbb{E}\left[ \mathbb{P}\left[A \vert \mathcal{F}\right] ; B \right] \overset{\eqref{condexpec}}{=} \mathbb{E}\left[ 1_{A} ; B \right] = \mathbb{P}\left( A \cap B \right),\: \forall B\in \mathcal{F},
\]
or 
\[
\mathbb{E}\left[ \mathbb{P}\left[A \vert \eta\right] ; B \right] = \mathbb{P}\left( A \cap B \right),\: \forall B\in \sigma(\eta).\label{unique}
\]
\comment{
So for two random elements $\xi$ in $(S,\mathcal{S})$ and $\eta$ in $(U,\mathcal{U})$, the probability 
\[
\mathbb{P}\left( \lbrace \xi \in A \rbrace \cap \lbrace \eta \in B \rbrace \right) = \mathbb{P}_{\xi,\eta}(A \cap B)
\]
of an event, where $A \times B \in \mathcal{S}\otimes \mathcal{U}$, can not only be computed with the joint distribution $\mathbb{P}_{\xi,\eta}$, but also as the expectation value of the conditional probability $\mathbb{P}\left[ \lbrace \xi \in A \rbrace \vert \eta \right]$, restricted to $\lbrace \eta \in B \rbrace$ by multiplication with $1_{\lbrace \eta \in B \rbrace}$.
Because by \cite[Lemma 1.13]{kall}, $\mathbb{P}\left[A \vert \eta\right]$ takes the form $f(\eta)$ for a measurable function $f$ on $U$ \cite[cf.][p.83]{kall} and with this function, the same probability can also be computed by the integral
\[
\int\limits_{B}f(u)\mathbb{P}_{\eta}(du).
\]
Note also that we now talk about arbitrary random elements again, and not just about random variables. The random variable that we need for our definition of the conditional expectation appears as an indicator function $1_{\lbrace \xi \in A \rbrace}$ for the measurable set $\lbrace \xi \in A \rbrace$, which lies in the $\sigma$-field generated by $\xi$.\\
We might now be interested in the probability of the random element $\xi$ to hit the set $A\in \mathcal{S}$, i.e. of the event $\lbrace \xi \in A \rbrace$, which is described by the expectation value of the random variable $1_{\lbrace \xi \in A \rbrace}$. If in addition, there is another random element $\eta$, then this probability, with present knowledge about $\eta$, is described by the random variable $\mathbb{P}\left[ \lbrace \xi \in A \rbrace \vert \eta \right]$ in the following sense. From observations, we might know, or at least assume with high probability, that $\eta$ only hits values in a set $B\in \mathcal{U}$. The probability of $\xi$ to hit $A$, given this knowledge, is thus described as the expectation value of the random variable $\mathbb{P}\left[ \left\lbrace \xi \in A \right\rbrace \vert \eta \right] \cdot 1_{\left\lbrace \eta \in B \right\rbrace}$.
}\\

For two random elements $\xi$ in $(S,\mathcal{S})$ and $\eta$ in $(U,\mathcal{U})$, in case there is a probability kernel $\mu$ from $U$ to $S$ satisfying
\[
\mu(\eta,A) = \mathbb{P}\left[ \xi \in A \vert \eta \right]
\]
a.s. in $\left(\Omega, \sigma(\eta), \mathbb{P}\right)$ for every $A \in \mathcal{S}$,
one calls $\mu$ a regular conditional distribution.\\
\comment{If for example $\xi$ is $\eta$-measurable, i.e. $\sigma(\xi) \subset \sigma(\eta)$, then $(u,A)\mapsto 1_{\lbrace\xi \in A \rbrace}$, $u\in U$, $A\in \mathcal{S}$, is such a regular version.\\
In case $\xi$ is independent of $\eta$, then the function $(u,A)\mapsto \mathbb{P}(\lbrace \xi \in A \rbrace)$, $u\in U$, $A\in \mathcal{S}$, is one \cite[cf.][p.84]{kall}.\\
In the latter case, one has 
\[
\mathbb{P}(\left\lbrace \xi \in A \right\rbrace \cap \left\lbrace \eta \in B \right\rbrace)
= \mathbb{P}_{\xi,\eta}(A \cap B)
= \mathbb{P}_{\xi}(A) \cdot  \mathbb{P}_{\eta}(B)
= \mathbb{P}(\left\lbrace \xi \in A \right\rbrace) \cdot  \mathbb{P}( \left\lbrace \eta \in B \right\rbrace),
\]
for $A \times B \in \mathcal{S}\otimes \mathcal{U}$.
}\\

In this work, we will only be confronted with the case, when $S = \mathbb{R}^{n_{S}}$ and $U = \mathbb{R}^{n_{u}}$, for maybe distinct $n_{S},n_{U} \in \mathbb{N}$.
Those spaces, by \cite[p.7]{kall}, are Borel spaces, because they are separable and complete with respect to a norm.\\
Therefore, for two random elements $\xi$ in $S = \mathbb{R}^{n_{S}}$ and $\eta$ in $U = \mathbb{R}^{n_{U}}$, \cite[Theorem 5.3]{kall} states, that there is indeed a probability kernel $\mu$ from $U$ to $S$ satisfying 
\begin{align}
\mathbb{P}\left[\xi \in .\vert \eta \right] = \mu(\eta,.)\label{unique_prob_kernel}
\end{align}
a.s. and that $\mu$ is $\mathbb{P}_{\eta}$-a.e. unique in the first argument.\\
For this probability kernel, the disintegration theorem \cite[Theorem 5.4]{kall} gives us the identities
\begin{align}
\mathbb{E}\left[ f(\xi,\eta) \vert \eta \right] &= \int\limits_{S} \mu(\eta,ds)f(s,\eta) \; \text{a.s. and}\\
\mathbb{E}\left[ f(\xi,\eta)\right] &= \mathbb{E}\left[ \int\limits_{S} \mu(\eta,ds) f(s,\eta)\right],
\end{align}
for every $f \in L_{1}(S\times U, \mathcal{S}\otimes \mathcal{U}, \mathbb{P}_{\xi,\eta})$.\\
These equations already look quite similar to what most people know as Bayes' Theorem. We want to use this identity for our setting and therefore shrink it to the cases that we need.\\

So coming back from the general case to the concrete problems in our case, and switching to the notation of this work, we assume:
\begin{assumption}\label{continuous_density}
The two random elements $\phi$ in $S = \mathbb{R}^{n_{S}}$ and $d$ in $U = \mathbb{R}^{n_{U}}$, with $n_{S},n_{U}\in \mathbb{N}$, have a joint distribution $\mathbb{P}_{\phi,d}$ which is absolutely continuous with respect to the Lebesgue measure $\lambda_{\mathbb{R}^{n_{S}+n_{U}}}$ on $S\times U$.
\end{assumption} 

We denote by $\lambda_{\mathbb{R}^{n_{S}}}$ the Lebesgue measure on $S$, and by $\lambda_{\mathbb{R}^{n_{U}}}$ the one on $U$ and will often write $dx$ for $\lambda(dx)$, if we mean integration with respect to a Lebesgue measure. By the Theorem of Radon-Nikodym \cite[5.2, Theorem 1]{henze}, and the remark after it, there exists a $\lambda_{\mathbb{R}^{n_{S}+n_{U}}}$-a.e. unique density $\mathcal{P}_{\phi,d}\in L_{1}(S\times U, \mathcal{S}\otimes \mathcal{U},\lambda_{\mathbb{R}^{n_{S}+n_{U}}})$ of $\mathbb{P}_{\phi,d}$, i.e.
\[
\mathbb{P}_{\phi,d}(B) = \int\limits_{B} \mathcal{P}_{\phi,d}(s,u) d(s,u)\ \forall B\in \mathcal{S}\otimes \mathcal{U}.
\]
As defined in \eqref{density}, we denote this relation by $\mathbb{P}_{\phi,d} = \mathcal{P}_{\phi,d}\cdot\lambda_{\mathbb{R}^{n_{S}+n_{U}}}$.\\

By Theorem \ref{fubini}, we then have
\begin{align}
\mathbb{P}_{d} &= \mathcal{P}_{d}\cdot\lambda_{\mathbb{R}^{n_{U}}} \text{ for the } \lambda_{\mathbb{R}^{n_{U}}}\text{-density }  \mathcal{P}_{d}(u) = \int\limits_{S} \mathcal{P}_{\phi,d}(s,u) ds \ \text{and}\\
\mathbb{P}_{\phi} &= \mathcal{P}_{\phi}\cdot\lambda_{\mathbb{R}^{n_{S}}} \text{ for the } \lambda_{\mathbb{R}^{n_{S}}}\text{-density } \mathcal{P}_{\phi}(s) = \int\limits_{U} \mathcal{P}_{\phi,d}(s,u) du.
\end{align}
By \cite[Lemma 1.26 and Lemma 1.12]{kall}, the function
\[
(s,u) \mapsto \frac{\mathcal{P}_{\phi,d}(s,u)}{\mathcal{P}_{d}(u)},
\]
as a map on $\left( S\times \lbrace u\in U: \mathcal{P}_{d}(u)>0 \rbrace, \mathcal{S}\otimes \sigma\left( \lbrace u\in U: \mathcal{P}_{d}(u)>0 \rbrace \cap \mathcal{U} \right), \lambda_{\mathbb{R}^{n_{S}+n_{U}}}\right)$, is measurable in $u$ for every $s\in S$ and by construction, $\left(\frac{\mathcal{P}_{\phi,d}(.,u)}{\mathcal{P}_{d}(u)} \right)\cdot\lambda_{\mathbb{R}^{n_{S}}}$ is a probability measure on $S$ for every $u\in \lbrace u\in U: \mathcal{P}_{d}(u)>0 \rbrace$.\\
\comment{
By $\lbrace u\in U: \mathcal{P}_{d}(u)>0 \rbrace \cap \mathcal{U}$, we mean all sets $\lbrace u\in U: \mathcal{P}_{d}(u)>0 \rbrace \cap B$, with $B\in \mathcal{U}$. Because $\lbrace u\in U: \mathcal{P}_{d}(u)>0 \rbrace$ is $\mathcal{U}$-measurable, it is the set of all $B\in \mathcal{U}$, that are also subsets of $\lbrace u\in U: \mathcal{P}_{d}(u)>0 \rbrace$.
}\\

We can therefore define:
\begin{definition}[Conditional probability density]\label{condprobdens}
If the joint distribution $\mathbb{P}_{\phi,d}$ of the random elements $\phi$ and $d$ from Assumption \ref{continuous_density} has the $\lambda_{\mathbb{R}^{n_{S}+n_{U}}}\text{-density } \mathcal{P}_{\phi,d}$, we define the conditional probability density $\mathcal{P}_{\phi \vert d}$ as
\begin{align}
\mathcal{P}_{\phi \vert d}(s,u) &:=  \frac{\mathcal{P}_{\phi,d}(s,u)}{\mathcal{P}_{d}(u)}, \text{ if } \mathcal{P}_{d}(u)>0 
\end{align}
and equal to any probability density function w.r.t. $\lambda_{\mathbb{R}^{n_{S}}}$, if $\mathcal{P}_{d}(u)=0$.
\end{definition}

\begin{theorem}
With Assumption \ref{continuous_density}, for every set $\lbrace \phi \in A \rbrace= \phi^{-1}(A) \in \sigma(\phi)$, with $A\in \mathcal{S}$, it is 
\[
\left(\mathcal{P}_{\phi \vert d}(.,d) \cdot \lambda_{\mathbb{R}^{n_{S}}}\right)(A) = \int\limits_{A} \mathcal{P}_{\phi \vert d}(s,d) ds = \mathbb{P}\left[ \phi \in A \vert d \right]
\]
$\mathbb{P}$-a.s in $(\Omega,\sigma(d),\mathbb{P})$.\\
This defines the probability kernel from $U$ to $S$ from \eqref{unique_prob_kernel} that is $\mathbb{P}_{d}$-a.e unique in the first argument.
\end{theorem}

\begin{proof}
First of all, note that 
\begin{align}
0&= \int\limits_{\lbrace u\in U: \mathcal{P}_{d}(u)=0 \rbrace} \mathcal{P}_{d}(u) du\\
&\overset{\mathrm{Thm.\, \ref{fubini}}}{=} \int\limits_{S\times \lbrace u\in U: \mathcal{P}_{d}(u)=0 \rbrace} \mathcal{P}_{\phi,d}(s,u) d(s,u).
\end{align}
Because $\mathcal{P}_{\phi,d}(s,u)\geq 0$ $\lambda_{\mathbb{R}^{n_{S}+n_{U}}}$-a.s.,
it is 
\[
\mathcal{P}_{\phi,d}(s,u)= 0
\]
$\lambda_{\mathbb{R}^{n_{S}+n_{U}}}$-a.s. in $S \times \lbrace u\in U:  \mathcal{P}_{\phi}(u)=0 \rbrace$.\\
That is, the set
\[
\left( S\times \lbrace u\in U: \mathcal{P}_{d}(u) = 0 \rbrace \right) \cap \lbrace (s,u)\in S\times U: \mathcal{P}_{\phi,d}(s,u)>0\rbrace\label{zero}
\]
is a $\lambda_{\mathbb{R}^{n_{S}+n_{U}}}$-null set.\\
By Definition \ref{condprobdens}, the map
\[
(u,A)\mapsto \left( \mathcal{P}_{\phi \vert d}(.,u)\cdot\lambda_{\mathbb{R}^{n_{S}}}\right)(A),
\]
for $u\in U$ and $A\in \mathcal{S}$, defines a probability kernel from $U$ to $S$ and for every $f \in L_{1}(S\times U, \mathcal{S}\otimes \mathcal{U}, \mathbb{P}_{\phi,d})$:
\begin{align}
\mathbb{E}\left[ f(\phi,d) \right] &= \int\limits_{S\times U} f(s,u) \cdot \mathcal{P}_{\phi,d}(s,u) d(s,u)\\
&\overset{\eqref{zero}}{=} \int\limits_{S\times \lbrace u\in U: \mathcal{P}_{d}(u) > 0 \rbrace} f(s,u) \cdot \frac{\mathcal{P}_{\phi,d}(s,u)}{\mathcal{P}_{d}(u)}\cdot\mathcal{P}_{d}(u) d(s,u)\\
&\overset{\mathrm{Thm.}\, \ref{fubini},\: \mathrm{Def.}\, \ref{condprobdens}}{=} \int\limits_{U} \left( \int\limits_{S} \mathcal{P}_{\phi \vert d}(s,u)\cdot f(s,u)ds \right) \cdot\mathcal{P}_{d}(u) du \label{bayes}\\
&= \int\limits_{U} \left( \int\limits_{S} \mathcal{P}_{\phi \vert d}(s,u)\cdot f(s,u) ds \right)\mathbb{P}_{d}(du)\\
&= \mathbb{E}\left[ \int\limits_{S} \mathcal{P}_{\phi \vert d}(s,d)\cdot f(s,d) ds\right].
\end{align}
For $f=1_{A\times B}$, with measurable sets $A\in \mathcal{S}$ and $B\in \mathcal{U}$, this gets
\begin{align}
\mathbb{E}\left[\ \mathbb{P}\left[ \phi \in A \vert d \right] ; d \in B \right] &\overset{\eqref{unique}}{=} \mathbb{P}\left( \lbrace \phi \in A \rbrace \cap \lbrace d \in B \rbrace \right)\\
&=\mathbb{P}_{\phi,d}\left( A \times B \right)\\
&= \mathbb{E}\left[  \int\limits_{A} \mathcal{P}_{\phi \vert d}(s,d) ds ; d \in B \right].
\end{align}
Therefore, for every set $\lbrace \phi \in A \rbrace= \phi^{-1}(A) \in \sigma(\phi)$, with $A\in \mathcal{S}$, by \eqref{unique}, it is 
\[
\left(\mathcal{P}_{\phi \vert d}(.,d) \cdot \lambda_{\mathbb{R}^{n_{S}}}\right)(A) = \int\limits_{A} \mathcal{P}_{\phi \vert d}(s,d) ds = \mathbb{P}\left[ \phi \in A \vert d \right]
\]
$\mathbb{P}$-a.s in $L_{1}(\Omega,\sigma(d),\mathbb{P})$.\\
Thus, we found the probability kernel from $U$ to $S$ from \eqref{unique_prob_kernel} that is $\mathbb{P}_{d}$-a.e unique in the first argument, i.e. 
\begin{align}
\mu(d,A)= \left(\mathcal{P}_{\phi \vert d}(.,d) \cdot \lambda_{\mathbb{R}^{n_{S}}}\right)(A)
\end{align}
$\mathbb{P}$-a.e. in $\left(\Omega, \sigma(d), \mathbb{P}\right)$ for every $A\in \mathcal{S}$.
\end{proof}

Note also, that $\mathcal{P}_{\phi \vert d}(s,u)$ is a.e. unique in $L_{1}(S\times U, \mathcal{S}\otimes\mathcal{U}, \lambda_{\mathbb{R}^{n_{S}}}\otimes \mathbb{P}_{d})$.\\
The same way, one can define the function $\mathcal{P}_{d \vert \phi}$ and construct a probability kernel from $S$ to $U$ by
\[
(s,B)\mapsto \left( \mathcal{P}_{d \vert \phi}(s,.)\cdot\lambda_{\mathbb{R}^{n_{U}}}\right)(B),
\]
for $s\in S$ and $B\in \mathcal{U}$.\\

If Assumption \ref{continuous_density} is satisfied, we therefore define:
\begin{definition}[Prior, posterior, evidence, likelihood]\label{Prior, posterior, evidence, likelihood}
With the notation from above, we call $\mathcal{P}_{\phi}$ the prior, $\mathcal{P}_{\phi \vert d}$ the posterior, $\mathcal{P}_{d}$ the evidence and $\mathcal{P}_{d \vert \phi}$ the likelihood in case $\phi$ is the unknown quantity (signal) and $d$ the known one (data).
\end{definition}
\comment{
Physicists often write $\mathcal{P}(\phi \vert d)$ instead of $\mathcal{P}_{\phi \vert d}(s,u)$ by identifying the index with the argument variables. In case they need to distinguish them, they write $\mathcal{P}(\phi=s \vert d=u)$.
}\\

By the above construction, we can also state the following theorem:
\begin{theorem}[Bayes' Theorem]\label{bayesthm}
If $\phi$ is a random vector in $S=\mathbb{R}^{n_{S}}$ and $d$ is a random vector in $U=\mathbb{R}^{n_{U}}$, for $n_{S},n_{U}\in \mathbb{N}$, such that $\mathbb{P}_{\phi,d}$ is absolutely continuous w.r.t. the Lebesgue measure, then:
\[
\mathcal{P}_{\phi \vert d}(s,u) = \frac{\mathcal{P}_{d \vert \phi}(u,s)\cdot \mathcal{P}_{\phi}(s)}{\mathcal{P}_{d}(u)} \text{ for } \lambda_{\mathbb{R}^{n_{S}}}\otimes \mathbb{P}_{d}\text{-a.e. }(s,u)\in S\times  U.
\]
\end{theorem}
\begin{proof}
Note first, that because $\mathbb{P}_{d}=\mathcal{P}_{d}\cdot \lambda_{\mathbb{R}^{n_{U}}}$, the set
\[
\lbrace u\in U: \mathcal{P}_{d}(u)=0 \rbrace
\]
is a $\mathbb{P}_{d}$-null set and that therefore
\[
S\times \lbrace u\in U: \mathcal{P}_{d}(u)=0 \rbrace
\]
is a $\lambda_{\mathbb{R}^{n_{S}}}\otimes\mathbb{P}_{d}$-null set. This way, the right side in the statement of the theorem is well-defined as a function in $L_{1}\left(S\times U, \mathcal{S}\otimes \mathcal{U},\lambda_{\mathbb{R}^{n_{S}}}\otimes\mathbb{P}_{d} \right)$.\\
For ${\mathcal{P}_{\phi}(s)}>0$, by definition, one has for $\mathbb{P}_{d}$-a.e. $u\in U$, that
\[
\mathcal{P}_{\phi \vert d}(s,u) = \frac{\mathcal{P}_{\phi,d}(s,u)\cdot \mathcal{P}_{\phi}(s)}{\mathcal{P}_{\phi}(s)  \mathcal{P}_{d}(u)} = \frac{\mathcal{P}_{d \vert \phi}(u,s)\cdot \mathcal{P}_{\phi}(s)}{\mathcal{P}_{d}(u)}.
\]
As in \eqref{bayes}, for $f \in L_{1}(S\times U, \mathcal{S}\otimes \mathcal{U}, \mathbb{P}_{\phi,d})$, it follows
\begin{align}
\int\limits_{S\times U} \mathcal{P}_{\phi \vert d}(s,u)\cdot\mathcal{P}_{d}(u)\cdot f(s,u) d(s,u)
&= \int\limits_{S\times U} \mathcal{P}_{\phi,d}(s,u) \cdot f(s,u) d(s,u) = \mathbb{E}\left[ f(\phi,d) \right]\\
&= \int\limits_{S\times U} \mathcal{P}_{d \vert \phi}(u,s)\cdot\mathcal{P}_{\phi}(s)\cdot f(s,u) d(s,u).
\end{align}
With $f = 1_{\lbrace s\in S:  \mathcal{P}_{\phi}(s)=0 \rbrace\times U}$, this leads to 
\[
\int\limits_{S\times U} \mathcal{P}_{\phi \vert d}(s,u)\cdot\mathcal{P}_{d}(u)\cdot f(s,u) d(s,u) = \int\limits_{S\times U} \mathcal{P}_{d \vert \phi}(u,s)\cdot\mathcal{P}_{\phi}(s) \cdot f(s,u) d(s,u) = 0.
\]
Because $\mathcal{P}_{\phi \vert d}(s,u)\geq 0$ $\lambda_{\mathbb{R}^{n_{S}}} \otimes \mathbb{P}_{d}$-a.s., we therefore have 
\[
\mathcal{P}_{\phi \vert d}(s,u) = 0
\]
$\lambda_{\mathbb{R}^{n_{S}}} \otimes \mathbb{P}_{d}$-a.s. in $\lbrace s\in S:  \mathcal{P}_{\phi}(s)=0 \rbrace\times U$.\\
That is,
\[
\left( \lbrace s\in S : \mathcal{P}_{\phi}(s)=0 \rbrace\times U\right)\cap \lbrace (s,u)\in S\times U:\mathcal{P}_{\phi \vert d}(s,u) > 0\rbrace
\]
is a $\lambda_{\mathbb{R}^{n_{S}}}\otimes \mathbb{P}_{d}$-null set, which finishes the proof.
\end{proof}
\comment{This way, if we are in the setting of Theorem \ref{bayesthm}, we are able to compute the posterior in two different ways. One possibility is the way we constructed it from the joint distribution. The other possibility arises by Bayes' Theorem, if we know the prior, the likelihood and the evidence. }\\

In the setting of Theorem \ref{bayesthm}, one can also define:

\begin{definition}[Information Hamiltonian]\label{hamiltonian}
If Assumption \ref{continuous_density} is satisfied, the information Hamiltonian \cite[cf.][p.2]{enss} on $U\times S$ is for $\lambda_{\mathbb{R}^{n_{U}+n_{S}}}$-a.e. $(u,s)\in U\times S$ given by
\begin{align}
H_{d,\phi}(u,s) := \begin{cases}
-\log \left( \mathcal{P}_{\phi,d}(s,u)\right) = -\log \left( \mathcal{P}_{d \vert \phi}(u,s) \right) - \log \left( \mathcal{P}_{\phi}(s) \right), & \text{if } \mathcal{P}_{d \vert \phi}(u,s)>0,\\
\infty, & \text{else.}
\end{cases}
\end{align}

\end{definition}
The information Hamiltonian can thus be infinite on some set.\\
We then compute the posterior by
\[
\mathcal{P}_{\phi \vert d}(s,u) = \frac{e^{-H_{d,\phi}(u,s)}}{\mathcal{P}_{d}(u)}\ \lambda_{\mathbb{R}^{n_{S}}}\otimes \mathbb{P}_{d}\text{-a.e.},\label{hamiltonianrelation}
\]
where we set $e^{- \infty}:=0$.\\
Note that on the subset where $\mathcal{P}_{d}(u)=0$ the posterior by definition equals any probability density function w.r.t. $\lambda_{\mathbb{R}^{n_{S}}}$. But if we condition on $d$, these values are actually canceled out in the integration, as one multiplies with $\mathcal{P}_{d}(u)$ then.\\
That is, the information Hamiltonian contains all information about $\phi$.\\
\comment{
Although the change from probability densities to Hamiltonians is only a coordinate transformation, it allows to connect to the rich physical literature on statistical mechanics and thermodynamics in which a large tool set of computational methods are developed.
}

\chapter{The setting}
\label{chap:setting}
In physics, a process at some time is usually described by a field $\phi_{\mathrm{real},0}$, i.e. a vector valued function $\phi_{\mathrm{real},0}: D \rightarrow \mathbb{R}^{m}$ with $m \in \mathbb{N}$, $D\subset \mathbb{R}^{k}$ and $k \in \mathbb{N}$, which belongs to a certain Banach space of functions $V$, endowed with a norm $\Vert . \Vert_{V}$.
In some cases, the signal might belong to a Hilbert space, and the norm then comes from a scalar product $\langle .,. \rangle_{V}$.\\
This field, over a time interval $[0,T]$ with $T>0$, mostly develops according to some evolution equation, or at least does so in a simplified model. Such an equation can come in the form of a partial differential equation or an integro-differential equation, with initial condition $\phi_{\mathrm{real}}(0) = \phi_{\mathrm{real},0}$ and can often be brought into the form
\begin{align}
\partial_{t}\phi(.,t) = F[t,\phi(.,t)],\ t\in [0,T],\label{pde}
\end{align}
where 
\begin{align}
\phi(.,t)\in V
\end{align}
for every $t\in [0,T]$ and with a function 
\begin{align}
F&:[0,T]\times V\rightarrow V.
\end{align}
One then has to specify the subspace of functions in $\mathrm{map}(D \times [0,T],\mathbb{R}^{m})$, in which one searches for a solution of the equation and quite often brings a partial differential equation into its weak form, resulting in a weak solution first. If the latter is regular enough, it is also a classical solution of the original problem. We assume our equation to have a unique solution for every initial condition $\phi_{\mathrm{real},0}\in V$, so that the problem is well-posed.

\section{The signal}
In many cases, an evolution equation is not analytically solvable. But very often, one can approximate it by functions which are piecewise constant in time and which are contained in a finite dimensional subspace of $V$ for every time. Physical processes are often relatively smooth in time, because the underlying physical system does not change its properties instantaneously. A possible approximation by a function from a finite dimensional space is often a result of the modeling part in the equation, because only a limited amount of properties and information about the real field can be included here, leading to an approximation already in the first place, when the equation is developed.\\
There are many different of such simplification techniques, also depending on the equation. An idea of when such an approximation could be justified and how it could be constructed will be explained in the following and referred to in the further chapters. Although, this will be done in quite a general way and without going into detail, because we do not assume a concrete underlying equation and only want to deliver the concept.\\

If our field follows an equation like in \eqref{pde}, one might regard the field as a function, which depends only on time and has values in $V$, i.e. $\phi_{\mathrm{real}}:[0,T] \rightarrow V$.
\eqref{pde} then takes the form
\[
\dfrac{d}{dt}\phi(t) = F[t,\phi(t)],\ t\in [0,T].
\]
One should then specify the subspace $W$ of functions in $\mathrm{map}([0,T],V)$, endowed with a norm $\Vert . \Vert_{W}$, in which one searches for solutions of this problem.\\
Also differentiability of functions in $W$ has to be specified.
We assume the equation to have a unique (maybe weak) solution $\phi_{\mathrm{real}}:[0,T] \rightarrow V$, $\phi_{\mathrm{real}} \in W$, for every initial condition $\phi_{\mathrm{real},0}\in V$.\\
Considering the smoothness in time, we further assume the space $W$ to be of the type, that for any $\phi_{\mathrm{real}}\in W$, if $W$ is a space of equivalence classes of functions, we have at least one representative which is piecewise continuous, so that pointwise evaluation $\phi_{\mathrm{real}}(t)\in V$, for $t\in [0,T]$, makes sense at all, and always consider this representative.\\
We want to approximate the solution of our equation, taking a sequence of discretizations in time $\left(\left(t_{i}=i\dfrac{T}{2^N}\right)_{i=0}^{2^N}\right)_{N=1}^{\infty}$ and finite dimensional subspaces $\left(V_{n}\right)_{n=1}^{\infty}$ of $V$, and constructing solutions of simplified problems which are step functions in time, i.e. constant in $V_{n}$ for every time interval $(t_{i},t_{i+1}]$, and which converge to the actual solution in $W$ with $N,n\rightarrow \infty$.\\
This will be done step by step to illustrate the concept and the procedure of how such an approximation could be derived.\\
Firstly, we assume:
\begin{assumption}\label{regularity_W}
$W$ can be continuously imbedded in the space $C([0,T],V)$ of uniformly continuous functions from $[0,T]$ to $V$, endowed with the norm 
\begin{align}
\sup\limits_{t\in[0,T]}\Vert . \Vert_{V}.\label{sup}
\end{align}
This way, convergence in $W$ is reached by convergence in $C([0,T],V)$ for $N,n\rightarrow \infty$.
\end{assumption}
For every $N\in \mathbb{N}$, we set $\delta t := \dfrac{T}{2^N}$ and $t_{i}:=i\cdot \delta t$ for $i=0,...,2^N$.\\
We assume then that:

\begin{assumption}\label{ass:discretetime}
For given $N\in \mathbb{N}$, we find an operator valued function
\[
A_{N}: \lbrace t_{0},..., t_{2^N-1}\rbrace \rightarrow \mathrm{map}(V,V),
\]
that has the following property:\\
If we start with $\phi_{\mathrm{real}}(t_{i})$, $i \in \lbrace 0,...2^N-1 \rbrace$ and $f\in V$ and let $\phi_{\mathrm{real}}(t_{i})$ evolve according to the exact equation until $t_{i+1}$, then:
\begin{align}\sup\limits_{t\in (t_{i},t_{i+1}]} \Vert \phi_{\mathrm{real}}(t) - A_N(t_{i}) f \Vert_{V} \leq C_{1}\cdot \dfrac{\epsilon_{1,N}}{2^N} + \Vert \phi_{\mathrm{real}}(t_{i}) - f \Vert_{V}, \label{discretetime}
\end{align}
with $\epsilon_{1,N} \geq 0$, $\lim\limits_{N\rightarrow \infty}\epsilon_{1,N} = 0$ and some $C_{1} > 0$.
\end{assumption}

We then define:
\begin{definition}\label{brev_phi}
With Assumption \ref{regularity_W} and Assumption \ref{ass:discretetime}, we define $\phi_{N}: [0,T] \rightarrow V$ by
\begin{align}
&\phi_{N}(t) :=A_{N}(t_{i}) \phi_{N}(t_i) \text{ for }t\in(t_{i},t_{i+1}],\  i=0,...,2^N-1 \text{ and } \\
& \phi_{N,0} = \phi_{N}(0) := \phi_{\mathrm{real},0}.
\end{align}
\end{definition}
This way we get a step function in $W$, which is $C_{1}\cdot \epsilon_{1,N}$ close in $V$ to the solution of our differential equation for every time $t\in[0,T]$. By Assumption \ref{regularity_W}, it converges to $\phi_{\mathrm{real}}$ in $W$ with $N \rightarrow \infty$.\\

We further assume, that an appropriate functional basis to represent good approximations of the solution is known:
\begin{assumption}\label{ass:finitedim}
There is a nested family of finite dimensional subspaces $V_{n}=\mathrm{span}(f_{1},...,f_{n})$ of $V$ and a sequence of numbers $\epsilon_{2,n}\geq 0$ with
\[
\lim\limits_{n\rightarrow \infty}\epsilon_{2,n} = 0,
\]
such that for every admissible initial condition $f\in V$ of evolution equation \eqref{pde} there exists a function $f^{(n)}\in V_{n}$ with
\[\Vert f - f^{(n)} \Vert_{V} \leq \epsilon_{2,n}.\]

We further assume the basis elements of each subspace $V_{n}$ to be linear independent, so that every $f^{(n)}\in V_n$ has a unique representation 
\begin{align}
f^{(n)}= \sum\limits_{j=1}^n c_j\cdot f_j,
\end{align}
with $c_1,...,c_n\in \mathbb{R}$.
\end{assumption}

\comment{
The basis elements of $V_{n}$ could for example be interpolation functions or, in case $V$ is a separable Hilbert space like $L_{2}(\mathbb{R}^{m})$, they could be elements of some orthonormal basis. If $V = L_{2}([0,L])$ for some $L>0$, this could for example be the Fourier basis and the linear combination would then be an aborted Fourier series.
}\\

We will see that Assumption \ref{ass:discretetime} and Assumption \ref{ass:finitedim}, together with one more assumption, assure the existence of a step function in $\mathrm{map}([0,T],V_{n})$, with values in an $n$-dimensional space, which is $C_{1}\cdot \epsilon_{1,N} + 2\cdot\epsilon_{2,n}$ close in $V$ to the solution of our differential equation for every time $t\in [0,T]$ and for which we get a concrete construction scheme.
Because this expression converges to zero as $N,n \rightarrow \infty$, by Assumption \ref{regularity_W}, the step function also converges to $\phi_{\mathrm{real}}$ in $W$, so that one can choose a degree of accuracy that one wants to reach and pick $N$ and $n$ large enough to assure an approximation of the solution that is piecewise constant in time and spacial finite dimensional, and that reaches this accuracy level in $V$ for every time $t\in [0,T]$. In particular, it stays within this range until time $T$.\\
One can also identify the simplified field with a time discrete function in $\mathrm{map}(\lbrace t_{0},...,t_{2^N} \rbrace,V_{n})$, because it is constant in every interval $(t_{i},t_{i+1}]$, $i=0,...,2^N-1$.\\
We are actually interested in the signal at time $T$.
Because we assume that we are not able to compute the exact solution $\phi_{\mathrm{real}}$, we want to approximate it with our discretized signal. If we find a possibility to construct the latter time step for time step, starting with an approximation of our initial condition, we have a possibility to compute an approximation for time $T$ which converges to $\phi_{\mathrm{real}}(T)$ in $V$ as $N,n\rightarrow \infty$. The reason why we are interested in an approximation of $A_{N}$, is because this allows us to inductively construct our simplified signal.\\
Between each updating step, one usually has to save the discrete signal for the current time $t_{i}$, $i \in \left\lbrace 0,...,2^N \right\rbrace$, in a computer. Unfortunately, one only has access to a finite amount of memory here. With our construction scheme, we will therefore in general anyway have to choose some level of accuracy and live with an approximation which is very close to the exact solution in $W$. That is, we change our original differential equation by a time discrete and spatially finite dimensional evolution equation. The solutions of this equation, as we will see, are $C_{1}\cdot \epsilon_{1,N} + 2\cdot \epsilon_{2,n}$ close to $\phi_{\mathrm{real}}$ in $V$ at every time and thus, by Assumption \ref{regularity_W}, also arbitrarily close to it in $W$ for certain $N$ and $n$. Usually, large numbers $N$ and $n$ are required for a high accuracy. 
\\
We now characterize a concrete case, when such an approximation can be constructed.\\
In addition to the Assumptions \ref{regularity_W}, \ref{ass:discretetime} and \ref{ass:finitedim}, we also suppose to have:
\begin{assumption}\label{ass:finiteevolution}
For every $n,N \in \mathbb{N}$, we find an operator valued function
\[ A_{n,N}: \lbrace t_{0},..., t_{2^N-1}\rbrace \rightarrow \mathrm{map}(V_{n},V_{n}),
\]
such that for every $f\in V$ and $g\in V_n$ it is
\begin{align}
\Vert A_N(t_i) f - A_{n,N}(t_{i})g\Vert_{V} \leq \Vert f-g \Vert_V + \frac{\epsilon_{2,n}}{2^N}.\label{finiteevolution}
\end{align}
\end{assumption}
In this case, we actually have a construction rule for our simplified signal.\\
Because then we can define:
\begin{definition}\label{phi_n,N}
We suppose the Assumptions \ref{regularity_W}, \ref{ass:discretetime}, \ref{ass:finitedim} and \ref{ass:finiteevolution} to be valid and define $\phi^{(n)}_{N}: \lbrace t_{0},...,t_{2^N} \rbrace \rightarrow V_{n}$ as
\begin{align}
&\phi^{(n)}_{N}(t_{i+1}) :=A_{n,N}(t_{i}) \phi^{(n)}_{N}(t_i) \text{ for } i=0,...,2^N-1 \text{ and } \\
&\phi_{N,0}^{(n)} = \phi^{(n)}_{N}(0) := f^{(n)}\\
&\text{for some function } f^{(n)} \in V_{n} \text{ with } \Vert f^{(n)} - \phi_{\mathrm{real},0} \Vert_{V} \leq \epsilon_{2,n}.
\end{align} 
\end{definition}

As already mentioned, due to finite amount of computer memory, we change the original evolution equation and its solutions by this simplified equation and the just constructed signal.
\begin{theorem}\label{signalapproximation}
If the Assumptions \ref{regularity_W}, \ref{ass:discretetime}, \ref{ass:finitedim} and \ref{ass:finiteevolution} hold, then for given $n,N\in \mathbb{N}$ we can identify $\phi_{N}^{(n)}$ with the step function
\begin{align}
\phi_{N}^{(n)}(0)&=\phi_{N,0}^{(n)},\\
\phi_{N}^{(n)}(t) &= \phi_{N}^{(n)}(t_{i+1})\text{ for } t\in(t_{i},t_{i+1}],\ i=0,...,2^N-1
\end{align}
and get
\begin{align}
\Vert \phi_{\mathrm{real},0} - \phi^{(n)}_{N,0} \Vert_{V}\leq \epsilon_{2,n}
\end{align}
and
\begin{align}
\sup\limits_{t\in (t_{i},t_{i+1}]} \Vert \phi_{\mathrm{real}}(t) - \phi^{(n)}_{N}(t) \Vert_{V}  \leq C_{1}\cdot \epsilon_{1,N} + 2\cdot\epsilon_{2,n}
\end{align}
for every $i\in \lbrace 0,...,2^N-1 \rbrace$.\\
Furthermore, $\phi_{N}^{(n)}$ converges to $\phi_{\mathrm{real}}$ in $W$.
\end{theorem}

\begin{proof}
The first inequality holds by Definition \ref{phi_n,N}.
With Definition \ref{brev_phi}, Definition \ref{phi_n,N} and Assumption \ref{ass:finiteevolution} we get inductively
\begin{align}
\Vert \phi_{N}(t_{i}) -\phi^{(n)}_{N}(t_{i}) \Vert_V \leq \left(\frac{i}{2^N}+1\right)\cdot \epsilon_{2,n} \leq 2\cdot\epsilon_{2,n}\label{phi_approx}
\end{align}
for all $i=0,...,2^N$.\\
For $i\in \lbrace 0,...,2^N-1 \rbrace$ it follows
\begin{align}
\sup\limits_{t\in (t_{i},t_{i+1}]} \Vert \phi_{\mathrm{real}}(t) - \phi^{(n)}_{N}(t) \Vert_{V}  &\overset{\mathrm{Def.}\, \ref{phi_n,N}}{=} \sup\limits_{t\in (t_{i},t_{i+1}]} \Vert \phi_{\mathrm{real}}(t) - \phi^{(n)}_{N}(t_{i+1}) \Vert_{V}\\
&\leq \sup\limits_{t\in (t_{i},t_{i+1}]} \Vert \phi_{\mathrm{real}}(t) - \phi_{N}(t_{i+1}) \Vert_{V} + \Vert \phi^{(n)}_{N}(t_{i+1}) - \phi^{(n)}_{N}(t_{i+1})\Vert_V\\
&\overset{\eqref{phi_approx}}{\leq} \sup\limits_{t\in (t_{i},t_{i+1}]} \Vert \phi_{\mathrm{real}}(t) - \phi_{N}(t_{i+1}) \Vert_{V} + 2\cdot\epsilon_{2,n} \\
	 &\overset{\mathrm{Def.}\, \ref{brev_phi}}{=} \sup\limits_{t\in (t_{i},t_{i+1}]} \Vert \phi_{\mathrm{real}}(t) - A_{N}(t_{i}) \phi_{N}(t_{i}) \Vert_{V} + 2\cdot\epsilon_{2,n} \\
	&\overset{\mathrm{Ass.}\, \ref{ass:discretetime}}{\leq}  \Vert \phi_{\mathrm{real}}(t_{i}) - \phi_{N}(t_{i}) \Vert_{V} + C_{1}\cdot \dfrac{\epsilon_{1,N}}{2^N} + 2\cdot\epsilon_{2,n} \\
	&\leq ... \overset{\mathrm{Ass.}\, \ref{ass:discretetime},\mathrm{Def.}\, \ref{brev_phi}}{\leq} i\cdot C_{1}\cdot \dfrac{\epsilon_{1,N}}{2^N} + 2\cdot\epsilon_{2,n}\\
	&\leq 2^N\cdot C_{1}\cdot \dfrac{\epsilon_{1,N}}{2^N} + 2\cdot\epsilon_{2,n} 
	= C_{1}\cdot \epsilon_{1,N} + 2\cdot\epsilon_{2,n}.
\end{align}
This converges to zero as $N,n\rightarrow \infty$, so that by Assumption \ref{regularity_W}, $\phi_{N}^{(n)}$ converges to $\phi_{\mathrm{real}}$ in $W$.\\
\end{proof}
We can thus identify $\phi_{\mathrm{real}}$ with the time discrete signal $\phi_{N}^{(n)}$, and further $\phi_{\mathrm{real}}(t_{i})$ with the vector of prefactors $\left(c_{1,t_{i}},...,c_{n,t_{i}}\right)^T$ in the linear combination $\phi^{(n)}_{N}(t_{i}) = \sum\limits_{j=1}^{n} c_{j,t_{i}}f_{j}$ for each $i\in {0,...,2^N}$.

\section{Connection to the data}
\label{sec:Connection to the data}
In this section, the connection between the measured data and our signal shall be inspected.\\
We fix a certain $n$ and $N$ in the following.\\
$\phi^{(n)}_{N}(t_{i})$ , for $i\in {0,...,2^N}$, only depends on $\left(c_{1,t_{i}},...,c_{n,t_{i}}\right)^T$.\\
For a clear arrangement, we write 
\[
\phi(t_{i})
\]
for $\phi^{(n)}_{N}(t_{i})$ and $\left(c_{1,t_{i}},...,c_{n,t_{i}}\right)^T$ and switch between the interpretation as an $n$-dimensional function in $V_n$ and a vector in $\mathbb{R}^{n}$.\\
We start again with an exact solution $\phi_{\mathrm{real}}$ of the given evolution equation and replace it by our approximation $\phi$ later.\\
We assume that we measured our signal $\phi_{\mathrm{real},0}$ at the initial time and obtained some data $d_{0} \in \mathbb{R}^{Y_{0}}$, where $0< Y_{0} < \infty $.\\
We would now like to simulate the data $d(T)$ at the end of our time interval, that we will most likely measure at this time.\\
By the relation of the signal to the data, in this section it will be explained on the one hand why it is necessary for us to work with probability distributions. On the other hand, the role of this relation within our probabilistic setting will be clarified.\\
We assume first, we are given some measurement device, maybe depending on time, which is able to measure the signal at every time $t_{i}$, for $i=0,...,2^N$, and gives us some data $d(t_{i})\in \mathbb{R}^{Y_{i}}$, where $0< Y_{i} < \infty$. We also allow a measurement error $n(t_{i}) \in \mathbb{R}^{Y_{i}}$.
\begin{assumption}\label{datacont}
If for given $i \in \lbrace 0,...,2^N \rbrace$, we could measure the real signal at the time $t_i$, we assume to get a relation
\[
d_{\mathrm{real}}(t_{i}) = R_{N}(t_{i})\phi_{\mathrm{real}}(t_{i}) + n(t_{i}), 
\]
with a known, linear operator valued function $R_{N}$ on $\lbrace t_{0},...,t_{2^N}\rbrace$, i.e.
\[
R_{N}(t_{i}) \in \mathrm{lin map}\left(V,\mathbb{R}^{Y_{i}}\right)
\]
for $i=0,...,2^N$, and some realization $n(t_{i})$ of an uncertainty, that is a random vector with distribution $\mathbb{P}_{n(t_{i})}$.
\end{assumption}
\comment{
$R_{N}(t_{i})$, for $i=0,...,2^N$, is often called a response operator. This mapping represents our measurement device. It is quite reasonable to assume that we know how a measurement is related to a signal by $R_{N}$, because in general we know this device quite well. The unknown part of the relation is also included as an uncertainty of our measurement, represented by the random vector $n(t_{i})$. In Assumption \ref{datacont}, $n(t_{i})$ actually means a realization of a random vector with distribution $\mathbb{P}_{n(t_{i})}$.
}\\

For every $m\in \mathbb{N}$ we allways consider the euclidean norm in $\mathbb{R}^{m}$ which is given by \[\Vert x \Vert_{2} = \left( \sum\limits_{i=1}^{m} \vert x_{i} \vert^{2} \right)^{\frac{1}{2}}.
\]
For the response operators, we make the following
\begin{assumption}\label{bound_response}
We assume that every $R_{N}(t_{i})$ is bounded as an operator in the space $\mathrm{map}\left(V,\mathbb{R}^{Y_{i}}\right)$, with a common upper bound $C_{2}>0$ which is independent of $N$. This way, for any $N \in \mathbb{N}$ and $i \in \lbrace 0,...,2^N \rbrace$:
\begin{align}
\Vert R_{N}(t_{i}) \Vert = \sup_{f\in V, \Vert f \Vert_{V} \leq 1} \Vert R_{N}(t_{i})f \Vert_{2} \leq C_{2}.
\end{align}
\end{assumption}
\comment{
This assumption should be physically reasonable, because usually a measurement device averages over some property of the signal, weighting each position by some bounded value. This way, it maximally increases the norm of the signal by the maximal absolute value of all those weights. The measurement device can often be described by an integral operator, and the norm in $V$ also often is some integral expression. 
The bound is usually also independent of the time that we measure at, because we take the same device for every measurement. Also if some external parameters, like temperature, change during $[0,T]$, having an effect on this measurement device, it is quite unlikely, that the whole mechanism changes its properties completely. This justifies to also assume that we can choose a $C_2$ that is independent of time.
}\\
\begin{lemma}\label{response_approx}
In case Assumption \ref{bound_response} and all the requirements of Theorem \ref{signalapproximation} are satisfied, then for every $i \in \lbrace 0,...,2^N \rbrace$ we have for our approximation $\phi(t_{i}) \in V_{n}$ and the exact solution $\phi_{\mathrm{real}}(t_{i})$ of our differential equation:
\[
\left\Vert R_{N}(t_{i})\phi_{\mathrm{real}}(t_{i})  - R_{N}(t_{i})\phi(t_{i}) \right\Vert_{2} \leq C_{2}\cdot (C_{1}\cdot \epsilon_{1,N} + 2\cdot\epsilon_{2,n}).
\]
\end{lemma}
\begin{proof}
\[
\left\Vert R_{N}(t_{i})\phi_{\mathrm{real}}(t_{i})  - R_{N}(t_{i})\phi(t_{i}) \right\Vert_{2} \overset{\mathrm{Ass.}\, \ref{bound_response}}{\leq} C_{2}\cdot \left\Vert \phi_{\mathrm{real}}(t_{i}) - \phi(t_{i}) \right\Vert_{V} \overset{\mathrm{Thm.}\, \ref{signalapproximation}}{\leq} C_{2}\cdot (C_{1}\cdot \epsilon_{1,N} + 2\cdot\epsilon_{2,n}).
\]
\end{proof}
Because $R_{N}(t_{i})$ is linear and because $\phi(t_{i}) = \sum\limits_{j=1}^{n} c_{j,t_{i}} f_{j}$ is a linear combination for every $i\in \lbrace 0,...,2^N \rbrace$, we can identify $\phi(t_{i})$ with $\left(c_{1,t_{i}},...,c_{n,t_{i}}\right)^T$ as we already did above and define matrices $R_{N,n}(t_{i})\in \mathbb{R}^{Y_{i}\times n}$ as
\[
R_{N,n}(t_{i})_{kj} := \left( R_{N}(t_{i})f_{j} \right)_{k} \text{ for } 1\leq k \leq Y_{i},\; 1\leq j \leq n.
\] 
This way, we have
\[
R_{N,n}(t_{i})\phi(t_{i}) = R_{N}(t_{i})\phi(t_{i}),
\]
where on the left side we considered $\phi(t_{i})\in \mathbb{R}^{n}$ and on the right side its representation $\phi(t_{i})\in V_{n}$.

If the requirements of Lemma \ref{response_approx} are fulfilled, we can therefore approximate the relation in Assumption \ref{datacont} by
\[
d(t_{i}) = R_{N,n}(t_{i})\phi(t_{i}) + n(t_{i}),\label{datadiscrete}
\]
with the same random vector $n(t_{i})$.\\
For accuracy reasons, we should include the bound $C_{2}$ in our choice of $N$ and $n$ if we simplify the relation in Assumption \ref{datacont} of our data to the signal by \eqref{datadiscrete}.\\

\comment{
One should note, that this is an approximation of the underlying relation between data and signal. We leave the data that we actually measured unchanged. But the changed connection \eqref{datadiscrete} might lead to a shift in the signal $\phi(t_{i})$, that we obtain in the retracing from our measurement at time $t_{i}$, $i=0,...,2^N$.
}\\

In reality, we might only be able to measure $\phi(t_{i})\in \mathbb{R}^{n}$ for every $i \in \lbrace 0,...,2^N\rbrace$, instead of the real signal, and know the response operators $R_{N,n}(t_{i}) \in \mathbb{R}^{Y_{i}\times n}$ for large $N$ and $n$. We will therefore work with those matrices in the following.\\

\comment{
Consider for example the case, when $V=L_{2}([0,L])$ for $L>0$ and when $V_{n}$ is spanned by the first $n$ Fourier modes. It is quite likely that a measurement device is only able to measure frequencies within a finite range and therefore only finitely many Fourier modes. In this case, the replacement of $R_{N}$ by $R_{N,n}$ is justified.
}\\

Relation \eqref{datadiscrete} however, even for fixed $n(t_{i})$, which usually is not known, is in general still not invertible. This is because the number of resolution elements $n$ will mostly be much larger than the dimension of the data space $Y_{i}$, as usually a large $n$ is necessary to get $\epsilon_{2,n}$ from the Assumptions \ref{ass:finitedim} and \ref{ass:finiteevolution} small. So $\phi(t_{i})$ has many more degrees of freedom than $d(t_{i})$. This inhibits relation \eqref{datadiscrete} to be invertible and we have to think about a way to overcome this obstacle.\\
We therefore assume:
\begin{assumption}
$\phi(t_{i})$ is a random vector in $\mathbb{R}^{n}$ with a distribution $\mathbb{P}_{\phi(t_{i})}$ for each $i\in \lbrace 0,...,2^N \rbrace$.
\end{assumption}
Our relation \eqref{datadiscrete} is thus a restriction of $\phi(t_{i})$ within the space of all random vectors in $\left( \mathbb{R}^{n}, \mathcal{B}(\mathbb{R}^{n}) \right)$. This way, for each time $t_{i}$, $i=0,...,2^N$, $\phi(t_{i})$ represents all signals in $V_{n}$ which might have led to the measurement $d(t_{i})$, arranged according to how likely it is that the signal we measured was $\phi(t_{i})$.
This classification takes the form of a probability distribution $\mathbb{P}_{\phi(t_{i})}$ and usually is the result of a mixture of physical laws and experience.\\
With relation \eqref{datadiscrete}, we will in a first step construct a simulation scheme for $d(t_{i+1})$, if we know $d(t_{i})$, the distributions $\mathbb{P}_{\phi(t_{i})}, \mathbb{P}_{n(t_{i})}$, $\mathbb{P}_{\phi(t_{i+1})}, \mathbb{P}_{n(t_{i+1})}$, as well as $R_{N,n}(t_{i})$, $R_{N,n}(t_{i+1})$ and $A_{N,n}(t_{i})$.\\
By the resulting updating rule, we then obtain an algorithm to simulate all the $d(t_{i})$ for $i\in \lbrace 1,...,2^N \rbrace$, given $d_{0}$, all the uncertainty distributions, $\mathbb{P}_{\phi(t_{i})}$ and the matrices $R_{N,n}(t_{i})$ for $i=0,...,2^N$, and $A_{N,n}(t_{i})$ for $i=0,...,2^N-1$.\\
In particular, we construct a simulation scheme for the data vector $d(t_{2^N})=d(T)$.\\
The knowledge of \eqref{datadiscrete} is necessary for the computation of the joint distribution $\mathbb{P}_{\phi(t_{i}), d(t_{i})}$ from the distributions $\mathbb{P}_{\phi(t_{i})}$ and $\mathbb{P}_{n(t_{i})}$. The latter is needed to calculate the posterior, which shall be evolved to the later time.\\
Because in the first step, only two time points in a fixed space-time discretization are considered, we introduce a short notation:
\begin{definition}\label{simplifying_notation}
The following notation will be used in the next sections:
\begin{tabbing}
\= $\phi := \phi(t_{i}) \in \mathbb{R}^{n}\ \ \ \ \ \ \ \ \ \ \ \ \ \ \ \ \ \ \ \ $ \= $\phi':=\phi(t_{i+1})\in \mathbb{R}^{n}$\\
\> $\mathbb{R}^{Y}:= \mathbb{R}^{Y_{i}}$ 		\> $\mathbb{R}^{Y'}:= \mathbb{R}^{Y_{i+1}}$\\
\> $t:=t_{i}$									\> $t':=t_{i+1}$\\
\> $d := d(t_{i}) \in \mathbb{R}^{Y}$		\> $d' := d(t_{i+1})\in \mathbb{R}^{Y'}$\\
\> $n := n(t_{i}) \in \mathbb{R}^{Y}$		\> $n' := n(t_{i+1})\in \mathbb{R}^{Y'}$\\
\> $R:= R_{N,n}(t_{i}) \in \mathbb{R}^{Y\times n}$  \> $R':= R_{N,n}(t_{i+1}) \in \mathbb{R}^{Y'\times n}$\\
\> $A:= A_{N,n}(t_{i}) \in \mathrm{map}(\mathbb{R}^{n},\mathbb{R}^{n})$.
\end{tabbing}
\end{definition}
Note that $\phi$ and $\phi'$ are determined by their representation as a linear combination of the basis elements $f_{1},...,f_{n}$ of $V_{n}$ from Assumption \ref{ass:finitedim}. We can thus identify the operator $A\in \mathrm{map}(V_{n},V_{n})$ with the operator in $\mathrm{map}(\mathbb{R}^n,\mathbb{R}^{n})$, that maps the one linear combination into the other, and call this operator $A$ again.

\section{Wiener filter}
\label{Wiener filter}
In this section, we aim to construct the posterior $\mathcal{P}_{\phi \vert d}$ for our initial time $t$ from $\mathbb{P}_{\phi}$ and $\mathbb{P}_{n}$. This will be done in terms of a signal processing scheme that is called the Wiener filter.\\
Our notation in the following will include $\vert A \vert := \vert \det A \vert$ for a matrix $A$ and the complex scalar product $x^{\dagger}y = \sum\limits_{i=1}^{m}\overline{x}_{i}y_{i}$ for vectors $x,y\in \mathbb{C}^{m}$ or the real scalar product $x^{T}y = \sum\limits_{i=1}^{m}x_{i}y_{i}$ for vectors $x,y\in \mathbb{R}^{m}$.\\
For the general case of the Wiener filter setting, we assume that the following three conditions are fulfilled.

\begin{assumption}\label{Wiener_filter_setting} The general setting satisfies the following:
\begin{itemize}
\item
Our prior knowledge about the signal justifies to assume the distribution of $\phi$ to be Gaussian with zero mean, i.e. that it is absolutely continuous w.r.t. $\lambda_{\mathbb{R}^{n}}$ and
\[
\mathcal{P}_{\phi}(s) = \mathcal{G}(s,\Phi) = \frac{1}{\vert 2\pi \Phi \vert^{1/2}} \exp\left( -\frac{1}{2}s^{T} \Phi^{-1} s \right), \label{distributionsignal}
\]
with a positive definite matrix $\Phi \in \mathbb{R}^{n\times n}$.

\item The data is related to the signal by
\[d = R\phi + n, \label{datasignalrelation}\]
with the matrix $R$ from Definition \ref{simplifying_notation}.

\item
The noise $n$ is a realization of a Gaussian distributed random vector $n$ with zero mean, i.e. of a vector which is absolutely continuous w.r.t. $\lambda_{\mathbb{R}^{Y}}$ and for which
\[
\mathcal{P}_{n}(u) = \mathcal{G}(u,N) = \frac{1}{\vert 2\pi N \vert^{1/2}} \exp\left( -\frac{1}{2}u^{T} N^{-1} u \right) \label{distributionnoise}
\]
for a positive definite matrix $N \in \mathbb{R}^{Y\times Y}$.
\item Furthermore, $n$ is independent of $\phi$, i.e.
\[
\mathbb{P}_{\phi,n} = \mathbb{P}_{\phi}\otimes \mathbb{P}_{n}. \label{independence}
\]
\end{itemize}
\comment{
The Gaussian distribution for $\phi$ will also be a reasonable assumption in Chapter \ref{chap:Updating the data}, albeit possibly with a shifted mean value. An explanation for the validity of this supposition will be given in Section \ref{sec:prior_knowledge}.
}
\end{assumption}
\begin{lemma}\label{likelihood}
In the setting of Assumption \ref{Wiener_filter_setting}, the likelihood is given by
\begin{align}
\mathcal{P}_{d \vert \phi}(u,s) = \mathcal{G}(u - Rs,N)
\end{align}
for $\lambda_{\mathbb{R}^{Y+n}}$-a.e. $(u,s)\in \mathbb{R}^{Y}\times \mathbb{R}^{n}$.
\end{lemma}
\begin{proof}
For all $f \in L_{1}(\mathbb{R}^{n}\times \mathbb{R}^{Y},\mathcal{B}(\mathbb{R}^{n}\times \mathbb{R}^{Y}), \mathbb{P}_{\phi,d})$:
\begin{align}
\mathbb{E} \left[ f(\phi, d) \right] &\overset{\eqref{datasignalrelation}}{=} \mathbb{E} \left[ f(\phi, R\phi + n) \right]\\
&= \int\limits_{\mathbb{R}^{n}\times \mathbb{R}^{Y}} f(s,Rs + u) \mathbb{P}_{\phi,n}(d(s,u))\\
&\overset{\mathrm{Thm.} \ref{fubini}, \eqref{distributionnoise},\eqref{independence}}{=} \int\limits_{\mathbb{R}^{n}} \int\limits_{\mathbb{R}^{Y}} f(s,Rs + u)\cdot \mathcal{G}(u,N) du \mathbb{P}_{\phi}(ds)\\
&= \int\limits_{\mathbb{R}^{n}} \int\limits_{\mathbb{R}^{Y}} f(s,\tilde{u})\cdot \mathcal{G}(\tilde{u}-Rs,N) d(\tilde{u}-Rs) \mathbb{P}_{\phi}(ds)\\
&\overset{\mathrm{renaming}}{=} \int\limits_{\mathbb{R}^{n}} \int\limits_{\mathbb{R}^{Y}} f(s,u)\cdot \mathcal{G}(u-Rs,N) du \mathbb{P}_{\phi}(ds),
\end{align}
where the second last equality follows with substitution \cite[Lemma 1.22]{kall}, because $(s,u)\mapsto (s,Rs + u)$ is a continuous and therefore $\lambda_{\mathbb{R}^{n+Y}}$-measurable function. In the last equality we used translation invariance of the Lebesgue measure. By Definition \ref{Prior, posterior, evidence, likelihood}, the likelihood is therefore given by
\begin{align}
\mathcal{P}_{d \vert \phi}(u,s) = \mathcal{G}(u - Rs,N) 
\end{align}
for $\lambda_{\mathbb{R}^{Y+n}}$-a.e. $(u,s)\in \mathbb{R}^{Y}\times \lbrace s\in \mathbb{R}^{n}: \mathcal{P}_{\phi}(s) >0 \rbrace \overset{\mathcal{P}_{\phi}(s)>0\, \forall s\in \mathbb{R}^{n}\, \text{by}\, \eqref{distributionsignal}}{=}  \mathbb{R}^{Y}\times \mathbb{R}^{n}$.
\end{proof}

\comment{
Physicists often write $\mathcal{P}(d \vert \phi)=\mathcal{G}(d-R\phi,N)$ for this.
}\\

Now we are able to compute the posterior.
\begin{theorem}\label{thm:posterior}
The posterior in the setting of Assumption \ref{Wiener_filter_setting} on the set \(\mathbb{R}^{n}\times \lbrace u\in \mathbb{R}^{Y}: \mathcal{P}_{d}(u)>0 \rbrace\) is given by
\begin{align}
\mathcal{P}_{\phi \vert d}(s,u) &= \mathcal{G}(s-m(u),D)\ \lambda_{\mathbb{R}^{n+Y}}\text{-a.s with}\\
\text{uncertainty variance }D &= \left( \Phi^{-1} + R^{T}N^{-1}R  \right)^{-1} \text{ and}\\
\text{mean }m(u) & =Dj(u)=DR^{T}N^{-1}u.
\end{align}
In short, it is $\mathcal{P}_{\phi \vert d}(s,u) = \mathcal{G}(s-m(u),D)$ $\lambda_{\mathbb{R}^{n}}\otimes \mathbb{P}_{d}$-a.e in $\mathbb{R}^{n}\times \mathbb{R}^{Y}$.
\end{theorem}
\begin{proof}
The information Hamiltonian from Definition \ref{hamiltonian} reads
\begin{align}
H_{d,\phi}(u,s) &=  -\log \mathcal{P}_{d \vert \phi}(u,s) - \log \mathcal{P}_{\phi}(s)\\
&= \frac{1}{2} (u-Rs)^{T}N^{-1}(u-Rs) + \frac{1}{2} s^{T}\Phi^{-1}s + c_{1}\\
&\overset{N^{-1T}=N^{-1} ,(Rs)^{T}= s^{T}R^{T} }{=} \frac{1}{2} \left[ s^{T} \underbrace{\left( \Phi^{-1} + R^{T}N^{-1}R  \right)}_{\substack{=:D^{-1}}} s - s^{T} \underbrace{R^{T}N^{-1}u}_{\substack{=:j(u)}} - (R^{T}N^{-1}u)^{T}s \right] + c_{2}(u)\\
&\overset{D=D^{T}}{=} \frac{1}{2} (s - \underbrace{Dj(u)}_{\substack{=:m(u)}})^{T} D^{-1} (s - Dj(u))  + c_{3}(u)\\
&= \frac{1}{2} (s - m(u))^{T} D^{-1} (s - m(u))  + c_{3}(u),\label{hamiltonian_wiener_filter}
\end{align}
with $c_{1}$, $c_{2}(u)$ and $c_{3}(u)$ independent of $s$.
One has to note here, that $D$ is well defined and invertible, since with $\Phi$ and $N$ also their inverses are positive definite. Thus, $R^{T}N^{-1}R$ is still positive semidefinite and therefore $\Phi^{-1} + R^{T}N^{-1}R$ is positive definite, implying its invertibility.
With the definitions of 
\[
D=\left( \Phi^{-1} + R^{T}N^{-1}R  \right)^{-1}
\]
and 
\[
m(d)=Dj(d)=DR^{T}N^{-1}d,
\]
one has for $f \in L_{1}(\mathbb{R}^{n}\times \mathbb{R}^{Y},\mathcal{B}(\mathbb{R}^{n}\times \mathbb{R}^{Y}), \mathbb{P}_{\phi,d})$:
\begin{align}
\mathbb{E} \left[ f(\phi, d) \right] &\overset{\eqref{bayes}}{=} \int\limits_{\mathbb{R}^{Y}} \int\limits_{\mathbb{R}^{n}} \mathcal{P}_{\phi \vert d}(s,u)\cdot f(s,u) ds \mathcal{P}_{d}(u) du\\
&\overset{\eqref{hamiltonianrelation}}{=} \int\limits_{\left\lbrace u\in \mathbb{R}^{Y}: \mathcal{P}_{d}(u)>0 \right\rbrace} \int\limits_{\mathbb{R}^{n}} \frac{e^{-H_{d,\phi}(u,s)}}{\mathcal{P}_{d}(u)}\cdot f(s,u) ds \mathcal{P}_{d}(u) du\\
&\overset{\eqref{hamiltonian_wiener_filter}}{=} \int\limits_{\left\lbrace u\in \mathbb{R}^{Y}: \mathcal{P}_{d}(u)>0 \right\rbrace} \int\limits_{\mathbb{R}^{n}} \frac{e^{-\frac{1}{2} (s -m(u))^{T} D^{-1} (s - m(u))  + c_{3}(u)}}{\mathcal{P}_{d}(u)} \cdot f(s,u) ds \mathcal{P}_{d}(u) du\\
&= \int\limits_{\mathbb{R}^{Y}} \int\limits_{\mathbb{R}^{n}} \frac{e^{-\frac{1}{2} (s - m(u))^{T} D^{-1} (s - m(u)) } }{\vert 2\pi D \vert^{1/2}} \cdot f(s,u) ds \mathcal{P}_{d}(u) du. \label{posteriorcalculation}
\end{align}
Note in last step, that we know the normalization constant must be $\vert 2\pi D \vert^{1/2}$, because $(u,A)\mapsto \left(\mathcal{P}_{\phi \vert d}(.,u)\cdot \lambda_{\mathbb{R}^{n}}\right)(A)$, for $A\in  \mathcal{B}\left(\mathbb{R}^{n}\right)$ and $u \in \mathbb{R}^{Y}$, defines a probability kernel from $\mathbb{R}^{Y}$ to $\mathbb{R}^{n}$.\\
That is, one has 
\[
\int\limits_{\mathbb{R}^{n}} \mathcal{P}_{\phi \vert d}(s,u) ds \overset{!}{=} 1,
\]
and for arbitrary $m(u)\in \mathbb{R}^{n}$ it is
\[
\int\limits_{\mathbb{R}^{n}} e^{-\frac{1}{2} (s - m(u))^{T} D^{-1} (s - m(u))} ds = \vert 2\pi D \vert^{1/2}.
\]
Therefore, we have shown that on the set $\mathbb{R}^{n}\times \lbrace u\in \mathbb{R}^{Y}: \mathcal{P}_{d}(u)>0 \rbrace$:
\begin{align}
\mathcal{P}_{\phi \vert d}(s,u) &= \mathcal{G}(s-m(u),D)\ \lambda_{\mathbb{R}^{n+Y}}\text{-a.s with}\\
\text{uncertainty variance }D &= \left( \Phi^{-1} + R^{T}N^{-1}R  \right)^{-1} \text{ and}\\
\text{mean }m(u) & =Dj(u)=DR^{T}N^{-1}u.
\end{align}
\end{proof}

In \cite[p.3]{enss}, in terms of signal reconstruction by this (generalized) Wiener filter \cite[cf.][]{wiener}, $D$ is also called the information propagator and $j(d)$ the information source, so that the a posteriori mean field $m$ is obtained by applying $D$ to $j(d)$. One also calls 
\[
W:= DR^{T}N^{-1}\label{def:Wienerfilter}
\]
the Wiener filter and get's the linear connection 
\[
m=Wd
\]
between the data and the posterior mean field.\\
One should also note, that there are two equivalent definitions of the Wiener filter. It is
\begin{align}
W &= (\Phi^{-1} + R^{T}N^{-1}R)^{-1} R^{T}N^{-1}\\
&= \Phi R^{T}(R\Phi R^{T} + N)^{-1}.\label{Wienerfilter_representations}
\end{align}
The first equality is called signal space representation and the second one data space representation, where the names refer to the spaces in which operator inversions happen \cite[p.3]{enss}. The second representation can also cope with negligible noise, $N \rightarrow 0$, in case $R\Phi R^{T}$ is (at least pseudo-) invertible.
For our update of the data to a later time, as mentioned in the end of Section \ref{sec:Connection to the data}, we now do not aim to evolve only the a posteriori mean field $m(d)$, but the complete posterior $\mathcal{P}_{\phi \vert d}$.\\
By relation \eqref{datadiscrete} we have
\[
d' = R'\phi' + n'
\]
at the later time $t'$. With help of the resulting posterior $\mathcal{P}_{\phi' \vert d'}$ for this time, we then construct a rule to compute $d'$ from $d$, $\mathcal{P}_{\phi \vert d}$ and $\mathcal{P}_{\phi' \vert d'}$.

\chapter{Updating the data}
\label{chap:Updating the data}
In this chapter, with the notation from Definition \ref{simplifying_notation}, a rule to update the data vector $d$ to the new data $d'$, according to the evolving posterior $\mathcal{P}_{\phi \vert d}$ and the posterior $\mathcal{P}_{\phi' \vert d'}$ at time $t'$, and only depending on the old data, shall be constructed. We will try to follow the steps in \cite{enss} closely, imbedding them into our mathematical framework.
Throughout this chapter, we suppose all the assumptions from Chapter \ref{chap:setting} to be valid.

\section{Field dynamics}
\label{sec:Field dynamics}
First of all, remember Definition \ref{phi_n,N} and Definition \ref{simplifying_notation}, which give us the relation
\begin{align}
\phi' = A \phi,
\end{align}
with $A\in \mathrm{map}(\mathbb{R}^n,\mathbb{R}^{n})$.\\
On the one hand, we would like to stay with a matrix in our updating, instead of a complicated operator from $\mathbb{R}^{Y'}$ to $\mathbb{R}^{Y}$, because this way we can assure measurability of the update procedure.\\
And on the other hand, we would like to have the evolution of our signal invertible at least for small time steps, because we need the requirement of the transformation formula for Lebesgue integrable functions to be fulfilled, in order to compute the evolved posterior $\mathcal{P}_{\phi' \vert d}$.
\\
This is why we make the following
\begin{assumption}\label{ass:linearization}
If $N\in \mathbb{N}$ is large enough, then for arbitrary $i\in \lbrace 0,...,2^N-1 \rbrace$, $n\in \mathbb{N}$ and the time evolution from $t=t_i$ to $t'= t_i+ \delta t$, with $\delta t = \frac{T}{2^N}$, we can find a matrix $L\in \mathbb{R}^{n\times n}$ and a vector $c\in \mathbb{R}^{n}$, such that for all $x\in \mathbb{R}^{n}$ the evolution operator $A$, which describes the evolution of the time discrete signal from $t$ to $t'$, satisfies
\begin{align}
Ax &= x + \delta t \left( Lx +c \right) + \mathcal{O}( (\delta t)^{2} )\\
&= \left( \mathbbm{1}_{\mathbb{R}^{n}} + \delta t \cdot L \right) x + \delta t \cdot c + \mathcal{O}( (\delta t)^{2} ).
\end{align}
$\mathcal{O}( (\delta t)^{2} )$ in this context shall indicate the existence of some $C\geq 0$, such that we have
\begin{align}
\sup\limits_{\Vert x \Vert_{2}\leq 1}\Vert Ax - \left( \mathbbm{1}_{\mathbb{R}^{n}} + \delta t \cdot L \right) x - \delta t \cdot c \Vert_{2} \leq C\cdot (\delta t)^{2}.
\end{align}
Equivalently, $\mathcal{O}( (\delta t)^{2} )$ means that $x \mapsto Ax - \left( \mathbbm{1}_{\mathbb{R}^{n}} + \delta t \cdot L \right) x - \delta t \cdot c$ is bounded by $C\cdot (\delta t)^{2}$ as an operator in $\mathrm{map}(\mathbb{R}^{n},\mathbb{R}^{n})$, where $\mathbb{R}^{n}$ is endowed with the norm $\Vert.\Vert_{2}$.\\
We also assume, that $C=C_{\mathrm{max}}$ is the maximum of those bounds over all the $2^N$ evolution steps and that $C$ is independent of $N$.
\end{assumption}

In the following, we illustrate a case in which Assumption \ref{ass:linearization} is valid.\\
Remember the evolution equation \eqref{pde} that we started with in Chapter \ref{chap:setting}. We are interested in the evolution from the time $t=t_{i}$ to the time $t'=t_{i+1}=t_{i} + \delta t$.\\
We suppose Assumption \ref{ass:finitedim} to hold and introduce the finite dimensional subspaces $V_{n}$ of $V$ again.\\
Now we fix a maybe large $N_0\in \mathbb{N}$ and assume for all $N\geq N_0$ that:
\begin{assumption}\label{one_step_evolution}
For arbitrary $i\in \left\lbrace 0,...,2^N-1 \right\rbrace$, the evolution of the signal $\phi_{\mathrm{real}}$ from $t=t_i$ to $t'=t_i + \delta t$ can be described by an operator
\[
A_i:V\times [t,t'] \rightarrow V,
\]
in the sense that
\begin{align}
\phi_{\mathrm{real}}\left(\tilde{t}\right)=A_i\left(\phi_{\mathrm{real}}(t),\tilde{t}\right) \text{ for } \tilde{t}\in [t,t'].
\end{align}
If one restricts $A_i$ to $V_n$ for $n\in \mathbb{N}$, then one has
\[
A_i:V_n\times [t,t'] \rightarrow V_n.
\]
Furthermore, we assume that there is an $M_{i,N}\geq 1$, such that for all $\tilde{t}\in [t,t']$ and $f,g\in V$:
\begin{align}
\left\Vert A_i\left(f,\tilde{t}\right) - A_i\left(g,\tilde{t}\right) \right\Vert_V \leq M_{i,N}\cdot \Vert f-g \Vert_V.
\end{align}
For arbitrary $N\geq N_0$, the number $M_N:= \max\left\lbrace M_{i,N}: i\in \lbrace 0,...,2^N-1 \rbrace \right\rbrace$ has the property
\begin{align}
M_N^{2^N} \leq M \label{smoothness_of_small_steps}
\end{align}
for some $M\geq 1$.
\end{assumption}

\comment{
The inequality \eqref{smoothness_of_small_steps} in Assumption \ref{one_step_evolution} can be interpreted in the following way. If we make more steps in our time discretization, then we assume the error for every single time step to get smaller, because we assume the signal not to change its properties instantaneously. This smoothness assumption is imprinted in \eqref{smoothness_of_small_steps}. 
}\\

Now for a fix $N\geq N_0$ we can simplify the real signal in the following way:
\begin{definition}\label{def:discrete_signal}
If the Assumptions \ref{ass:finitedim} and \ref{one_step_evolution} are valid, then for $n\in \mathbb{N}$ a simplified signal $\phi^{(n)}_{N}: [0,T] \rightarrow V_{n}$ can be defined as
\begin{align}
&\phi^{(n)}_{N}(\tilde{t}) :=A_i\left(\phi^{(n)}_{N}(t_i),\tilde{t}\right)  \text{ if } \tilde{t}\in (t_{i},t_{i+1}] \text{ for } i\in \lbrace 0,...,2^N-1 \rbrace \text{ and } \\
&\phi_{N,0}^{(n)} = \phi^{(n)}_{N}(0) := f^{(n)}\\
&\text{for some function } f^{(n)} \in V_{n} \text{ with } \left\Vert f^{(n)} - \phi_{\mathrm{real},0} \right\Vert_{V} \leq \epsilon_{2,n}.
\end{align} 
\end{definition}
This way, we get the following estimate:
\begin{lemma}\label{lem:approx_bound}
With the Assumptions \ref{ass:finitedim} and \ref{one_step_evolution}, we can estimate for $\tilde{t}\in [0,T]$:
\begin{align}
\left\Vert \phi_{\mathrm{real}}(\tilde{t}) - \phi^{(n)}_{N}(\tilde{t}) \right\Vert_V \leq M\cdot \epsilon_{2,n},
\end{align}
with $M$ from Assumption \ref{one_step_evolution}.
\end{lemma}
\begin{proof}
For $\tilde{t}=0$, this is true by Definition \ref{def:discrete_signal}, because $M_{N}\geq 1$ and therefore $M_N\leq M$ with \eqref{smoothness_of_small_steps}.\\
Otherwise, it is $\tilde{t}\in (t_{i},t_{i+1}]$ for some $i\in \lbrace 0,...,2^N-1 \rbrace$.\\
Then:
\begin{align}
\left\Vert \phi_{\mathrm{real}}(\tilde{t}) - \phi^{(n)}_{N}(\tilde{t}) \right\Vert_V &\overset{\mathrm{Ass.}\, \ref{one_step_evolution}, \mathrm{Def.}\, \ref{def:discrete_signal}}{=} \left\Vert A_i\left(\phi_{\mathrm{real}}(t_{i}),\tilde{t}\right) - A_i\left(\phi^{(n)}_{N}(t_{i}),\tilde{t}\right) \right\Vert_V\\
&\overset{\mathrm{Ass.}\, \ref{one_step_evolution}}{\leq} M_N\cdot \left\Vert \phi_{\mathrm{real}}(t_{i}) - \phi^{(n)}_{N}(t_{i}) \right\Vert_V\\
&\overset{\mathrm{Ass.}\, \ref{one_step_evolution} , \mathrm{Def.}\, \ref{def:discrete_signal}}{\leq} \cdots \overset{\mathrm{Ass.}\, \ref{one_step_evolution} , \mathrm{Def.}\, \ref{def:discrete_signal}}{\leq} M_N^{2^N}\cdot \epsilon_{2,n} \overset{\mathrm{Ass.}\, \ref{one_step_evolution}}{\leq} M\cdot \epsilon_{2,n}.
\end{align}
\end{proof}
Since $M$ is constant, $M\cdot \epsilon_{2,n}$ converges to zero as $n\rightarrow \infty$, so that our simplified signal represents a reasonable approximation of the real signal.\\
We want to regard the simplified signal as a vector in $\mathbb{R}^n$ and therefore further define:
\begin{definition}\label{def:discrete_evolution}
In the setting of Lemma \ref{lem:approx_bound}, for $n\in \mathbb{N}$ and $V_n= \mathrm{span}\left(f_1,...,f_n\right)$, we first define the projection operator $P_n:V_n \rightarrow \mathbb{R}^{n}$ by
\begin{align}
P_n\left(\sum\limits_{k=1}^{n}a_{k}\cdot f_{k}\right):=(a_1,...,a_n)^T.
\end{align}
For arbitrary $i\in \lbrace 0,...,2^N-1 \rbrace$, $t=t_i$ and $t'=t_{i+1}$, we then define the operator
\[
A_i^{(n)}:\mathbb{R}^{n}\times [t,t'] \rightarrow \mathbb{R}^{n},
\]
that describes the evolution of the simplified field in $\mathbb{R}^n$ from $t$ to $t'$, as
\begin{align}
A_i^{(n)}\left((a_1,...,a_n)^T,\tilde{t}\right) := P_n \circ A_i\left(\sum\limits_{k=1}^{n}a_{k}\cdot f_{k},\tilde{t}\right)
\end{align}
for $(a_1,...,a_n)^T \in \mathbb{R}^{n}$ and $\tilde{t}\in [t,t']$.
\end{definition}
Note that $A_i^{(n)}$ is well-defined for every $i$, because by Assumption \ref{one_step_evolution} we have 
\[A_i:V_n\times [t,t'] \rightarrow V_n.\]
The operators from Definition \ref{def:discrete_evolution} allow us to approximate the original equation in the time interval $[t,t']$ by the evolution of $P_n\circ \phi^{(n)}_{N}$ with $\phi^{(n)}_{N}$ from Definition \ref{def:discrete_signal}.\\
For $\tilde{t}\in [t,t']$, we have 
\begin{align}
A_i^{(n)}\left(.,\tilde{t}\right) = \sum\limits_{k=1}^{n} e_{k}\cdot A^{(n)}_{i,k}\left(.,\tilde{t}\right),\label{fin_dim_A}
\end{align}
for $e_{k}= (0,...,0,\underbrace{1}_{\substack{k-th\; place}},0,...,0)^T$ and some operators $A^{(n)}_{i,k}\left(.,\tilde{t}\right)\in \mathrm{map}(\mathbb{R}^{n},\mathbb{R})$.\\
Because a signal that does not evolve at all should stay the same, we expect $A_i(.,t)$ to be the identity in $V$ for all $i$.\\
Therefore, we assume
\begin{assumption}\label{ass:identity_starting_time}
In the setting of Lemma \ref{lem:approx_bound}, for arbitrary $i\in \lbrace 0,...,2^N-1 \rbrace$ and the time $t=t_i$, it is
\begin{align}
A_i^{(n)}(.,t)=\mathbbm{1}_{\mathbb{R}^{n}}
\end{align}
and for $k=1,...,n$, it is
\[
A^{(n)}_{i,k}(.,t) \overset{\eqref{fin_dim_A}}{=} e_{k}^{T}.
\]
\end{assumption}
If we also assume that:
\begin{assumption}\label{smoothness_evol_op}
For arbitrary $i\in \left\lbrace 0,...,2^N-1 \right\rbrace$, $t=t_i$, $t'=t_{i+1}$, each $x\in \mathbb{R}^{n}$ and for all $k=1,...,n$, the maps $A^{(n)}_{i,k}(x,.):[t,t']\rightarrow \mathbb{R}$ are twice Fr\'{e}chet differentiable.\\
For arbitrary but fixed $\xi_{i,k}\in (t,t')$, every operator $x\mapsto \dfrac{1}{2}\dfrac{\partial^2}{\partial\tilde{t}^{2}} A^{(n)}_{i,k}(x,\xi_{i,k})$ is bounded in $\mathrm{map}(\mathbb{R}^{n}, \mathbb{R})$ by some $C_{i,k}>0$.\\
The number $C = \max \left\lbrace C_{i,k} : i\in \left\lbrace 0,...,2^N-1 \right\rbrace,\,  k\in \lbrace 1,...,n \rbrace \right\rbrace$ is independent of $N$.
\end{assumption}

Then we can state the following lemma which brings us close to Assumption \ref{ass:linearization}:
\begin{lemma}\label{lemma:smoothness_evol_op}
In the setting of Lemma \ref{lem:approx_bound}, suppose that also the Assumptions \ref{ass:identity_starting_time} and \ref{smoothness_evol_op} hold.\\
Let $i\in \left\lbrace 0,...,2^N-1 \right\rbrace$ be arbitrary and $t=t_i$ as well as $t'=t_{i+1}$.\\
Then for every $x\in \mathbb{R}^{n}$, $k \in \lbrace 1,...,n \rbrace$ and $L_{i,k}(x):= \dfrac{\partial}{\partial\tilde{t}} A^{(n)}_{i,k}(x,t)$, it is
\begin{align}
A^{(n)}_{i,k}\left(x,\tilde{t}\right)= e_{k}^{T}x + L_{i,k}(x)\cdot \left(\tilde{t}-t\right) + \dfrac{1}{2}\dfrac{\partial^2}{\partial \tilde{t}^{2}} A_{i,k}(x,\xi_{i,k})\cdot \left(\tilde{t}-t\right)^{2}
\end{align}
for $\tilde{t}\in (t,t']$ with some numbers $\xi_{i,k}\in (t,t')$.\\
Furthermore,
\begin{align}
\sum\limits_{k=1}^{n}e_{k}\cdot \dfrac{1}{2}\dfrac{\partial^2}{\partial\tilde{t}^{2}} A^{(n)}_{i,k}(x,\xi_{i,k})\cdot \left(\tilde{t}-t\right)^{2} = \mathcal{O}\left(\left(\tilde{t}-t\right)^{2}\right),\label{order_approx}
\end{align} 
where the bound for the operator norm is given by $\dfrac{n}{2}\cdot C$ with $C$ from Assumption \ref{smoothness_evol_op}.
\end{lemma}

\begin{proof}
With Assumption \ref{smoothness_evol_op}, by Taylor's Theorem \cite[Theorem III.5.5]{werner}, it is
\begin{align}
A^{(n)}_{i,k}\left(x,\tilde{t}\right) &= A^{(n)}_{i,k}(x,t) + \underbrace{\dfrac{\partial}{\partial\tilde{t}} A^{(n)}_{i,k}(x,t)}_{\substack{=:L_{i,k}(x)}}\cdot \left(\tilde{t}-t\right) + \dfrac{1}{2}\dfrac{\partial^2}{\partial\tilde{t}^{2}} A^{(n)}_{i,k}(x,\xi_{i,k})\cdot \left(\tilde{t}-t\right)^{2}\\
&\overset{\mathrm{Ass.}\, \ref{ass:identity_starting_time}}{=} e_{k}^{T}x + L_{i,k}(x)\cdot \left(\tilde{t}-t\right) + \dfrac{1}{2}\dfrac{\partial^2}{\partial \tilde{t}^{2}} A_{i,k}(x,\xi_{i,k})\cdot \left(\tilde{t}-t\right)^{2}
\end{align}
for $\tilde{t}\in (t,t']$ and some $\xi_{i,k}$ with $t<\xi_{i,k}<\tilde{t}$.\\
By Assumption \ref{smoothness_evol_op}, every functional $x\mapsto \dfrac{1}{2}\dfrac{\partial^2}{\partial\tilde{t}^{2}} A^{(n)}_{i,k}(x,\xi_{i,k})$ is bounded by some $C_{i,k}>0$ in $\mathrm{map}(\mathbb{R}^{n}, \mathbb{R})$, and the maximum over the $C_{i,k}$ is less or equal to $C$.\\
Therefore, the map $x \mapsto \sum\limits_{k=1}^{n}e_{k}\cdot \dfrac{1}{2}\dfrac{\partial^2}{\partial\tilde{t}^{2}} A^{(n)}_{i,k}(x,\xi_{i,k})\cdot \left(\tilde{t}-t\right)^{2}$ is bounded by $\frac{n}{2}\cdot C \cdot \left(\tilde{t}-t\right)^{2}$ as an operator in $\mathrm{map}(\mathbb{R}^{n},\mathbb{R}^{n})$.\\
In the notation from Assumption \ref{ass:linearization}, this means
\begin{align}
\sum\limits_{k=1}^{n}e_{k}\cdot \dfrac{1}{2}\dfrac{\partial^2}{\partial\tilde{t}^{2}} A^{(n)}_{i,k}(x,\xi_{i,k})\cdot \left(\tilde{t}-t\right)^{2} = \mathcal{O}\left(\left(\tilde{t}-t\right)^{2}\right)
\end{align} 
with the bound $\dfrac{n}{2}\cdot C$.
\end{proof}

We need one more assumption about the smoothness of the operators $A_i^{(n)}$. 
\begin{assumption}\label{ass:derivative_of_A}
For arbitrary $i\in \left\lbrace 0,...,2^N-1 \right\rbrace$, $t=t_i$ and $k\in \lbrace 1,...,n\rbrace$, the map
$x\mapsto L_{i,k}(x)= \dfrac{\partial}{\partial\tilde{t}} A^{(n)}_{i,k}(x,t)$ is twice Fr\'{e}chet differentiable with vanishing second derivative.
\end{assumption}

Now we can state the following:
\begin{theorem}\label{thm:assumption}
If in addition to the assumptions from Lemma \ref{lemma:smoothness_evol_op} also Assumption \ref{ass:derivative_of_A} holds, then for arbitrary $i\in \left\lbrace 0,...,2^N-1 \right\rbrace$ and $t=t_i$, $t'=t_{i+1}$, $\phi=P_n \circ \phi^{(n)}_{N}(t)$, $\phi'=P_n \circ \phi^{(n)}_{N}(t')$ and $A(.)=A_i^{(n)}(., t')$, Assumption \ref{ass:linearization} is satisfied. The bound for the $\mathcal{O}\left(\left(\delta t\right)^{2}\right)$ estimate in the assumption is given by $\frac{n}{2}C$ with $C$ from Assumption \ref{smoothness_evol_op}.
\end{theorem}

\begin{proof}
With Assumption \ref{ass:derivative_of_A}, for the evolution from $t$ to $t'$ one has by Taylor's Theorem, that
\begin{align}
L_{i,k}(x) = \underbrace{L_{i,k}(0)}_{\substack{=:c_{k}}} + \underbrace{DL_{i,k}(0)}_{\substack{=:L_{k}^{T}}}x
\end{align}
for every $x\in \mathbb{R}^n$, so that $L_{i,k}(x) = c_{k} + L_{k}^{T}x$ with $c_{k}\in \mathbb{R}$ and $L_{k}\in \mathbb{R}^{n}$.\\ 
Therefore, with $L:=\sum\limits_{k=1}^{n}e_{k} \cdot L_{k}^{T} \in \mathbb{R}^{n\times n}$ and $c:=\sum\limits_{k=1}^{n}e_{k}\cdot c_{k}\in \mathbb{R}^{n}$, we have 
\begin{align}
A_i^{(n)}\left(x,\tilde{t}\right) &\overset{\eqref{fin_dim_A}}{=} \sum\limits_{k=1}^{n} e_{k} \cdot A^{(n)}_{i,k}\left(x,\tilde{t}\right)\\
&\overset{\mathrm{Lem.}\, \ref{lemma:smoothness_evol_op}}{=} \sum\limits_{k=1}^{n} \left( e_{k} \cdot e_{k}^{T}x + e_{k} \cdot (\tilde{t} - t) \cdot L_{i,k}(x) \right)+ \mathcal{O}\left(\left(\tilde{t}-t\right)^{2}\right)\\
&= \mathbbm{1}_{\mathbb{R}^{n}}x + \left(\tilde{t}-t\right)\cdot (Lx + c) + \mathcal{O}\left(\left(\tilde{t}-t\right)^{2}\right).\label{whole_approx}
\end{align}
With $\delta t = t' - t$, it follows
\begin{align}
\phi' \overset{\mathrm{Def.}\, \ref{def:discrete_signal}}{=} A_i^{(n)}(\phi, t') &\overset{\eqref{whole_approx}}{=} \mathbbm{1}_{\mathbb{R}^{n}}\phi + \delta t\cdot (L\phi + c) + \mathcal{O}((\delta t)^{2}).
\end{align}
Independence of the bound from $N$ follows directly with Assumption \ref{smoothness_evol_op} and Lemma \ref{lemma:smoothness_evol_op}.
\end{proof}

\begin{theorem}\label{thm:invertdynamics}
Assume in addition to the assumptions of Theorem \ref{thm:assumption}, that for arbitrary $i\in \left\lbrace 0,...,2^N-1 \right\rbrace$ and $k\in \lbrace 1,...,n \rbrace$, the maps $x\mapsto L_{i,k}(x) = \dfrac{\partial}{\partial\tilde{t}} A^{(n)}_{i,k}(x,t)$ have a common bound $C_L$ which is independent of $N$.\\
For a fix $i\in \left\lbrace 0,...,2^N-1 \right\rbrace$, consider the matrix $L=\sum\limits_{k=1}^{n}e_{k} \cdot L_{k}^{T}$, which is constructed from the map $x\mapsto L_{i,k}(x)$ in the proof of Theorem \ref{thm:assumption}.\\
Then there is an $N_0\in \mathbb{N}$, such that if $N\geq N_0$, the operator $\left( \mathbbm{1}_{\mathbb{R}^{n}} + \delta t \cdot L \right)$ is invertible with
\begin{align}
\left( \mathbbm{1}_{\mathbb{R}^{n}} + \delta t \cdot L \right)^{-1} = \mathbbm{1}_{\mathbb{R}^{n}} - \delta t \cdot L + \mathcal{O}( (\delta t)^{2} ).
\end{align}
If $N_0$ is small enough, so that $\delta t < \frac{1}{4\cdot n\cdot C_L}$, then the constant of the bound for the error of this estimate is smaller than $8\cdot n^2 \cdot C_L^2$.
\end{theorem}
\begin{proof}
The vector $c$ in the proof of Theorem \ref{thm:assumption} is bounded by $n\cdot C_L$, because
\begin{align}
\Vert c \Vert_2 &= \left\Vert \sum\limits_{k=1}^{n}e_{k}\cdot c_{k} \right\Vert_2 = \left\Vert \sum\limits_{k=1}^{n}e_{k}\cdot L_{i,k}(0) \right\Vert_2\\
&\leq \sum\limits_{k=1}^{n} \Vert L_{i,k}(0) \Vert_2 \leq n\cdot C_L.
\end{align}
For arbitrary $x\in \mathbb{R}^n$,
\begin{align}
\Vert Lx \Vert_2 \leq \Vert Lx +c \Vert_2 + \Vert c \Vert_2 = \left\Vert \sum\limits_{k=1}^{n}e_{k} L_{i,k}(x) \right\Vert_2 + \Vert c \Vert_2 \leq 2\cdot n\cdot C_L
\end{align}
by the triangle inequality and the bounded $C_L$ for the operators $L_{i,k}$.\\
It follows
\begin{align}
\Vert L \Vert \leq 2\cdot n \cdot C_L.
\end{align}

If we choose $N_0$ large enough, so that $\delta t= \frac{T}{2^N}$ gets small for $N\geq N_0$, then
\[
\Vert \delta t \cdot L \Vert = \delta t \cdot \Vert L \Vert
= \delta t \cdot \sup\limits_{\Vert x \Vert_{2} \leq 1} \Vert  Lx \Vert_{2} < 1.
\]
By the Neumann series \cite[cf.][Theorem II.1.11]{werner}, it follows
\begin{align}
\left( \mathbbm{1}_{\mathbb{R}^{n}} + \delta t \cdot L \right)^{-1} = \sum\limits_{k=0}^{\infty} ( -\delta t \cdot L)^{k} = \mathbbm{1}_{\mathbb{R}^{n}} - \delta t \cdot L + \mathcal{O}( (\delta t)^{2} ).\label{invertdynamics}
\end{align}
If $\delta t\leq 1$ and small enough, to that $\delta t\cdot 2\cdot n\cdot C_L < \frac{1}{2}$, the constant of the bound for the error of the estimate follows with
\begin{align}
\left\Vert \left( \mathbbm{1}_{\mathbb{R}^{n}} + \delta t \cdot L \right)^{-1} -  \left( \mathbbm{1}_{\mathbb{R}^{n}} - \delta t \cdot L\right)\right\Vert_2 &= \left\Vert \sum\limits_{k=2}^{\infty} ( -\delta t \cdot L)^{k}  \right\Vert_2\\
&\leq \left\Vert L \right\Vert_2^2 \cdot (\delta t)^2 \cdot\sum\limits_{k=0}^{\infty} \left(\delta t\cdot \left\Vert L \right\Vert_2 \right)^k\\
&\leq (\delta t)^2 \cdot 4\cdot n^2 \cdot C_L^2 \cdot \sum\limits_{k=0}^{\infty} \left(\delta t\cdot \left\Vert L \right\Vert_2\right)^k\\
& \leq (\delta t)^2 \cdot  \frac{4\cdot n^2 \cdot C_L^2}{1- \delta t\cdot 2\cdot n\cdot C_L}\\
&< (\delta t)^2 \cdot  8\cdot n^2 \cdot C_L^2,
\end{align}
where the second last inequality follows by the geometric series.
\end{proof}

\comment{
In the setting of Theorem \ref{thm:invertdynamics}, the approximation error that occurs if we abort the series in \eqref{invertdynamics} after $\mathbbm{1}_{\mathbb{R}^{n}} - \delta t \cdot L$ is in $\mathcal{O}((\delta t)^2)$, with the same constant $8\cdot n^2 \cdot C_L^2$ for every evolution step and independently of $N$. Also the error that comes with the approximation of $A$ by $\left( \mathbbm{1}_{\mathbb{R}^{n}} + \delta t \cdot L \right) + \delta t \cdot c$ in Theorem \ref{thm:assumption} is in $\mathcal{O}((\delta t)^2)$, and the constant of the bound $\frac{n}{2}C$, with $C$ from Assumption \ref{smoothness_evol_op}, is the same for every step and independent of $N$. This way, every simplification up to order $\mathcal{O}((\delta t)^2)$, which is made with an abortion of the series in \eqref{invertdynamics} after $\mathbbm{1}_{\mathbb{R}^{n}}- \delta t \cdot L$ and with respect to the modeling of the evolution of the simplified signal by an operator of the form $\left( \mathbbm{1}_{\mathbb{R}^{n}} + \delta t \cdot L \right) + \delta t \cdot c$ in every single time step, leads to an approximation of the simulation until time $T$ which is still in $\mathcal{O}\left( \dfrac{T^2}{2^N}\right)$. The constant of the bound for the error here is quadratic in $n$.  All the above assumptions are fulfilled if the evolution equation is smooth enough in time. This is the case if the modeled physical process is sufficiently stable.\\
Lemma \ref{lem:approx_bound} assures that
\begin{align}
\sup\limits_{\tilde{t}\in [0,T]} \left\Vert\phi_{\mathrm{real}}(\tilde{t}) - \phi^{(n)}_{N}(\tilde{t}) \right\Vert_V \leq M\cdot \epsilon_{2,n}
\end{align}
for some $M\geq 1$.\\
Suppose that the desired bound for the error from Theorem \ref{thm:assumption} and approximations like in \eqref{invertdynamics} of the whole simulation requires 
\[
\frac{n^2}{2^N}\leq \epsilon
\]
for some $\epsilon> 0$.\\
Even if $n$ has to be large in order to get $M\cdot \epsilon_{2,n}$ small and thereby stay close to the real signal, the error from the approximation in \eqref{invertdynamics} and Theorem \ref{thm:assumption} can still be kept small by a sufficiently large chosen $N$.
}

\section{Prior knowledge}
\label{sec:prior_knowledge}
In this section we want to specify $\mathbb{P}_{\phi}$ from physical knowledge about the evolution equation, as well as from experience gained through measurements in the past.\\
The physical knowledge can not be incorporated in the general setting, because the equation differs from case to case.\\
For the initial time $t$, we assume that we can at least approximate the distribution which describes our prior knowledge of the field $\phi$ by a Gaussian distribution.

\begin{assumption}
$\mathbb{P}_{\phi}$ has the $\lambda_{\mathbb{R}^{n}}$-derivative
\begin{align}
\mathcal{P}_{\phi}(s) = \mathcal{G}(s - \psi, \Phi),
\end{align}
with a positive definite covariance matrix $\Phi\in \mathbb{R}^{n\times n}$ and a mean $\psi \in \mathbb{R}^{n}$. Compare this also to \eqref{distributionsignal} in Assumption \ref{Wiener_filter_setting}.
\end{assumption}
In case one only has an estimate for the mean value $\psi:=\mathbb{E}[\phi]$ and the covariance structure $\Phi:=\mathbb{E}\left[(\phi-\psi)(\phi-\psi)^{T}\right]$ of the discretized field at time $t$ from experience and if one assumes that $\phi$ has a density w.r.t. the Lebesgue measure, then this supposition can be justified by the maximum entropy principle, which will be discussed in Chapter \ref{chap:Maximum_Entropy_Principle}. Because in this case, by \eqref{least_information_entropy}, the entropy $h(\phi)$ is maximized for a random vector $\phi$ with Gaussian distribution, so that the density $\mathcal{P}_{\phi}(s)$ is given by $\mathcal{G}(s - \psi, \Phi)$. As entropy is a measure for uncertainty and the lack of information, and because we only want the prior to contain information that we actually have, this is a reasonable choice for our prior.

\section{Data constraints}
\label{sec:dataconstraints}
Having specified a prior, we would like to compute the posterior from $\mathcal{P}_{\phi}$, $\mathcal{P}_n$ and the constraints that arise due to the relation between the signal and our data vector.\\
As in \eqref{datasignalrelation}, we assume this relation to be given by
\[
d = R\phi + n.
\]
We also assume \eqref{distributionnoise} and \eqref{independence} in Assumption \ref{Wiener_filter_setting} to hold.
\begin{theorem}
The posterior $\mathcal{P}_{\phi\vert d}$ is given by
\begin{align}
\mathcal{P}_{\phi \vert d}(s,u) &= \mathcal{G}(s-m(u),D)\ \lambda_{\mathbb{R}^{n}}\otimes \mathbb{P}_{d}\text{-a.e in } \mathbb{R}^{n}\times \mathbb{R}^{Y} \text{ with}\\
\text{uncertainty variance }D &= \left( \Phi^{-1} + R^{T}N^{-1}R  \right)^{-1} \text{ and}\\
\text{mean }m(u) &= \psi + W(u-R\psi) = D(R^{T}N^{-1}d + \Phi^{-1}\psi).
\end{align}
\end{theorem}
\begin{proof}
The calculation of our likelihood in the Wiener filter Section \ref{Wiener filter} in Lemma \ref{likelihood} was independent of the concrete choice of $\mathbb{P}_{\Phi}$. Thus, also with our new assumption of a possible mean $\psi$, one still has the likelihood
\begin{align}
\mathcal{P}_{d \vert \phi}(u,s) = \mathcal{G}(u - Rs,N)
\end{align}
for $\lambda_{\mathbb{R}^{Y+n}}$-a.e. $(u,s)\in \mathbb{R}^{Y}\times \lbrace s\in \mathbb{R}^n: \mathcal{P}_{\phi}(s) >0 \rbrace = \mathbb{R}^{Y}\times \mathbb{R}^{n}$.
With the definitions of the positive definite information propagator $D$ and the information source $j(d)$ from Section \ref{Wiener filter}, and because $D^{-1}=(D^{-1})^{T}$, the information Hamiltonian in this case reads
\begin{align}
H_{d,\phi}(u,s) &=  -\log \mathcal{P}_{d \vert \phi}(u,s) - \log \mathcal{P}_{\phi}(s)\\
&= \frac{1}{2} (u-Rs)^{T}N^{-1}(u-Rs) + \frac{1}{2} (s - \psi)^{T}\Phi^{-1}(s - \psi) + c_{1}\\
&= \frac{1}{2} \left[ s^{T} D^{-1} s - s^{T}j(u) - j(u)^{T}s - s^{T}\Phi^{-1}\psi -\psi^{T}\Phi^{-1}s \right] + c_{2}(u)\\
&= \frac{1}{2} \left[ s^{T} D^{-1} s - s^{T}D^{-1}Dj(u) - (Dj(u))^{T}D^{-1}s - s^{T}D^{-1}D\Phi^{-1}\psi \right.\\
&\left. -(D\Phi^{-1}\psi)^{T}D^{-1}s \right] + c_{2}(u)\\
&= \frac{1}{2} (s-m(u))^{T}D^{-1}(s-m(u)) + c_{3}(u),
\end{align}
with $c_{1},c_{2}(u)$ and $c_{3}(u)$ independent of $s$ and the mean field 
\[
m(d) = D( j(d) + \Phi^{-1}\psi ) = D(R^{T}N^{-1}d + \Phi^{-1}\psi) = Wd + D\Phi^{-1}\psi.
\]
Furthermore, we have the equality
\begin{align}
\psi - WR\psi &\overset{\eqref{def:Wienerfilter}}{=}  \psi - DR^{T}N^{-1}R\psi - D\Phi^{-1}\psi + D\Phi^{-1}\psi\\
&= \psi - D\underbrace{(\Phi^{-1} + R^{T}N^{-1}R)}_{\substack{=D^{-1}}}\psi + D\Phi^{-1}\psi\\
&= D\Phi^{-1}\psi,
\end{align}
so that
\[
m(d) = Wd + D\Phi^{-1}\psi = \psi + W(d-R\psi).
\]
As in \eqref{posteriorcalculation} in the proof of Theorem \ref{thm:posterior}, one then shows, that the posterior in this case is given by
\begin{align}
\mathcal{P}_{\phi \vert d}(s,u) &= \mathcal{G}(s-m(u),D)\ \lambda_{\mathbb{R}^{n}}\otimes \mathbb{P}_{d}\text{-a.e in } \mathbb{R}^{n}\times \mathbb{R}^{Y} \text{ with}\\
\text{uncertainty variance }D &= \left( \Phi^{-1} + R^{T}N^{-1}R  \right)^{-1} \text{ and}\\
\text{mean }m(u) &= \psi + W(u-R\psi) = D(R^{T}N^{-1}d + \Phi^{-1}\psi).
\end{align}

\end{proof}
\section{Field evolution}
\label{field_evolution}
For an observed data vector $d$ at the initial time, we now know that our posterior is given by $\mathcal{P}_{\phi \vert d}(s,d) = \mathcal{G}(s-m(d),D)\ \lambda_{\mathbb{R}^{n}}$-a.e. in $\mathbb{R}^{n}$ , which determines the posterior mean field by  $m(d)=\psi + W(d-R\psi) = D(R^{T}N^{-1}d + \Phi^{-1}\psi)$.\\
This posterior evolves together with the field to the later time $t'$, leading to the new posterior $\mathcal{P}_{\phi' \vert d}$.\\
To compute this new posterior, we need the evolution equation in operator form. We wish to apply the transformation formula for Lebesgue integrable functions, which requires invertibility as well as continuous differentiability of the operator and its inverse. This is why we first simplify the evolution, using our information about the field dynamics from Section \ref{sec:Field dynamics}.\\
With Assumption \ref{ass:linearization}, we approximate $\phi'$ by
\begin{align}
\phi'' = \left( \mathbbm{1}_{\mathbb{R}^{n}} + \delta t \cdot L \right) \phi + \delta t \cdot c =: g(\phi)\label{evolutionfunction}
\end{align}
and get
\begin{align}
\Vert \phi'' - \phi' \Vert_{2} \leq C\cdot (\delta t)^{2}
\end{align}
for some $C \geq 0$, and $C$ is independent of $N$ and equal for every evolution step.\\
\comment{
One should remember the definition $\delta t= \frac{T}{2^N}$ here. The bound in the above approximation, $C\cdot (\delta t)^{2}$, thus reads as $C\cdot \dfrac{T^2}{2^{2N}}$. We make this approximation in every evolution step $t_{i}$ to $t_{i+1}$ for $i=0,...,2^N-1$, and $C$ is the same for all steps by Assumption \ref{ass:linearization}.\\
After $2^N$ time steps, this leads to an approximation error in $\mathbb{R}^{n}$ still less or equal to $C\cdot \dfrac{T^2}{2^N}$. This error gets arbitrarily small for large $N$, which justifies the approximation of $\phi'$ by $\phi''$. Nevertheless, one should be aware of this approximation when choosing $N$ and $n$ in the beginning.
}\\

We would like to have a probability kernel which comes from a density like in our construction of the posterior in Section \ref{sec:expectation} and which satisfies $(\mathcal{P}_{\phi'' \vert d}(.,d) \cdot \lambda_{\mathbb{R}^{n}})(.) = \mathbb{P}\left[ \phi''\in . \vert d \right]$ a.s.
In order to get this, we need a bit of preparation in terms of functional calculus \cite[cf.][VII.1]{werner}.\\
For every real or complex quadratic $n$-dimensional Matrix $A$, there is a unitary matrix $U\in \mathbb{C}^{n}$, i.e. $UU^{\dagger}=U^{\dagger}U=\mathbbm{1}_{\mathbb{C}^{n}}$, such that
\begin{align}
U^{\dagger}AU = D, \label{diagonal}
\end{align}
with a diagonal matrix $D$, which contains the eigenvalues $a_{k}$, $k=1,...n$, of $A$.\\
Equivalently, there is an orthonormal basis $\lbrace x_{k}\rbrace_{k=1}^{n}$ of $\mathbb{C}^{n}$, consisting of eigenvectors of $A$.
The orthogonal Projektions $P_{j}$ onto the eigenspaces $V_{j}$ of the $m$ pairwise distinct eigenvalues $a_{k_{j}}$, $j=1,...,m$, are then given by
\begin{align}
P_{j}:= \sum\limits_{l=1}^{\mathrm{dim}(V_{j})}x_{k_{j},l}\cdot x_{k_{j},l}^{\dagger}, \label{projection}
\end{align}
where $\left\lbrace x_{k_{j},l} \right\rbrace_{l=1}^{\mathrm{dim}(V_{j})}$ is the subset of vectors of $\lbrace x_{k}\rbrace_{k=1}^{n}$, that builds an orthonormal basis of the eigenspace $V_{j}$ for the eigenvalue $a_{k_{j}}$ of $A$.\\
With those orthogonal projections, one gets
\[
A = \sum\limits_{j=1}^{m}a_{k_{j}}P_{j}.
\]
The connection here is, that
\begin{align}
\sum\limits_{k=1}^{n}a_{k} e_{k}e_{k}^{\dagger} = D = U^{\dagger}AU = \sum\limits_{j=1}^{m}a_{k_{j}} U^{\dagger}P_{j}U = \sum\limits_{j=1}^{m}a_{k_{j}}\sum\limits_{l=1}^{\mathrm{dim}(V_{j})}U^{\dagger}x_{k_{j},l}\cdot x_{k_{j},l}^{\dagger}U, \label{conn_U}
\end{align}
with $e_{k}=(0,....,0,\underbrace{1}_{\substack{k\mathrm{-th}\; \mathrm{entry}}},0,...,0)^{\dagger}$.\\
So the $k$-th column of $U$ is given by one of the vectors $x_{k_{j},l}$, $l=1,...,\mathrm{dim}(V_{j})$, where $a_{k}=a_{k_{j}}$ for $j\in \lbrace 1,...,m\rbrace.$\\
For the trace of $A$, one has
\begin{align}
\mathrm{Tr}(A) = \sum\limits_{k=1}^{n} x_{k}^{\dagger}Ax_{k} = \sum\limits_{k=1}^{n}\sum\limits_{j=1}^{m}a_{k_{j}} x_{k}^{\dagger}P_{j}x_{k} = \sum\limits_{k=1}^{n}\sum\limits_{j=1}^{m}a_{k_{j}}\sum\limits_{l=1}^{\mathrm{dim}(V_{j})} \underbrace{x_{k}^{\dagger}x_{k_{j},l}}_{\substack{=\delta_{k\, k_{j},l}}}x_{k_{j},l}^{\dagger}x_{k} = \sum\limits_{k=1}^{n} a_{k}.
\label{trace}\end{align}
\begin{definition}[Functional calculus]\label{func_calc}
If $f$ is a function which is well defined for all the eigenvalues of $A$, one can define the operator 
\begin{align}
f(A):=\sum\limits_{j=1}^{m}f(a_{k_{j}})P_{j}.
\end{align}
\end{definition}
A couple of properties of this operator are stated in the following:

\begin{lemma}\label{lemma:func_calc}
For the trace of $f(A)$ one has
\begin{align}
\mathrm{Tr}(f(A)) = \sum\limits_{k=1}^{n}f(a_{k}).\label{tr_func_calc}
\end{align}
With $U$ from \eqref{diagonal} and the eigenvalues $a_1,...,a_n$ of $A$,
\begin{align}
U^{\dagger}f(A)U = f(D)\label{trafo_func_calc}
\end{align}
is a diagonal matrix with entries $f(a_{k})$ for $k=1,...,n$.
\begin{align}
\det(f(A)) = \prod\limits_{k=1}^{n}f(a_{k}).\label{det_func_calc}
\end{align}
If $f=\exp$, one has
\begin{align}
\vert \exp(A) \vert = \exp(\mathrm{Tr}(A)).\label{det_exp}
\end{align}
If $A>0$, then $\log(A)$ is well defined and
\begin{align}
\mathrm{Tr}(\log(A)) = \log\left(\vert A \vert\right).\label{trace_log}
\end{align}
\end{lemma}

\begin{proof}
\eqref{tr_func_calc} follows with
\begin{align}
\mathrm{Tr}(f(A)) = \sum\limits_{k=1}^{n}\sum\limits_{j=1}^{m}f(a_{k_{j}})x_{k}^{\dagger}P_{j}x_{k} = \sum\limits_{k=1}^{n}f(a_{k}).
\end{align}
For $U$ from \eqref{diagonal}, with \eqref{conn_U}, also \eqref{trafo_func_calc} follows. Because then
\begin{align}
U^{\dagger}f(A)U = \sum\limits_{j=1}^{m}f(a_{k_{j}}) U^{\dagger}P_{j}U = \sum\limits_{k=1}^{n} f(a_{k}) e_{k}e_{k}^{\dagger} = f(D)
\end{align}
is a diagonal matrix with entries $f(a_{k})$ for $k=1,...,n$.\\
\eqref{det_func_calc} is given because
\begin{align}
\det(f(A)) = \det(UU^{\dagger}f(A))=\det(U^{\dagger}f(A)U)\overset{\eqref{trafo_func_calc}}{=}\prod\limits_{k=1}^{n}f(a_{k}).
\end{align}
The properties of the two special cases can be seen as follows.\\
If $f=\exp$, \eqref{det_exp} is given because
\begin{align}
\vert \exp(A) \vert &= \det(\exp(A)) \overset{\eqref{det_func_calc}}{=} \prod\limits_{k=1}^{n}\exp(a_{k}) \overset{\eqref{trace}}{=} \exp(\mathrm{Tr}(A)).
\end{align}
If $A>0$, so that $a_{k}>0$ for $k=1,...,n$ and so that all eigenvalues of $A$ are real values in particular, then $\log(A)$ is well defined and \eqref{trace_log} follows with
\begin{align}
\mathrm{Tr}(\log(A)) \overset{\eqref{tr_func_calc}}{=} \sum_{k=1}^{n}\log(a_{k}) = \log\left(\prod\limits_{k=1}^{n}(a_{k})\right) = \log(\det(D)) \overset{\eqref{diagonal}}{=} \log\left(\vert A \vert\right).
\end{align}
\end{proof}

Now we can state the following:
\begin{lemma}\label{invertability}
For $\left\vert \delta t \right\vert <1$ small enough and $n\geq 2$, the map $g: x \mapsto \left( \mathbbm{1}_{\mathbb{R}^{n}} + \delta t \cdot L \right)x + \delta t \cdot c$ is invertible, continuously differentiable and the same holds for its inverse. For the Jacobian of $g^{-1}$ at $g(x)\in \mathbb{R}^{n}$, one has
\[
\vert Dg^{-1}(g(x)) \vert = \vert Dg(x) \vert^{-1} = 1+ \delta t\cdot \mathrm{Tr}(L) + \mathcal{O}((\delta t)^{2}).
\]
\end{lemma}

\begin{proof}
Differentiability of the bounded linear map $g$  acting on the finite dimensional space $\mathbb{R}^{n}$ should be obvious. If $\delta t$ is small enough, so that $\Vert \delta t \cdot L \Vert < 1$, invertibility follows like in \eqref{invertdynamics} in the proof of Theorem \ref{thm:invertdynamics}.\\
For every $x\in \mathbb{R}^{n}$ with $\Vert x \Vert_{2}\leq 1$, we have 
\begin{align}
\Vert g^{-1}(x) \Vert_{2} &= \left\Vert \sum\limits_{k=0}^{\infty}(-\delta t \cdot L)^{k}(x - \delta t\cdot c) \right\Vert_{2}\\
&\leq \sum\limits_{k=0}^{\infty} \Vert \delta t \cdot L \Vert^{k} \cdot (\Vert x \Vert_{2} + \delta t\cdot \Vert c \Vert_{2})\\
&\overset{\mathrm{geom. series}}{=} (1-\delta t \cdot \Vert L \Vert)^{-1}\cdot (\Vert x \Vert_{2} + \delta t\cdot \Vert c \Vert_{2})\\
&\leq (1-\delta t \cdot \Vert L \Vert)^{-1}\cdot (1 + \delta t\cdot \Vert c \Vert_{2}).
\end{align}
Therefore, also $g^{-1}$ is linear and bounded as a map on the finite dimensional space $\mathbb{R}^{n}$. That is, $g$ is continuously invertible and also $g^{-1}$ is continuously differentiable. 
For the Jacobian of $g$ at $x\in \mathbb{R}^{n}$ one has
\begin{align}
\vert Dg(x) \vert &= \vert \mathbbm{1}_{\mathbb{R}^{n}} + \delta t \cdot L \vert\\
&\overset{\vert . \vert > 0}{=} \exp \log \left\vert \underbrace{\mathbbm{1}_{\mathbb{R}^{n}} + \delta t \cdot L }_{\substack{= \exp(\delta t \cdot L) + \mathcal{O}((\delta t)^{2})}} \right\vert\\
&\overset{*}{=} \exp \log \vert \exp(\delta t \cdot L) \vert + \mathcal{O}((\delta t)^{2})\\
&\overset{\eqref{det_exp}}{=} \exp \log \exp(\mathrm{Tr}(\delta t \cdot L)) + \mathcal{O}((\delta t)^{2})\\
&= \exp(\mathrm{Tr}(\delta t \cdot L)) + \mathcal{O}((\delta t)^{2})\\
&\overset{\mathrm{Taylorexp.}}{=} 1 + \delta t \cdot \mathrm{Tr}(L) + \mathcal{O}((\delta t)^{2}).
\end{align}
The equality $ \mathbbm{1}_{\mathbb{R}^{n}} + \delta t \cdot L = \exp(\delta t \cdot L) + \mathcal{O}((\delta t)^{2})$ follows from
\begin{align}
\Vert \exp(\delta t \cdot L) - \mathbbm{1}_{\mathbb{R}^{n}} - \delta t \cdot L \Vert
&= \left\Vert \sum\limits_{k=2}^{\infty}\dfrac{(\delta t \cdot L)^{k}}{k!} \right\Vert\\
&\leq (\delta t)^{2}\cdot \Vert L \Vert^{2} \cdot  \sum\limits_{k=0}^{\infty} \frac{\Vert \delta t \cdot L\Vert^{k}}{k!} \overset{\Vert \delta t \cdot L\Vert < 1}{\leq} (\delta t)^{2}\cdot \Vert L \Vert^{2}\cdot \exp(1)\\
&= (\delta t)^{2}\cdot \Vert L \Vert^{2}\cdot e.
\end{align}
$*$ holds by the Leibniz formula for the determinant. For a matrix $A = (a_{ij})_{1\leq i,j \leq n}$ one has by \cite[3.2.5.\, Theorem]{fischer}:
\begin{align}
\det(A) = \sum\limits_{\sigma \in S_{n}} \mathrm{sgn}(\sigma)\cdot a_{1\sigma(1)} \cdots a_{n\sigma(n)}.
\end{align}
Assume now that $A$ is of the form $A = \exp(\delta t \cdot B) + (\delta t)^2\cdot D = \sum\limits_{k=0}^{\infty} \frac{(\delta t \cdot B)^{k}}{k!} + (\delta t)^2\cdot D$, for some matrices $B$ and $D$, with $\Vert \delta t \cdot B \Vert < 1$. Then the formula for the determinant can be split into the determinant of $\mathbbm{1}_{\mathbb{R}^{n}} +  \delta t \cdot B$ and sum over all expressions, that contain at least one prefactor with a $(\delta t)^2$ in it. If we factor $(\delta t)^2$ out of this determinant and call the remaining part $C\in \mathbb{R}$, then from this representation, it follows
\begin{align}
\left\vert \vert A \vert - \vert \mathbbm{1}_{\mathbb{R}^{n}} +  \delta t \cdot B \vert \right\vert &\overset{A,\, \mathbbm{1}_{\mathbb{R}^{n}} +  \delta t \cdot B > 0}{=} \vert \det(A) - \det(\mathbbm{1}_{\mathbb{R}^{n}} +  \delta t \cdot B) \vert\\
&=\vert (\delta t)^2\cdot C \vert = (\delta t)^2\cdot \vert C \vert.
\end{align}
This gives the $\mathcal{O}((\delta t)^2)$ in $*$ with $B=L$ and $(\delta t)^2\cdot D= -\sum\limits_{k=2}^{\infty}\dfrac{(\delta t \cdot L)^{k}}{k!}$.\\
It is further by the chain rule
\begin{align}
1 &= \vert \mathbbm{1}_{\mathbb{R}^{n}} \vert = \vert D\mathbbm{1}_{\mathbb{R}^{n}} x \vert\\
&= \vert D(g^{-1}\circ g)(x) \vert = \vert Dg^{-1}(g(x)) Dg(x)) \vert = \vert Dg^{-1}(g(x)) \vert \cdot \vert Dg(x) \vert.\label{chain_rule}
\end{align}
We now take $\delta t$ small enough, so that $\vert \delta t \cdot \mathrm{Tr}(L) + \mathcal{O}((\delta t)^{2}) \vert < 1$, where with $\mathcal{O}((\delta t)^{2})$ we mean the bound $C$ from $*$. If $\vert Dg(x) \vert$ is given by $1 + \delta t \cdot \mathrm{Tr}(L) + (\delta t)^{2}\cdot C'$, one therefore gets by the geometric series
\begin{align}
\vert \vert Dg^{-1}(g(x)) \vert - (1 - \delta t \cdot \mathrm{Tr}(L))\vert &\overset{\eqref{chain_rule}}{=} \vert \vert Dg(x) \vert^{-1} - 1 + \delta t \cdot \mathrm{Tr}(L)\vert\\
&= \left\vert \frac{1}{1 + \delta t \cdot \mathrm{Tr}(L) + (\delta t)^{2}\cdot C'} - 1 + \delta t \cdot \mathrm{Tr}(L)  \right\vert\\
&= \left\vert \sum\limits_{k=0}^{\infty} (-\delta t \cdot \mathrm{Tr}(L) - (\delta t)^{2}\cdot C')^{k} - 1 + \delta t \cdot \mathrm{Tr}(L) \right\vert \\
&= \left\vert -(\delta t)^{2}\cdot C' + \sum\limits_{k=2}^{\infty} (-\delta t \cdot \mathrm{Tr}(L) - (\delta t)^{2}\cdot C')^{k} \right\vert\\
&= (\delta t)^{2}\cdot \left\vert -C' + \sum\limits_{k=2}^{\infty} (\delta t)^{k-2}\cdot(-\mathrm{Tr}(L) + \delta t \cdot C')^{k}\right\vert.
\end{align}
This makes the $\mathcal{O}((\delta t)^{2})$, because the series in the last equation converge due to 
\[
\vert \delta t \cdot \mathrm{Tr}(L) + (\delta t)^{2}\cdot C' \vert < 1,
\]
and our proof is finished.
\end{proof}
\comment{
We will not need the estimate $\vert Dg^{-1}(g(x)) \vert = 1+ \delta t\cdot \mathrm{Tr}(L) + \mathcal{O}((\delta t)^{2})$ from Lemma \ref{invertability} for the rest of this work. Nevertheless, the validity of this approximation was pointed out in \cite[p.8]{enss} and therefore proved as well.
}\\

With Assumption \ref{ass:linearization} we can now approximate $\phi'$ by
\begin{align}
\phi'' = \left( \mathbbm{1}_{\mathbb{R}^{n}} + \delta t \cdot L \right) \phi + \delta t \cdot c = g(\phi)
\end{align}
and get the following:
\begin{theorem}\label{posterior_update}
The approximated posterior $\mathcal{P}_{\phi'' \vert d}$ is given by 
\begin{align}
\mathcal{P}_{\phi'' \vert d}(s',u) &= \mathcal{G}( s' - m''(u),D'')\ \lambda_{\mathbb{R}^{n}}\otimes \mathbb{P}_{d}\text{-a.e in } \mathbb{R}^{n}\times \mathbb{R}^{Y}, \text{ with}\\
m''(u) &= m(u) + \delta t \cdot (c+ Lm(u))\\
&= (\mathbbm{1}_{\mathbb{R}^{n}} + \delta t\cdot L)m(u) +\delta t \cdot c\\
&= (\mathbbm{1}_{\mathbb{R}^{n}} + \delta t\cdot L)[ \psi + W(u-R\psi) ] +\delta t \cdot c \text{ and}\\
D'' &= \left[(\mathbbm{1}_{\mathbb{R}^{n}} + \delta t\cdot L)^{-1T} D^{-1} (\mathbbm{1}_{\mathbb{R}^{n}} + \delta t\cdot L)^{-1}\right]^{-1}\\
&= (\mathbbm{1}_{\mathbb{R}^{n}} + \delta t\cdot L)D(\mathbbm{1}_{\mathbb{R}^{n}} + \delta t\cdot L)^{T}\\
&= D + \delta t (LD + DL^{T}) + (\delta t)^{2}LDL^{T}.
\end{align}
\end{theorem}
\begin{proof}
By Lemma \ref{invertability}, the map $x \mapsto \left( \mathbbm{1}_{\mathbb{R}^{n}} + \delta t \cdot L \right)x + \delta t \cdot c = g(x)$ is continuously differentiable and the same holds for its inverse. By the transformation formula for Lebesgue integrable functions (trans.) \cite[\S 13., Theorem 2.]{forster}, and the substitution $\phi''=g(\phi)$ (subst.) \cite[Lemma 1.22]{kall}, one therefore gets for any $f \in L_{1}(\mathbb{R}^{n} \times \mathbb{R}^{Y},\mathcal{B}(\mathbb{R}^{n}\times \mathbb{R}^{Y}), \mathbb{P}_{\phi'',d})$:
\begin{align}
\mathbb{E}\left[ f(\phi'',d)\right] &= \int\limits_{\mathbb{R}^{n}\times \mathbb{R}^{Y}} f(\tilde{s},u) \mathbb{P}_{\phi'',d}(d(\tilde{s},u))\\
&\overset{\mathrm{subst.}}{=} \int\limits_{\mathbb{R}^{n}\times \mathbb{R}^{Y}} f(g(s),u) \mathbb{P}_{\phi,d}(d(s,u))\\
&\overset{\mathrm{Thm.} \ref{fubini}, \eqref{bayes}}{=} \int\limits_{\mathbb{R}^{Y}} \int\limits_{\mathbb{R}^{n}} \mathcal{P}_{\phi \vert d}(s,u) \cdot f(g(s),u)ds \mathbb{P}_{d}(du)\\
&\overset{\mathrm{trans.}}{=} \int\limits_{\mathbb{R}^{Y}} \int\limits_{\mathbb{R}^{n}} \mathcal{P}_{\phi \vert d}(g^{-1}(s'),u) \cdot f(s',u) \cdot \vert Dg^{-1}(s') \vert ds' \mathbb{P}_{d}(du).
\end{align}
That is, we have
\[
\mathcal{P}_{\phi'' \vert d}(s',u) = \mathcal{P}_{\phi \vert d}(g^{-1}(s'),u) \cdot \vert Dg^{-1}(s') \vert\ \lambda_{\mathbb{R}^{n}}\otimes \mathbb{P}_{d}\text{-a.e in } \mathbb{R}^{n}\times \mathbb{R}^{Y}.
\]
We also know from Theorem \ref{thm:posterior}, that $\mathcal{P}_{\phi \vert d}(s,u) = \mathcal{G}(s-m(u),D)\ \lambda_{\mathbb{R}^{n}}\otimes \mathbb{P}_{d}\text{-a.e in } \mathbb{R}^{n}\times \mathbb{R}^{Y}$.\\
Now we recall, that
\begin{align}
g^{-1}(x) &= (\mathbbm{1}_{\mathbb{R}^{n}} + \delta t \cdot L)^{-1}(x-\delta t \cdot c) \text{ for } x\in \mathbb{R}^{n},
\end{align}
and that by Lemma \ref{invertability}, we have
\begin{align}
\vert Dg^{-1}(g(x)) \vert = \vert Dg(x) \vert^{-1} = \vert (\mathbbm{1}_{\mathbb{R}^{n}} + \delta t \cdot L) \vert^{-1}.
\end{align}
This leads to
\begin{align}
\mathcal{P}_{\phi'' \vert d}(s',u) &= \mathcal{G}(g^{-1}(s')-m(u),D)\cdot \vert Dg^{-1}(s') \vert \\
&= \mathcal{G}((\mathbbm{1}_{\mathbb{R}^{n}} + \delta t \cdot L)^{-1}(s'-\delta t \cdot c)-m(u),D)\cdot \vert (\mathbbm{1}_{\mathbb{R}^{n}} + \delta t \cdot L) \vert^{-1}\\
&= \mathcal{G}((\mathbbm{1}_{\mathbb{R}^{n}} + \delta t \cdot L)^{-1}(s'-\delta t \cdot c- (\mathbbm{1}_{\mathbb{R}^{n}} + \delta t \cdot L)m(u)),D)\cdot \vert (\mathbbm{1}_{\mathbb{R}^{n}} + \delta t \cdot L) \vert^{-1}\\
&= \mathcal{G}\left(s'-\delta t \cdot c- (\mathbbm{1}_{\mathbb{R}^{n}} + \delta t \cdot L)m(u) , \left[(\mathbbm{1}_{\mathbb{R}^{n}} + \delta t\cdot L)^{-1T} D^{-1} (\mathbbm{1}_{\mathbb{R}^{n}} + \delta t\cdot L)^{-1}\right]^{-1}\right)\\
&= \mathcal{G}\left(s'-\underbrace{(\delta t \cdot c + (\mathbbm{1}_{\mathbb{R}^{n}} + \delta t \cdot L)m(u)}_{\substack{=:m''(u)}}) , \underbrace{\left[(\mathbbm{1}_{\mathbb{R}^{n}} + \delta t\cdot L)^{-1T} D^{-1} (\mathbbm{1}_{\mathbb{R}^{n}} + \delta t\cdot L)^{-1}\right]^{-1}}_{\substack{=:D''}}\right)
\end{align}
$\lambda_{\mathbb{R}^{n}}\otimes \mathbb{P}_{d}$-a.e in $\mathbb{R}^{n}\times \mathbb{R}^{Y}$.\\
That is, we have 
\begin{align}
\mathcal{P}_{\phi'' \vert d}(s',u) &= \mathcal{G}( s' - m''(u),D'')\ \lambda_{\mathbb{R}^{n}}\otimes \mathbb{P}_{d}\text{-a.e in } \mathbb{R}^{n}\times \mathbb{R}^{Y}, \text{ with}\\
m''(u) &= m(u) + \delta t \cdot (c+ Lm(u))\\
&= (\mathbbm{1}_{\mathbb{R}^{n}} + \delta t\cdot L)m(u) +\delta t \cdot c\\
&= (\mathbbm{1}_{\mathbb{R}^{n}} + \delta t\cdot L)[ \psi + W(u-R\psi) ] +\delta t \cdot c \text{ and}\\
D'' &= \left[(\mathbbm{1}_{\mathbb{R}^{n}} + \delta t\cdot L)^{-1T} D^{-1} (\mathbbm{1}_{\mathbb{R}^{n}} + \delta t\cdot L)^{-1}\right]^{-1}\\
&= (\mathbbm{1}_{\mathbb{R}^{n}} + \delta t\cdot L)D(\mathbbm{1}_{\mathbb{R}^{n}} + \delta t\cdot L)^{T}\\
&= D + \delta t (LD + DL^{T}) + (\delta t)^{2}LDL^{T}.
\end{align}
\end{proof}

We will see later, that for the computation of $d'$ from $d$, we actually need $D''^{-1}$.\\
If $(\mathbbm{1}_{\mathbb{R}^{n}} - \delta t\cdot L)^{-1}$ is complicated to compute, we approximate $\mathcal{P}_{\phi'' \vert d}$ by a Gaussian density with mean $m^{*}(u)=m''(u)=(\mathbbm{1}_{\mathbb{R}^{n}} + \delta t\cdot L)[ \psi + W(u-R\psi) ] +\delta t \cdot c$ and a covariance matrix $D^{*}$ with inverse $D^{*-1}=D^{-1} - \delta t(D^{-1}L + L^{T}D^{-1})$.\\
This means, that we approximate $\phi'$ by a random element $\tilde{\phi}''$ which satisfies that
$\mathcal{P}_{\tilde{\phi}'' \vert d}$ is $\lambda_{\mathbb{R}^{n}}\otimes \mathbb{P}_{d}$-a.s. Gaussian density with this mean and covariance.
\begin{theorem}\label{thm:further_approximation}
With Assumption \ref{ass:linearization} and if the setting of Theorem \ref{thm:invertdynamics} is given, then for $N\in \mathbb{N}$ large enough,
\begin{align}
D''^{-1} = D^{-1} -\delta t(D^{-1}L + L^{T}D^{-1}) + \mathcal{O}((\delta t)^2)= D^{*}+\mathcal{O}((\delta t)^2).\label{D''approx}
\end{align}
If $N$ is chosen, such that $\delta t < \frac{1}{4\cdot n\cdot C_L}$, then constant for the bound for the error of this estimate is smaller than $16\cdot n^2 \cdot C_L^2 \cdot \Vert D^{-1}\Vert$.\\
Furthermore, an approximation of $\phi''$ by $\tilde{\phi}''$ leads to the same expectation value $m''(u)=m^{*}(u)$ and a variance $D^{*}$ which is $\mathcal{O}((\delta t)^2)$ close to the variance $D''$ of $\phi''$ in operator norm.\\
For $\delta t < \frac{1}{8n C_L\sqrt{3\Vert D^{-1} \Vert \Vert D \Vert}}$, the error for this estimate has a constant which is smaller than $288\cdot \Vert D \Vert^2 \cdot \Vert D^{-1} \Vert\cdot n^2 \cdot C_L^2 $.\\
For given $d=u\in \mathbb{R}^Y$, the information theoretical error 
\begin{align}
D(\mathcal{P}_{\phi'' \vert d}(.,u) &\Vert \mathcal{P}_{\tilde{\phi}'' \vert d}(.,u) ) = \mathbb{E}_{\phi'' \vert u} \left[ \log\left( \frac{\mathcal{P}_{\phi'' \vert d}(., u)}{\mathcal{P}_{\tilde{\phi}'' \vert d}(., u )} \right) \right],
\end{align}
as introduced in the next section and explained in Chapter \ref{chap:Maximum_Entropy_Principle}, that comes with the approximation of $\phi''$ by $\tilde{\phi}''$ is in $\mathcal{O}((\delta t)^2)$ as well.\\
If $\delta t < \frac{1}{4\cdot n\cdot C_L}$, the bound of the error has a constant smaller than $48\cdot n^3 \cdot C_L^2 \cdot \Vert D^{-1} \Vert \cdot \Vert D \Vert$.
\end{theorem}

\begin{proof}
If we are in the setting of Theorem \ref{thm:invertdynamics}, we can assume $\Vert \delta t \cdot L \Vert <1$ for $N$ large enough. That is, \eqref{D''approx} follows with 
\begin{align}
D''^{-1} &\overset{\mathrm{Thm.}\, \ref{posterior_update}}{=} (\mathbbm{1}_{\mathbb{R}^{n}} + \delta t\cdot L)^{-1T} D^{-1} (\mathbbm{1}_{\mathbb{R}^{n}} + \delta t\cdot L)^{-1}\\
&= \left(\sum\limits_{k=0}^{\infty}(-\delta t \cdot L^{T})^{k}\right)D^{-1}\left(\sum\limits_{k=0}^{\infty}(-\delta t \cdot L)^{k}\right)\\
&= D^{-1} -\delta t(D^{-1}L + L^{T}D^{-1}) + \mathcal{O}((\delta t)^2).
\end{align}
If $\delta t < \frac{1}{4\cdot n\cdot C_L}$, the bound for the constant of the error in \eqref{D''approx} follows similar to the computation in the proof of Theorem \ref{thm:invertdynamics}.\\
In this case, by definition of $D''$ and the assumption on the bound of $L$, it also follows
\begin{align}
\Vert D'' \Vert \leq \Vert D \Vert \cdot \left(  1+ \delta t \cdot 4 \cdot n \cdot C_L + \left( \delta t \cdot 4 \cdot n \cdot C_L\right)^2  \right) \leq 3\cdot\Vert D \Vert
\end{align}
An approximation of $\phi''$ by $\tilde{\phi}''$ clearly leads to the same expectation value $m''(u)=m^{*}(u)$. The statement that the variance $D^{*}$ of $\tilde{\phi}''$ is $\mathcal{O}((\delta t)^2)$ close to the variance $D''$ of $\phi''$ in operator norm follows with \eqref{D''approx}. Because we have $D^{*-1} = D''^{-1} + (\delta t)^2B$ for some matrix $B\in \mathbb{R}^{n\times n}$. With the preceding computation, one has $\Vert B \Vert \leq 16\cdot n^2 \cdot C_L^2 \cdot \Vert D^{-1}\Vert$.\\ For $\delta t$ small enough, so that $\Vert (\delta t)^2D''B\Vert <1$, by the Neumann series it follows
\begin{align}
\Vert D^{*} - D'' \Vert &= \Vert (D''^{-1} + (\delta t)^2B)^{-1} - D'' \Vert\\
&= \Vert (\mathbbm{1}_{\mathbb{R}^{n}} + (\delta t)^2 D''B)^{-1}(D''^{-1})^{-1} - D'' \Vert\\
&= \left\Vert \left(\sum\limits_{k=0}^{\infty}\left[-(\delta t)^2 D''B)\right]^{k}\right) D'' - D'' \right\Vert\\
&= \left\Vert \left(\sum\limits_{k=1}^{\infty}\left[-(\delta t)^2 D''B\right]^{k}\right)D'' \right\Vert\\
&\leq (\delta t)^2 \cdot \Vert D''B \Vert \cdot \Vert D'' \Vert \cdot\frac{1}{1-(\delta t)^2 \Vert D''B \Vert}.
\end{align}
For $\delta t < \frac{1}{8n C_L\sqrt{3\Vert D^{-1} \Vert \Vert D \Vert}}$, the stated bound for the constant of the error of this approximation now follows from the fact that $(\delta t)^2 \Vert D''B \Vert < \frac{1}{2}$.\\
For given $d=u\in \mathbb{R}^Y$, the information theoretical error that comes with the approximation of $\phi''$ by $\tilde{\phi}''$ is in $\mathcal{O}((\delta t)^2)$ because as we will see in \eqref{cross_entropy}:
\begin{align}
D(\mathcal{P}_{\phi'' \vert d}(.,u) \Vert \mathcal{P}_{\tilde{\phi}'' \vert d}(.,u) ) &= \frac{1}{2} \mathrm{Tr} \left[ D''D^{*-1} - \mathbbm{1}_{\mathbb{R}^{n}} - \log(D''D^{*-1})\right]\\
&=\frac{(\delta t)^2}{2} \mathrm{Tr} \left[ D''B \right] - \frac{1}{2}\mathrm{Tr} \left[ \log\left(\mathbbm{1}_{\mathbb{R}^{n}} + (\delta t)^2D''B\right)\right]\\
&= \frac{(\delta t)^2}{2} \mathrm{Tr} \left[ D''B \right] - \frac{1}{2}\sum\limits_{k=1}^{n} \log\left(1 + (\delta t)^2\left(D''B\right)_{kk} \right)\\
&= \frac{(\delta t)^2}{2} \mathrm{Tr} \left[ D''B \right] - \frac{(\delta t)^2}{2}\sum\limits_{k=1}^{n} \frac{1}{1+\xi_{k}}\cdot \left(D''B\right)_{kk},
\end{align}
with $0< \xi_{k} < (\delta t)^2 \left(D''B\right)_{kk}$ for $1\leq k \leq n$.\\
The last equality follows by a Taylor expansion of the logarithm in the summands around one \cite[cf.][Theorem III.5.5]{werner}.\\
So for $\delta t < \frac{1}{4\cdot n\cdot C_L}$, one has the estimate
\begin{align}
D(\mathcal{P}_{\phi'' \vert d}(.,u) \Vert \mathcal{P}_{\tilde{\phi}'' \vert d}(.,u) ) \leq (\delta t)^2 \mathrm{Tr} \left[ D''B \right] \leq n\cdot (\delta t)^2 \cdot \Vert D'' B \Vert \leq 48\cdot n^3 \cdot C_L^2 \cdot \Vert D^{-1} \Vert \cdot \Vert D \Vert
\end{align}
as stated.
\end{proof}
In our case, the agreement of $m''(u)$ and $m^{*}(u)$ and the $\mathcal{O}((\delta t)^{2})$ bound of the deviation of $D^{*-1}$ from $D''^{-1}$ will be enough. Because we will see in Section \ref{data_update}, that for the computation of the simulated data $d'$, we only need the information $m^{*}(d)$ and $D^{*-1}$ about the approximated evolved posterior $\mathcal{P}_{\tilde{\phi}'' \vert d}$.

\section{Prior update}
The statistics of our field at time $t'$ might for various reasons be different from the ones at the initial time. We should thus allow for a changed prior.
\begin{assumption}
$\mathbb{P}_{\phi'}$ has the $\lambda_{\mathbb{R}^{n}}$-derivative
\begin{align}
\mathcal{P}_{\phi'}(s) = \mathcal{G}(s - \psi', \Phi'),
\end{align}
with a positive definite covariance matrix $\Phi'\in \mathbb{R}^{n\times n}$ and a mean $\psi' \in \mathbb{R}^{n}$.
\end{assumption}
Here, we again use \eqref{least_information_entropy} to justify the distribution to be Gaussian. This is because we assume $\phi'$ to be a random vector with a density w.r.t. the Lebesgue measure which contains the least amount of information aside from a given mean $\psi'$ and covariance $\Phi'$. Therefore the evolved field maximizes the entropy $h(\phi').$ This maximum is exactly reached for a random vector $\phi'$ with Gaussian distribution, which has the density $\mathcal{P}_{\phi'}(s) = \mathcal{G}(s - \psi', \Phi')$.

\section{Data update}
\label{data_update}
In this section, we aim to find a concrete relation between the new data $d'$ and the old data $d$. 

For the later time, according to \eqref{datasignalrelation}, \eqref{distributionnoise} and \eqref{independence} in Assumption \ref{Wiener_filter_setting}, we suppose the random vector $d'$ to be related to the evolved field $\phi'$ by
\begin{align}
d' = R'\phi' + n',
\end{align}
with the response operator $R'$ from Definition \ref{simplifying_notation} and a still Gaussian distributed, zero centered noise $n'$, albeit maybe with a changed covariance matrix $N'$.

\begin{theorem}\label{posterior_evolved_field}
The new posterior $\mathcal{P}_{\phi' \vert d'}$ is given by
\begin{align}
\mathcal{P}_{\phi' \vert d'}(s,u) &= \mathcal{G}(s-m'(u),D')\ \lambda_{\mathbb{R}^{n}}\otimes \mathbb{P}_{d'}\text{-a.e in } \mathbb{R}^{n}\times \mathbb{R}^{Y'} \text{ with}\\
\text{uncertainty variance }D' &= \left( \Phi'^{-1} + R'^{T}N'^{-1}R'  \right)^{-1},\\
\text{mean }m'(u) &= \psi' + W'(u-R'\psi') = D'(R'^{T}N'^{-1}u + \Phi'^{-1}\psi') \text{ and}\\
\text{Wiener filter } W' &= D'R'^{T}N'^{-1} = \Phi' R'^{T}(R'\Phi' R'^{T} + N')^{-1}.
\end{align}
\end{theorem} 

\begin{proof}
As the proof of Theorem \ref{thm:posterior}.
\end{proof}

To simulate the evolved data by means of the measured data, we try to find a substitution for $d'$ in terms of $d$, i.e. $d' = g(d)$ for a measurable function $g: \mathbb{R}^{Y} \rightarrow \mathbb{R}^{Y'}$, such that the substituted posterior,
\begin{align}
\mathcal{P}_{\phi' \vert d'}(s,g(u)) = \mathcal{G}(s-(m'\circ g)(u),D')\ \lambda_{\mathbb{R}^{n}}\otimes \mathbb{P}_{d}\text{-a.e in } \mathbb{R}^{n}\times \mathbb{R}^{Y},
\end{align}
matches our approximation of the evolved posterior $\mathcal{P}_{\tilde{\phi}'' \vert d}$ as good as possible.
Note as a reason for our aim, that by substituting $d'=g(d)$, with (subst.) \cite[Lemma 1.22]{kall}, we have the relation
\begin{align}
\mathbb{E}\left[ f(\phi') \right] &\overset{\mathrm{Thm.}\ref{fubini},\eqref{bayes}}{=} \int\limits_{\mathbb{R}^{Y}} \int\limits_{\mathbb{R}^{n}} \mathcal{P}_{\phi' \vert d'}(s,u') \cdot f(s) ds \mathbb{P}_{d'}(du')\\
&\overset{\mathrm{subst.}}{=} \int\limits_{\mathbb{R}^{Y}} \int\limits_{\mathbb{R}^{n}} \mathcal{P}_{\phi' \vert d'}(s,g(u)) \cdot f(s) \mathbb{P}_{d}(du)
\end{align}
for every $f \in C_{b}(\mathbb{R}^{n})$ and that we want to match this expectation value as good as possible to $\mathbb{E}\left[ f\left(\tilde{\phi}''\right) \right]$.
$\mathcal{P}_{\phi' \vert d'}(.,u')$ and $\mathcal{P}_{\tilde{\phi}'' \vert d}(.,u)$ are both a.s. Gaussian, and therefore determined by their mean and covariance matrix as we will see in \eqref{characteristic_gauss}.
But unfortunately, in our case only the mean depends on the data. So even if we could find a measurable function $g$ such that $m^{*}(d) = (m'\circ g)(d) = m'(d')$, in general we still do not have $D^{*}=D'$, and therefore have to look for another possibility to determine $g$ by a reasonable approximation.
In general, we can not even find a substitution which satisfies $m^{*}(d) = (m'\circ g)(d) = m'(d')$, since $m'(d')=\psi' + W'(d'-R'\psi')$ and $m^{*}(d)=(\mathbbm{1}_{\mathbb{R}^{n}} + \delta t\cdot L)[ \psi + W(d-R\psi) ] +\delta t \cdot c$, where the Wiener filter $W'$ and the matrix $\mathbbm{1}_{\mathbb{R}^{n}} + \delta t\cdot L$ operate on different vector spaces \cite[cf.][p.9]{enss}.
As in \cite{enss}, entropic matching is used to overcome this problem, meaning that for fixed $u$, $g(u) = u'$, and the random vector
\begin{align}
\phi'\vert u'\label{cond_rand_vec}
\end{align}
which has the density $\mathcal{P}_{\phi' \vert d'}(.,u')$, we want to minimize the relative entropy
\begin{align}
D(\mathcal{P}_{\phi' \vert d'}(.,u') \Vert \mathcal{P}_{\tilde{\phi}'' \vert d}(.,u) ) &= \mathbb{E}_{\phi' \vert u'} \left[ \log\left( \frac{\mathcal{P}_{\phi' \vert d'}(., u')}{\mathcal{P}_{\tilde{\phi}'' \vert d}(., u )} \right) \right]\\
&=\int\limits_{\mathbb{R}^{n}} \mathcal{P}_{\phi' \vert d'}(s,u') \log\left( \frac{\mathcal{P}_{\phi' \vert d'}(s,u')}{\mathcal{P}_{\tilde{\phi}'' \vert d}(s,u)} \right) ds
\end{align}
with respect to $u'$.\\
\comment{
Some background on the concept of relative entropy and entropic matching is given in Chapter \ref{chap:Maximum_Entropy_Principle}.
}\\

Because both of the densities in this computation are Gaussian, the entropy simplifies a lot as we will see. 
To get to these simplifications, we need some properties of the expectation value of a Gaussian distribution.
First of all, we introduce a tool which allows us to compare probability measures \cite[cf.][chapter 4.]{kall}.
\begin{definition}[Characteristic function]
For a random vector $\xi$ in $\mathbb{R}^{n}$ with distribution $\mathbb{P}_{\xi}$, we define the characteristic function $\hat{\mathbb{P}}_{\xi}$ for $u\in \mathbb{R}^{n}$ by
\begin{align}
\hat{\mathbb{P}}_{\xi}(u) := \int\limits_{\mathbb{R}^{n}} e^{iu^{T}s} \mathbb{P}_{\xi}(ds) = \mathbb{E}\left[ e^{iu^{T}\xi} \right]. \label{characteristic_function}
\end{align}
\end{definition}
By \cite[Theorem 4.3.]{kall}, we get for two probability measures $\mu$ and $\nu$ on $\mathbb{R}^{n}$ the equivalence
\begin{align}
\mu &= \nu \text{ in the weak sense, i.e. } \forall f\in C_{b}(\mathbb{R}^{n}):\ \mu f = \nu f\\
&\Leftrightarrow \hat{\mu}(u) = \hat{\nu}(u)\ \forall u\in \mathbb{R}^{n}.\label{measure_equality}
\end{align}

\begin{lemma}\label{lemma_expectation_values}
If $\xi$ has a Gaussian distribution with density $\mathcal{G}(s - m,\Phi)$, then
\begin{align}
\hat{\mathbb{P}}_{\xi}(u) = \exp\left( im^{T}u -\frac{1}{2} u^{T}\Phi u \right).\label{characteristic_gauss}
\end{align}
Furthermore, one has
\[\mathbb{E}_{\xi}\left[ 1 \right] = 1,
\]
\[
\mathbb{E}\left[ \xi_{k} \right] =  m_{k},
\]
\begin{align}
\mathbb{E}\left[ \xi_{k}^{2} \right] = \Phi_{kk} + m_{k}^{2},
\end{align}
and if $m=0$,
\[
\mathbb{E}\left[ \xi_{k}\xi_{l} \right]= \Phi_{kl}.
\]
\end{lemma}
\begin{proof}
\eqref{characteristic_gauss} follows by computation:
\begin{align}
\hat{\mathbb{P}}_{\xi}(u) &= \int\limits_{\mathbb{R}^{n}} \frac{1}{\vert 2\pi \Phi \vert^{1/2}} \cdot \exp\left(iu^{T}s - \frac{1}{2} (s - m)^{T}\Phi^{-1}(s-m) \right)ds\\
&\overset{*}{=} \int\limits_{\mathbb{R}^{n}} \frac{1}{\vert 2\pi \Phi \vert^{1/2}} \cdot \exp\left(-\dfrac{1}{2}(s - m - i\Phi u)^{T}\Phi^{-1}(s-m-i\Phi u) \right)ds \cdot \exp\left( im^{T}u - \frac{1}{2} u^{T}\Phi u \right)\\
&= \exp\left( im^{T}u -\frac{1}{2} u^{T}\Phi u \right).
\end{align}
$*$ holds because
\begin{align}
iu^{T}s - \frac{1}{2} (s - m)^{T}\Phi^{-1}(s-m) &\overset{\Phi^{-1} = (\Phi^{-1})^{T}}{=} -\frac{1}{2} \left[ -2(i\Phi u)^{T}\Phi^{-1}s + (s - m)^{T}\Phi^{-1}(s-m) \right]\\
&= -\frac{1}{2} \left[ (s - m - i\Phi u)^{T}\Phi^{-1}(s-m-i\Phi u) - i^{2}u^{T}\Phi u - 2im^{T}u \right]\\
&= -\frac{1}{2} \left[ (s - m - i\Phi u)^{T}\Phi^{-1}(s-m-i\Phi u)\right] - \frac{1}{2}u^{T}\Phi u + im^{T}u.
\end{align}
The expression under the integral after $*$ is a probability density, so that the integral simplifies to one.\\
We continue as in the proof of \cite[Lemma 4.10]{kall}, extending the statement there from random variables to random vectors.  If $\xi$ has a Gaussian distribution, then it is $\mathbb{E}\left[\vert (\xi_{j})^{m}(\xi_{l})^{k-m}\vert\right]<\infty$ for $0\leq k \leq 2$, $1\leq j,l \leq n$ and $0\leq m \leq k$.\\
As in the proof of the lemma, by dominated convergence, with $\vert e^{iu_{j}} -1 \vert \leq u_{j}$ (or the same with $l$ respectively), for $u_{j}, u_{l}\in \mathbb{R}$ and $1\leq j,l \leq n$, one gets recursively
\[
\dfrac{\partial^{k}}{\partial u_{j}^{m}\partial u_{l}^{k-m}}\hat{\mathbb{P}}_{\xi}(u) = \mathbb{E}\left[(i\xi_{j})^{m}(i\xi_{l})^{k-m}\exp(iu^{T}\xi)\right] \text{ for } 0\leq k \leq 2,\ 0\leq m \leq k,
\]
where the existence of the right side follows by $\vert (i\xi_{j})^{m}(i\xi_{l})^{k-m}\exp(iu^{T}\xi) \vert \leq \vert (\xi_{j})^{m}(\xi_{l})^{k-m}\vert$.
From this identity, with $u=0$, it follows:
\[\mathbb{E}_{\xi}\left[ 1 \right] = 1,
\]
\[
\mathbb{E}\left[ \xi_{k} \right] = \left. -i\cdot \partial_{u_{k}}\hat{\mathbb{P}}_{\xi}(u) \right\vert_{u=0} \overset{\eqref{characteristic_gauss}}{=} m_{k},
\]
\begin{align}
\mathbb{E}\left[ \xi_{k}^{2} \right] &= \left. - \frac{\partial^{2}}{\partial u_{k}^{2}}\hat{\mathbb{P}}_{\xi}(u) \right\vert_{u=0} \overset{\eqref{characteristic_gauss}}{=} \left. - \partial_{u_{k}}\left[ \left( im_{k} - \sum\limits_{j=1}^{n} \Phi_{kj}u_{j}\right)\cdot \exp\left(im^{T}u - \frac{1}{2}u^{T}\Phi u\right)\right] \right\vert_{u=0} \\
&= \Phi_{kk} + m_{k}^{2},
\end{align}
as well as for $m=0$,
\[
\mathbb{E}\left[ \xi_{k}\xi_{l} \right]= \Phi_{kl}.
\]
\end{proof}

Now we can get a more simple expression for $D(\mathcal{P}_{\phi' \vert d'}(.,u') \Vert \mathcal{P}_{\tilde{\phi}'' \vert d}(.,u) )$.
\begin{corollary}\label{cross_entropy}
The relative entropy $D(\mathcal{P}_{\phi' \vert d'}(.,u') \Vert \mathcal{P}_{\tilde{\phi}'' \vert d}(.,u) )$ simplifies to
\begin{align}
&D(\mathcal{P}_{\phi' \vert d'}(.,u') \Vert \mathcal{P}_{\tilde{\phi}'' \vert d}(.,u) ) = \frac{1}{2} \mathrm{Tr} \left[ ( \delta m(u,u') \delta m(u,u')^{T} + D') D^{*-1} - \mathbbm{1}_{\mathbb{R}^{n}} - \log(D'D^{*-1})\right],
\end{align}
where $\delta m(u,u') = m'(u')-m^{*}(u)$.\\
Compare this also to \cite[p.10]{enss}.
\end{corollary}

\begin{proof}
Theorem \ref{posterior_evolved_field} together with Lemma \ref{lemma_expectation_values} lead to the following expectation values for $\phi' \vert u'$ from \eqref{cond_rand_vec}:
\begin{align}
\mathbb{E}_{\phi' \vert u'}\left[ 1 \right] &= 1,\\
\mathbb{E}[\phi' \vert u'] &= m'(u')
\end{align}
and
\begin{align}
\mathbb{E}_{\phi' \vert u'}\left[(\phi' \vert u')^{T}(\phi' \vert u')\right] &= \sum\limits_{k=1}^{n} \mathbb{E}\left[(\phi' \vert u')_{k}^{2}\right]\\
&= \sum\limits_{k=1}^{n}\left[ D'_{kk} + (m'_{k}(u'))^{2}\right]= \mathrm{Tr}(D') + m'(u')^{T}m'(u').
\end{align}
By the transformation formula for Lebesgue integrable functions, it also is
\begin{align}
\dfrac{1}{2}\mathbb{E}\left[(\phi' \vert u' - m'(u'))^{T}D'^{-1}(\phi' \vert u' - m'(u'))\right] = \frac{n}{2} = \frac{1}{2}\mathrm{Tr}(\mathbbm{1}_{\mathbb{R}^n}).
\end{align}
In addition, one has cyclicity of the trace operator $\mathrm{Tr}$, i.e.
\[
\mathrm{Tr}(ABC)=\mathrm{Tr}(BCA),
\]
and the equality
\[
\log \det(A) = \mathrm{Tr}(\log(A))
\]
for a matrix $A>0$ by \eqref{trace_log} in Lemma \ref{lemma:func_calc}.

Thus, the relative entropy in our case simplifies to\\
\begin{align}
&D(\mathcal{P}_{\phi' \vert d'}(.,u') \Vert \mathcal{P}_{\tilde{\phi}'' \vert d}(.,u) ) = \mathbb{E}_{\phi' \vert u'} \left[ \log\left( \frac{\mathcal{P}_{\phi' \vert d'}(., u')}{\mathcal{P}_{\tilde{\phi}'' \vert d}(., u )} \right) \right]\\
&= \mathbb{E}\left[ \frac{1}{2} \lbrace -(\phi' \vert u' - m'(u'))^{T}D'^{-1}(\phi' \vert u' - m'(u')) + (\phi' \vert u' - m^{*}(u))^{T}D^{*-1}(\phi' \vert u' - m^{*}(u)) \rbrace \right]\\
 &- \log \left\vert D'D^{*-1} \right\vert^{1/2} \\
&= \frac{1}{2} \mathrm{Tr} \left[ ( \delta m(u,u') \delta m(u,u')^{T} + D') D^{*-1} - \mathbbm{1}_{\mathbb{R}^{n}} - \log(D'D^{*-1})\right],
\end{align}
where $\delta m(u,u') = m'(u')-m^{*}(u)$.\\
Note that we used the well definedness of the square root of a positive operator by functional calculus as defined in Definition \ref{func_calc} within the computation.
\end{proof}
With help of the following theorem, we will be able to define the substitution function for the update of our data.
\begin{theorem}\label{thm:minimization_relative_entropy}
Let $u\in \mathbb{R}^Y$ be given.\\
If $(W'^{T} D^{*-1}W')>0$, then the relative entropy $u'\mapsto D(\mathcal{P}_{\phi' \vert d'}(.,u') \Vert \mathcal{P}_{\tilde{\phi}'' \vert d}(.,u) )$, with $u'\in \mathbb{R}^{Y'}$, is minimized for
\begin{align}
u' = (W'^{T}D^{*-1}W')^{-1}W'^{T}D^{*-1}(m^{*}(u) - \psi') + R'\psi'.
\end{align}
Otherwise there is some $u_{0}\in \mathbb{R}^{Y'}\backslash\left\lbrace 0 \right\rbrace$ with $W'^{T}D^{*-1}W'u_{0} = 0$.\\
If also $(D'\Phi'^{-1}\psi' - m^{*}(u))^{T}D^{*-1}W'=0 \in \mathbb{R}^{Y'}$, then $0\in \mathbb{R}^{Y'}$ is a minimizer for the relative entropy.\\
If not, then $(D'\Phi'^{-1}\psi' - m^{*}(u))^{T}D^{*-1}W' \in V^{\perp}_{0}$, where $V^{\perp}_{0}$ is the orthogonal complement of the nullspace $V_{0}$ of $W'^{T}D^{*-1}W'$. Then the vector which minimizes the relative entropy within $V^{\perp}_{0}$ is also a minimizer in $\mathbb{R}^{Y'}$.\\
This vector is also minimal in vector norm.\\
For the projection matrix $P_{V^{\perp}_{0}}$ of $V^{\perp}_{0}$ into $\mathbb{R}^{\mathrm{dim}(V^{\perp}_{0})}$  and the matrix $(P_{V^{\perp}_{0}})^{-1}$ which imbeds $\mathbb{R}^{\mathrm{dim}(V^{\perp}_{0})}$ back into $V^{\perp}_{0} \subset \mathbb{R}^{Y'}$, this vector is then given by
\begin{align}
u' = (P_{V^{\perp}_{0}})^{-1}\left(P_{V^{\perp}_{0}}(W'^{T} D^{*-1}W')(P_{V^{\perp}_{0}})^{-1}\right)^{-1} P_{V^{\perp}_{0}}W'^{T}D^{*-1} (m^{*}(u) -D'\Phi'^{-1}\psi').
\end{align}
\end{theorem}

\begin{proof}
By Corollary \ref{cross_entropy}, the relative entropy as a function of $u'$ reads as
\begin{align}
&\frac{1}{2} \mathrm{Tr} \left[ ( \delta m(u,u') \delta m(u,u')^{T} + D') D^{*-1} - \mathbbm{1}_{\mathbb{R}^{n}} - \log(D'D^{*-1})\right]\\
&= \frac{1}{2}( (m'(u')-m^{*}(u))^{T} D^{*-1} (m'(u')-m^{*}(u)) + c_{1}\\
&\overset{\mathrm{Thm.}\, \ref{posterior_evolved_field}}{=} \frac{1}{2}( (W'u' + D'\Phi'^{-1}\psi' - m^{*}(u))^{T} D^{*-1} (W'u' + D'\Phi'^{-1}\psi' - m^{*}(u)) + c_{1}\\
&=  \frac{1}{2} u'^{T}(W'^{T} D^{*-1}W')u' + (D'\Phi'^{-1}\psi' - m^{*}(u))^{T}D^{*-1}W'u' + c_{2}\label{update_data_entropy}
\end{align}
for $c_{1},c_{2}\in \mathbb{R}$ independent of $u'$.
So the entropy is a convex function of $u'$. As a general property of the relative entropy, from \eqref{positive_entropy}, we know $D(\mathcal{P}_{\phi' \vert d'}(.,u') \Vert \mathcal{P}_{\tilde{\phi}'' \vert d}(.,u) ) \geq 0$. 
Now we have to distinguish two cases.
\begin{enumerate}
\item If $(W'^{T} D^{*-1}W')>0$, then the Hessian matrix of the entropy, $W'^{T} D^{*-1}W'$, is positive definite, so that the root of the derivative of the entropy is a minimum by \cite[\S 7, Theorem 4]{forster2}.
We then have 
\begin{align}
0 &= \partial_{u'} D(\mathcal{P}_{\phi' \vert d'}(.,u') \Vert \mathcal{P}_{\tilde{\phi}'' \vert d}(.,u) )\\
&= W'^{T}D^{*-1}W'u' + W'^{T}D^{*-1}(D'\Phi'^{-1}\psi' - m^{*}(u))\\
&= W'^{T}D^{*-1}[W'(u' -R'\psi') + \psi' - m^{*}(u))]\\
&\Longleftrightarrow W'^{T}D^{*-1}W'u' = (W'^{T}D^{*-1}W')R'\psi' + W'^{T}D^{*-1}(m^{*}(u) - \psi')\\
&\Longleftrightarrow u' = (W'^{T}D^{*-1}W')^{-1}W'^{T}D^{*-1}(m^{*}(u) - \psi') + R'\psi'.
\end{align}
\item Otherwise there is some $u_{0}\in \mathbb{R}^{Y'}\backslash\left\lbrace 0 \right\rbrace$ with $W'^{T}D^{*-1}W'u_{0} = 0$.\\
Then, since $ D(\mathcal{P}_{\phi' \vert d'}(.,u') \Vert \mathcal{P}_{\tilde{\phi}'' \vert d}(.,u) )\geq 0$, one has 
\[(D'\Phi'^{-1}\psi' - m^{*}(u))^{T}D^{*-1}W'u_{0} = 0,\]
i.e. $(D'\Phi'^{-1}\psi' - m^{*}(u))^{T}D^{*-1}W' \in V^{\perp}_{0}$.\\
If this was not true, then one could find some $u_{1}= \alpha\cdot u_{0}$, with $\alpha\in \mathbb{R}$, such that the entropy at this point gets negative, which is a contradiction.\\
If also $(D'\Phi'^{-1}\psi' - m^{*}(u))^{T}D^{*-1}W'=0 \in \mathbb{R}^{Y'}$, then because $W'^{T}D^{*-1}W' \geq 0$, the minimum of the entropy is $c_{2}$ and it is for example attained for $0\in \mathbb{R}^{Y'}$.\\
In this case one can set $u' = 0$ and note that there is not enough information within the measured data vector $u$ to make any good prediction with our model.\\
Otherwise, if $(D'\Phi'^{-1}\psi' - m^{*}(u))^{T}D^{*-1}W' \in \mathbb{R}^{Y'}\backslash\left\lbrace 0 \right\rbrace$, $u_{0}$ can not be a minimizing vector, since then, by \cite[\S 7, Theorem 3]{forster2}, we would have a contradiction to
\begin{align}
0 &= \partial_{u'} D(\mathcal{P}_{\phi' \vert d'}(.,u') \vert \mathcal{P}_{\tilde{\phi}'' \vert d}(.,u) )\vert_{u'=u_{0}}\\
&= W'^{T}D^{*-1}W'u_{0} + W'^{T}D^{*-1}(D'\Phi'^{-1}\psi' - m^{*}(u))\\
&= W'^{T}D^{*-1}(D'\Phi'^{-1}\psi' - m^{*}(u)).
\end{align}
Then one can restrict the search for a minimizing vector of the relative entropy to $V^{\perp}_{0}$.\\ For each vector $u'= u_{0}+u_{1}$, with $u_{0}\in V_{0}$ and $u_{1}\in V^{\perp}_{0}$, the entropy of $u_{1}$ is equal to the one of $u'$. If $u_{\min,1}$ minimizes the entropy within $V^{\perp}_{0}$, then $u_{0}+u_{\min,1}$ minimizes it for every $u_{0}\in V_{0}$.\\
Due to
\[
\Vert u' \Vert_{2} = \Vert u_{0} \Vert_{2} + \Vert u_{1} \Vert_{2},
\]
the vector $u_{\min}= u_{0}+u_{\min,1}$ is therefore a minimum for the relative entropy which is also minimal in vector norm, if and only if $u_{0}=0$, i.e. if $u_{\min}=u_{\min,1}$.\\
We define $P_{V^{\perp}_{0}}$ to be the projection matrix of $V^{\perp}_{0}$ into $\mathbb{R}^{\mathrm{dim}(V^{\perp}_{0})}$  and $(P_{V^{\perp}_{0}})^{-1}$ the matrix which imbeds $\mathbb{R}^{\mathrm{dim}(V^{\perp}_{0})}$ back into $V^{\perp}_{0} \subset \mathbb{R}^{Y'}$.\\
\comment{
The projection matrix $P_{V^{\perp}_{0}}$ is given by
\begin{align}
P_{V^{\perp}_{0}}x&:= \left(x^{(1)T}x,...,x^{\left(\mathrm{dim}(V^{\perp}_{0})\right)T}x\right)^T \in \mathbb{R}^{\mathrm{dim}(V^{\perp}_{0})}
\end{align}
for $x\in \mathbb{R}^{Y'}$, where $\lbrace x^{(l)} \rbrace_{l=1}^{\mathrm{dim}(V^{\perp}_{0})}\subset \mathbb{R}^{Y'}$ is an orthonormal basis of $V^{\perp}_{0}$.\\
The projection of $\mathbb{R}^{\mathrm{dim}(V^{\perp}_{0})}$ back into $V^{\perp}_{0} \subset \mathbb{R}^{Y'}$ is then given by
\begin{align}
(P_{V^{\perp}_{0}})^{-1}c := \sum\limits_{l=1}^{\mathrm{dim}(V^{\perp}_{0})} c_{l}\cdot x^{(l)}\in \mathbb{R}^{Y'}
\end{align}
for $c\in \mathbb{R}^{\mathrm{dim}\left(V^{\perp}_{0}\right)}$.
Compare this also to \eqref{projection} in our section about functional calculus for matrices.
}\\

Then, for the minimizing vector $u_{\min}$, that is also minimal in vector norm, it is
\[
u_{\min}=(P_{V^{\perp}_{0}})^{-1}(\tilde{u}_{\min}),
\]
where $\tilde{u}_{\min}$ is the solution of
\begin{align}
&\min\limits_{\tilde{u}\in \mathbb{R}^{\mathrm{dim}\left(V^{\perp}_{0}\right)}} \left[ D(\mathcal{P}_{\phi' \vert d'}(.,u') \vert \mathcal{P}_{\tilde{\phi}'' \vert d}(.,u) )\vert_{u'=(P_{V^{\perp}_{0}})^{-1}(\tilde{u})} \right]\\
&=\min\limits_{\tilde{u}\in \mathbb{R}^{\mathrm{dim}\left(V^{\perp}_{0}\right)}} \left[
\frac{1}{2} \tilde{u}^{T} P_{V^{\perp}_{0}}(W'^{T} D^{*-1}W')(P_{V^{\perp}_{0}})^{-1}\tilde{u} + (D'\Phi'^{-1}\psi' - m^{*}(u))^{T}D^{*-1}W'(P_{V^{\perp}_{0}})^{-1}\tilde{u} + c_{2} \right].
\end{align}
This minimum is attained for
\begin{align}
\tilde{u}_{\min} &= \left(P_{V^{\perp}_{0}}(W'^{T} D^{*-1}W')(P_{V^{\perp}_{0}})^{-1}\right)^{-1} P_{V^{\perp}_{0}}W'^{T}D^{*-1} (m^{*}(u) -D'\Phi'^{-1}\psi'),
\end{align}
because the Hessian matrix $P_{V^{\perp}_{0}}(W'^{T} D^{*-1}W')(P_{V^{\perp}_{0}})^{-1}$ is positive definite.\\
In this case, $u' = u_{\min}=(P_{V^{\perp}_{0}})^{-1}(\tilde{u}_{\min})$ is the expression that was stated in the theorem.
\end{enumerate}
\end{proof}
For a measured data vector $d$, we want the simulated data $d'$ to minimize the relative entropy
\[
d'\mapsto D(\mathcal{P}_{\phi' \vert d'}(.,d') \Vert \mathcal{P}_{\tilde{\phi}'' \vert d}(.,d) )
\]
with a minimal vector norm.
\begin{definition}\label{substitution_function}
We define the substitution function $g: \mathbb{R}^Y \rightarrow \mathbb{R}^{Y'}$ for the data update from $d$ to $d'$ in the following way:\\
If $W'^{T}D^{*-1}W' > 0$, then we set
\begin{align}
g(u) = (W'^{T}D^{*-1}W')^{-1}W'^{T}D^{*-1}(m^{*}(u) - \psi') + R'\psi'\label{g_first_case}
\end{align}
for every $u\in \mathbb{R}^Y$.
Otherwise we set
\begin{align}
\label{g_second_case} g(u)=
 \begin{cases} 
0 \text{, if } (D'\Phi'^{-1}\psi' - m^{*}(u))^{T}D^{*-1}W'=0 \in \mathbb{R}^{Y'}\ \text{and} \\
(P_{V^{\perp}_{0}})^{-1}\left(P_{V^{\perp}_{0}}(W'^{T} D^{*-1}W')(P_{V^{\perp}_{0}})^{-1}\right)^{-1} P_{V^{\perp}_{0}}W'^{T}D^{*-1} (m^{*}(u) -D'\Phi'^{-1}\psi')\text{ else.}
\end{cases}
\end{align}
\end{definition}
All matrix multiplications are continuous and therefore measurable. That is, also in the second case, every preimage of a measurable set under $g$ is the union of measurable sets, i.e. again measurable. So in both cases, $g$ is measurable. Theorem \ref{thm:minimization_relative_entropy} proves, that the substitution $d' = g(d)$ has the properties that we desired.\\

Now for the definition of $g$ by the minimization of the relative entropy in Corollary \ref{cross_entropy}, we need the covariance matrix $D'$, which is only implicitly given by $D'= (\Phi' + R'^{T}N'^{-1}R')^{-1}$. We might also have some degrees of freedom here, leading to a parametrized representation of $D'$. We then also minimize the entropy with respect to these parameters. One should note here, that without any restrictions, the entropy is minimized with respect to $D'$ for $D'=D^{*}$, leading to 
\begin{align}
D(\mathcal{P}_{\phi' \vert d'}(.,u') \Vert \mathcal{P}_{\tilde{\phi}'' \vert d}(.,u) ) &= \frac{1}{2} \mathrm{Tr} \left[  \delta m(u,u') \delta m(u,u')^{T} D^{*-1}\right].
\end{align}
In general, when $D'$ is somehow parametrized due to the relation $D'= (\Phi' + R'^{T}N'^{-1}R')^{-1}$, then a $D'$ which minimizes the relative entropy with respect to these parameters can attain different forms from case to case.\\
So the general recipe to update the data is:
\begin{enumerate}
\item Minimize the relative entropy in Corollary \ref{cross_entropy} with respect to the degrees of freedom of $D'$.
\item Compute the substitution function $d'=g(d)$ from Definition \ref{substitution_function}.
\item If a measurement at time $t$ led to a data vector $d$, then the update of this vector to the later time $t'$ is given by $g(d)$.
\end{enumerate}
Note that the parameters which are fixed within the first step also determine $R',N'$ and $W'$, so that one has access to these expressions in step two.
The measured data vector $d$ is a realization of the random vector which we also called $d$. The distribution of this random vector is coupled to the one of $d'$ by $d'=g(d)$. Therefore, the measured vector $d$ leads to the simulation of a realization of $d'$ which is equal to $g(d)$.

\section{The algorithm}
Now let us summarize the whole procedure of data simulation with IFD.
At first, one chooses $\epsilon_{1,N}$ from Assumption \ref{ass:discretetime} according to the demanded degree of accuracy and a sufficiently large $N$ in order to reach this precision with a discretization of the time interval $[0,T]$.
One then constructs the operator $A_{N}: \lbrace t_{0},..., t_{2^N-1}\rbrace \rightarrow \mathrm{map}(V,V)$ from Assumption \ref{ass:discretetime}.\\
Afterwards, one defines $\phi_{N}: [0,T] \rightarrow V$ as in Definition \ref{brev_phi}.\\
At this point, one has a step function in time which has still values in $V$.\\
In a next step, one chooses $\epsilon_{2,n}$ from Assumption \ref{ass:finitedim} according to the demanded accuracy level and a basis of the subspace $V_{n}$ of $V$, where $n$ is large enough to reach also this precision.\\
With this basis, one constructs the operator valued function $A_{n,N}$ from Assumption \ref{ass:finiteevolution} and defines $\phi^{(n)}_{N}: \lbrace t_{0},...,t_{2^N} \rbrace \rightarrow V_{n}$ as in definition \ref{phi_n,N}.\\
This is our signal, discretized in time and approximated by a linear combination of the basis vectors of $V_{n}$, i.e. a representation finite dimensional in space and discrete in time.\\
This signal can be identified with the coefficients of the linear combination it exits of. If, this way, one interprets $A_{n,N}$ as an operator valued function which maps from $\lbrace t_{0},...,t_{2^N-1} \rbrace$ into $\mathrm{map}(\mathbb{R}_{n},\mathbb{R}_{n})$, then the approximated signal by construction evolves according to
\[
\phi^{(n)}_{N}(t_{i+1}) = A_{n,N}(t_{i})\phi^{(n)}_{N}(t_{i})
\] 
for $0\leq i \leq 2^N-1$.\\
By Theorem \ref{signalapproximation}, we also have the estimate
\[
\left\Vert \phi_{\mathrm{real}}(t_{i}) - \phi^{(n)}_{N}(t_{i}) \right\Vert_{V}\leq C_{1}\cdot \epsilon_{1,N} + 2\cdot\epsilon_{2,n}
\]
for every $i=0,...,2^N$.\\
In a next step, one specifies the discretized response operator $R_{n,N}(0)$ from relation \eqref{datadiscrete}.
Then, the prior $\mathcal{P}_{\phi^{(n)}_{N}(0)}$  and the distribution of the noise $n(0)$ at the initial time $t_{0}=0$ have to be determined.\\
For the time steps $t_{1},...,t_{2^N-1}$, if they are not known, one needs to make restrictions on $\mathcal{P}_{\phi^{(n)}_{N}(t_{i})}$, $R_{n,N}(t_{i})$, the distribution of $n(t_{i})$, as well as on $\mathcal{P}_{\phi^{(n)}_{N}(T)}$ and the distribution of $n(T)$, which lead to parametrized representations of these sizes.\\
With the resulting Wiener filter matrices $W_{t_{i}}$, or at least with the parametrized representations of them, one updates the measured data $d_{0}$, time step for time step, with help of the corresponding substitution functions from Definition \ref{substitution_function}, until the simulated data $d(T)$ has been computed.\\
This is our desired simulation of the data for time $T$.\\

\comment{
In Section \ref{field_evolution}, we have made the further approximation of $\phi^{(n)}_{N}(t_{i+1})$ by $\phi''^{(n)}_{N}(t_{i+1})$ within a simplification of the relation
\[
\phi^{(n)}_{N}(t_{i+1}) = A_{n,N}(t_{i})\phi^{(n)}_{N}(t_{i})
\]
with $i=0,...,2^N-1$.\\
By Assumption \ref{ass:linearization}, the bound $C\cdot (\delta t)^2$ of the error of this approximation in operator norm, for $\mathbb{R}^{n}$ endowed with the euclidean norm, does not depend on the number of steps in the discretization of time.\\
Nevertheless, one should be aware of this additional approximation when choosing $N$, because it leads to a possible error of
\[
2^N\cdot C\cdot (\delta t)^2 = 2^N\cdot C \cdot \dfrac{T^2}{2^{2N}}= C \dfrac{T^2}{2^N}
\]
after the whole simulation.\\
The further simplification of $\phi''^{(n)}_{N}(t_{i+1})$ by $\tilde{\phi}''^{(n)}_{N}(t_{i+1})$ in Theorem \ref{thm:further_approximation} leads to an additional error of order $(\delta t)^2$.\\
If also the assumptions of Theorem \ref{thm:invertdynamics} hold, then also this error, which arises in every evolution step, leads to an approximation still of order $\mathcal{O}\left(\dfrac{T^2}{2^N}\right)$ close to $\phi''^{(n)}_{N}$ for all time steps $t_{i}$, with $i=0,...,2^N$.\\
Also the fact of this error should be considered when choosing $N$, though.
}

\chapter{Maximum Entropy Principle}
\label{chap:Maximum_Entropy_Principle}

\section{Differential entropy as a measure of information}
In this chapter, we introduce the concept of entropy for a random vector $\xi$  in $\mathbb{R}^{n}$ with distribution $\mathbb{P}_{\xi}=\mathcal{P}_{\xi}\cdot \lambda_{\mathbb{R}^{n}}$.\\
This entropy is used as a measure of uncertainty for the random vector \cite[cf.][p.13]{cover}.
We begin with the definition of entropy for a continuous random variable $X$ with $\mathbb{P}_{X}=\mathcal{P}_{X}\cdot \lambda_{\mathbb{R}}$. With continuous we mean, that the function $x\mapsto \mathbb{P}(X\leq x)$ is continuous.

\begin{definition}[Differential entropy]
We define the differential entropy $h(X)$ of a continuous random variable $X$ as
\begin{align}
h(X):= - \int\limits_{\lbrace x\in \mathbb{R}: \mathcal{P}_{X}(x)>0 \rbrace} \mathcal{P}_{X}(x)\log \mathcal{P}_{X}(x) dx = -\mathbb{E}[\log \mathcal{P}_{X}(X)],
\end{align}
in case the integral exists \cite[cf.][chapter 8.1]{cover}
\end{definition}

One interpretation of differential entropy as a measure of information comes from its properties.\cite[Chapter 8.2]{cover} starts with a sequence of independent random variables $X_{1},X_{2},...$, all with the same distribution $\mathbb{P}_{X}=\mathcal{P}_{X}\cdot \lambda_{\mathbb{R}}$. For $\epsilon > 0$ and $n\in \mathbb{N}$, it then defines the typical set $A_{\epsilon}^{(n)}$ with respect to $\mathcal{P}_{X}$ as
\begin{align}
A_{\epsilon}^{(n)}:= \left\lbrace (x_{1},...,x_{n})\in \lbrace x\in \mathbb{R}: \mathcal{P}_{X}(x)>0 \rbrace^{n}: \left\vert -\dfrac{1}{n} \sum\limits_{k=1}^{n}\log \mathcal{P}_{X}(x_{k}) - h(X) \right\vert \leq \epsilon \right\rbrace.
\end{align}
In \cite[Theorem 8.2.2]{cover} it is then stated, that for the volume of this set, for sufficiently large $n$, one has 
\[\mathbb{P}\left(\left\lbrace (X_{1},...,X_{n})\in A_{\epsilon}^{(n)} \right\rbrace\right)> 1-\epsilon,\]
\[\int\limits_{A_{\epsilon}^{(n)}} d(x_{1},...,x_{n})\leq 2^{n(h(X)+\epsilon)}\]
and
\[\int\limits_{A_{\epsilon}^{(n)}} d(x_{1},...,x_{n})\geq (1-\epsilon)2^{n(h(X)-\epsilon)}.\]
\cite[Theorem 8.2.3]{cover} then states, that $A_{\epsilon}^{(n)}$ is also the smallest volume set, such that \[\mathbb{P}\left(\left\lbrace (X_{1},...,X_{n})\in A_{\epsilon}^{(n)} \right\rbrace\right) \geq 1-\epsilon,\] to first order in the exponent.\\
It is thus the smallest $n$-dimensional set, that approximately contains most of the probability mass of $(X_{1},...,X_{n})$. Its volume is about $2^{nh(X)}$. The side length of a rectangle with volume $2^{nh(X)}$ is $2^{h(X)}$.
In this way, differential entropy can be regarded as the $\log_{2}$ of the side length of the set that approximately contains most of the probability mass of the random vector $(X_{1},...X_{n})$. By this, low entropy for a random variable means that its probability mass is effectively contained in a small set, whereas high entropy indicates, that it is widely dispersed \cite[cf.][p.246 f.]{cover}.\\
In this way, differential entropy can be regarded as a measure of information, if we interpret wide dispersion of a random variable to indicate that we have only little information about it.\\
"It is [thus] a measure of the amount of information required on the average to describe the random variable" \cite[p.19]{cover}.\\
This concept of differential entropy can then straight forwardly be generalized to a random vector $\xi$ with distribution $\mathbb{P}_{\xi}=\mathcal{P}_{\xi}\cdot \lambda_{\mathbb{R}^{n}}$.
If for such a $\xi$, one only knows that it has zero mean, $\mathbb{E}[\xi]=0$, and covariance $K=\mathbb{E}\left[\xi\xi^{T}\right]$, then \cite[Theorem 8.6.5]{cover} states, that the entropy $h(\xi)$ is maximized, if and only if \begin{align}
\mathcal{P}_{\xi}(s)=\mathcal{G}(s,K).\label{least_information_entropy}
\end{align}
Because by \cite[Theorem 8.6.3]{cover}, the differential entropy is translation invariant, this property can be generalized to random vectors with a mean different from zero.

\section{Relative entropy}
Having a measure for information for our random vectors, we would now like to compare them w.r.t. this measure. This is done via relative entropy.

\begin{definition}[Relative entropy]
The relative entropy between two random vectors $\xi$ and $\eta$ in $\mathbb{R}^{n}$ with distributions $\mathbb{P}_{\xi}=\mathcal{P}_{\xi}\cdot \lambda_{\mathbb{R}^{n}}$ and $\mathbb{P}_{\eta}=\mathcal{P}_{\eta}\cdot \lambda_{\mathbb{R}^{n}}$ is defined by
\begin{align}
D(\xi \Vert \eta):=\int\limits_{\mathbb{R}^{n}} \mathcal{P}_{\xi}(s) \log \dfrac{\mathcal{P}_{\xi}(s)}{\mathcal{P}_{\eta}(s)} ds.
\end{align}
In this definition, we set $0\log\dfrac{0}{0}=0$ by continuity and remark that the relative entropy is only finite, if the set $\left\lbrace s\in \mathbb{R}^n: \mathcal{P}_{\xi}(s)>0\text{ and } \mathcal{P}_{\eta}(s)=0\right\rbrace$ is a $\lambda_{\mathbb{R}^{n}}$-null set \cite[cf.][p.251]{cover}.
\end{definition}
By \cite[p.19]{cover}, the relative entropy $D(\xi \Vert \eta)$ "is a measure of distance between two distributions" and a "measure of the inefficiency of assuming the distribution is [$\mathcal{P}_{\eta}$] when the true distribution is [$\mathcal{P}_{\xi}$]".
So if one knows that $\eta$ has the distribution $\mathbb{P}_{\eta}=\mathcal{P}_{\eta}\cdot \lambda_{\mathbb{R}^{n}}$, one might be interested in finding a random vector $\xi$ with distribution $\mathbb{P}_{\xi}=\mathcal{P}_{\xi}\cdot \lambda_{\mathbb{R}^{n}}$, and maybe some restrictions to the expectation values of some functions of it, that is the closest to $\eta$. This leads to the aim to minimize $D(\xi \Vert \eta)$ over all functions $\mathcal{P}_{\xi}:\mathbb{R}^{n}\rightarrow \mathbb{R}_{+}$ with
\[
\int\limits_{\mathbb{R}^{n}} \mathcal{P}_{\xi}(s)ds = 1,\label{density_constraint}
\]
and maybe several more constraints on expectation values.
One should remark here, that by \cite[Theorem 8.6.1]{cover}, one always has 
\begin{align}
D(\xi \Vert \eta)\geq 0, \label{positive_entropy}
\end{align}
with equality if and only if $\mathcal{P}_{\xi}=\mathcal{P}_{\eta}$ $\lambda_{\mathbb{R}^{n}}$-a.e.\\
That is, if the only constraint on $\mathcal{P}_{\xi}$ is given by \eqref{density_constraint}, then minimizing $D(\xi \Vert \eta)$ with respect to $\mathcal{P}_{\xi}$ leads to $\mathcal{P}_{\xi}=\mathcal{P}_{\eta}$ $\lambda_{\mathbb{R}^{n}}$-a.e.
\chapter{Example: Klein-Gordon field}
\label{chap:Example: Klein-Gordon field}
\section{Klein-Gordon field in one dimension}
\label{sec:Klein-Gordon field in one dimension}
We take the thermally excited Klein-Gordon field with one dimension in space \cite[cf.][p.5]{enss} to illustrate the concept of IFD. We suppose the field to be periodic in space over $[0,2\pi)$, and we assume it to follow the partial differential equation
\begin{align}
&\partial_{t}^2 \varphi(x,t) = (\partial_{x}^2 - \mu^2)\varphi(x,t) \text{ for } x\in (0,2\pi) \text{ and } t\in [0,T],\\
&\varphi(.,0),\chi(.,0):=\partial_t\varphi(.,0)\in L_2((0,2\pi)),\label{kg_equation}
\end{align}
with some $\mu\in \mathbb{R}$.\\
Our response will be a decomposition of integral operators on $L_{2}((0,2\pi))$, acting on the field $\varphi$ and its derivative in time $\chi$.\\
We assume that all fields in our setting and their derivatives are contained in $L_2((0,2\pi))$ for all times. This way we assure that our response describes the measurement in a well-defined way.\\
More than that, we suppose all our fields to have a Fourier space representation $(\hat{\varphi}_{\mathrm{c}}(k,t))_{k\in \mathbb{Z}}\in l_2(\mathbb{Z})$, with
\begin{align}
\hat{\varphi}_{\mathrm{c}}(k,t) = 0\ \forall k\text{ with } \vert k \vert\geq n,
\end{align}
for some $n\in \mathbb{N}$, and that for every $k\in \mathbb{Z}$ and $t\in [0,T]$, $\hat{\chi}_{\mathrm{c}}(k,t)$ is the time derivative of $\hat{\varphi}_{\mathrm{c}}(k,t)$, where both of the latter functions are marked with the index $\mathrm{c}$ to emphasize that they have values in $\mathbb{C}$.\\
The Fourier transform in our case is given by
\[
\hat{\varphi}_{\mathrm{c}}(k,t)= \int\limits_{0}^{2\pi} e^{ikx}\varphi(x,t) dx \in \mathbb{C}.
\]
Equality \eqref{kg_equation} in Fourier space representation, with $\omega_{k}:= \sqrt{k^2+\mu^2}$, reads as
\begin{align}
&\partial_{t}^2 \hat{\varphi}_{\mathrm{c}}(k,t) = -\omega_{k}^2\hat{\varphi}_{\mathrm{c}}(k,t) \text{ for } -n < k < n \text{ and}\\
&\left(\hat{\varphi}_{\mathrm{c}}(k,0)\right)_{k\in \mathbb{Z}},\left(\partial_t\hat{\varphi}_{\mathrm{c}}(k,0)\right)_{k\in \mathbb{Z}}\in l_2(\mathbb{Z}) \text{ for } k\in \mathbb{Z} \text{ and } t\in [0,T],\label{kg_equation_fourier}
\end{align}
and in our case all coefficients larger than or equal to $n$ vanish for all times.\\
Although this equation is exactly solvable, we want to test IFD as derived in the last chapters to develop the simulation scheme for a data vector.\\
Equality \eqref{kg_equation_fourier} can be brought into the form of \eqref{pde}, if one sets $\hat{\chi}_{\mathrm{c}}:= \partial_{t}\hat{\varphi}_{\mathrm{c}}$ and $\hat{\phi}_{\mathrm{c}}:= \begin{pmatrix}
\hat{\varphi}_{\mathrm{c}}\\
\hat{\chi}_{\mathrm{c}}
\end{pmatrix}.$
One then gets
\begin{align}
&\partial_{t} \hat{\phi}_{\mathrm{c}}(k,t) = \begin{pmatrix}
0 & 1\\
-\omega_{k}^2 & 0
\end{pmatrix}\hat{\phi}_{\mathrm{c}}(k,t) \text{ for } -n < k <n \text{ and } t\in [0,T].\label{partial_derivative}
\end{align}
In this case, by \citep[Theorem IV.2.9]{werner}, we also have
\begin{align}
\frac{1}{2\pi}\sum\limits_{k=-(n-1)}^{n-1}e^{-ikx}\hat{\phi}_{\mathrm{c}}(k,t)=\begin{pmatrix}
\varphi(x,t)\\
\chi(x,t)
\end{pmatrix}=: \phi(x,t)
\end{align}
in $L_{2}((0,2\pi))\times L_{2}((0,2\pi))$ for all $t\in [0,T]$.\\
By this fact, for two of our signals, evaluated at a time $t\in [0,T]$, which are vector valued functions $\phi = \begin{pmatrix}
\phi^{(\varphi)}\\
\phi^{(\chi)}
\end{pmatrix}$ and $\psi = \begin{pmatrix}
\psi^{(\varphi)}\\
\psi^{(\chi)}
\end{pmatrix}$ in $L_{2}((0,2\pi))\times L_{2}((0,2\pi))$ then, we get equality of the scalar products
\begin{align}
\phi^{\dagger}\psi &= \int\limits_{0}^{2\pi} \left( \overline{ \phi^{(\varphi)}(x) }\psi^{(\varphi)}(x) + \overline{\phi^{(\chi)}(x)} \psi^{(\chi)}(x) \right)dx\\
&= \frac{1}{2\pi}\sum\limits_{k=-(n-1)}^{n-1} \left( \overline{ \hat{\phi}_{\mathrm{c}}^{(\varphi)}(k) }\hat{\psi}_{\mathrm{c}}^{(\varphi)}(k) + \overline{\hat{\phi}_{\mathrm{c}}^{(\chi)}(k) }\hat{\psi}_{\mathrm{c}}^{(\chi)}(k) \right)
\end{align}
in $L_{2}((0,2\pi))\times L_{2}((0,2\pi))$ and $l_{2}(\mathbb{Z})\times l_{2}(\mathbb{Z})$ respectively \cite[cf.][p.5]{enss}.

\section{Exact solution}
The Klein-Gordon equation with the restrictions from Section \ref{sec:Klein-Gordon field in one dimension} is exactly solvable. We thus compute the exact solution in order to be able to make statements about errors which arise within the approximation steps of IFD. 
\eqref{kg_equation_fourier} has solutions of the form
\begin{align}
\hat{\varphi}_{\mathrm{c}}(k,t) &= a_{k}e^{i\omega_{k}t} + b_{k}e^{-i\omega_k t},\\
\hat{\chi}_{\mathrm{c}}(k,t) &= i\omega_k \left( a_{k}e^{i\omega_{k}t} - b_{k}e^{-i\omega_k t} \right),
\end{align}
for $\vert k \vert< n$ and $a_{k}, b_{k}\in \mathbb{C}$.\\
Together with the initial conditions given by $\hat{\varphi}_{\mathrm{c}}(k,0)$ and $\hat{\chi}_{\mathrm{c}}(k,0)$, and because
\begin{align}
\hat{\varphi}_{\mathrm{c}}(k,t)=\overline{\hat{\varphi}_{\mathrm{c}}(-k,t)},\\
\hat{\chi}_{\mathrm{c}}(k,t) = \overline{\hat{\chi}_{\mathrm{c}}(-k,t)},\label{complex_conj}
\end{align}
due to our field and its derivative being real valued in position space, we get
\begin{align}
\hat{\varphi}_{\mathrm{c}}(k,t) &= a_{k}e^{i\omega_{k}t} + \overline{a_{-k}}e^{-i\omega_k t},\\
\hat{\chi}_{\mathrm{c}}(k,t) &= i\omega_k \left( a_{k}e^{i\omega_{k}t} - \overline{a_{-k}}e^{-i\omega_k t} \right)\label{exact_solution}
\end{align}
and 
\begin{align}
a_k = \frac{1}{2}\hat{\varphi}_{\mathrm{c}}(k,0) - \frac{i}{2\omega_k}\hat{\chi}_{\mathrm{c}}(k,0),
\end{align}
for $\vert k \vert< n$.
This leads to 
\begin{align}
\hat{\phi}_{\mathrm{c}}(k,t)&= \begin{pmatrix}
\hat{\varphi}_{\mathrm{c}}(k,t)\\
\hat{\chi}_{\mathrm{c}}(k,t)
\end{pmatrix} 
= \begin{pmatrix}
\cos(\omega_k t) & \frac{\sin(\omega_k t)}{\omega_k}\\
-\omega_k \sin(\omega_k t) & \cos(\omega_k t)
\end{pmatrix} 
\begin{pmatrix}
\hat{\varphi}_{\mathrm{c}}(k,0)\\
\hat{\chi}_{\mathrm{c}}(k,0)
\end{pmatrix},\label{exact_evolution}
\end{align}
with $\hat{\varphi}_{\mathrm{c}}(k,0),\hat{\chi}_{\mathrm{c}}(k,0) \in \mathbb{C}$ and $\vert k \vert< n$. 
For two times $0\leq t < s \leq T$, the exact evolution of $\hat{\phi}_{\mathrm{c}}$ is therefore satisfies
\begin{align}
\hat{\phi}_{\mathrm{c}}(s) &= A_{\mathrm{c}}(t,s)\hat{\phi}_{\mathrm{c}}(t),\label{exact_evolution2}
\end{align}
where $A_{\mathrm{c}}(t,s)$ is a matrix in $\left( \mathbb{R}^2 \right)^{2n-1 \times 2n-1}$ with
\begin{align}
&(A_{\mathrm{c}}(t,s))_{kq}= \begin{pmatrix}
\cos(\omega_{k}s) & \frac{\sin(\omega_{k}s)}{\omega_{k}}\\
-\omega_{k}\sin(\omega_{k}s) & \cos(\omega_{k}s)
\end{pmatrix} \begin{pmatrix}
\cos(\omega_{k}t) & -\frac{\sin(\omega_{k}t)}{\omega_{k}}\\
\omega_{k}\sin(\omega_{k}t) & \cos(\omega_{k}t)
\end{pmatrix}\delta_{kq}\\
&=\begin{pmatrix}
\cos(\omega_{k}s)\cos(\omega_{k}t)+\sin(\omega_{k}s)\sin(\omega_{k}t) & \frac{1}{\omega_{k}}\left(\sin(\omega_{k}s)\cos(\omega_{k}t) - \cos(\omega_{k}s)\sin(\omega_{k}t)\right) \\
\omega_{k}\left( \cos(\omega_{k}s)\sin(\omega_{k}t) - \sin(\omega_{k}s)\cos(\omega_{k}t) \right) & \cos(\omega_{k}s)\cos(\omega_{k}t) + \sin(\omega_{k}s)\sin(\omega_{k}t)
\end{pmatrix}\delta_{kq}\\
&=\begin{pmatrix}
\cos(\omega_k (s-t)) & \frac{1}{\omega_k}\sin(\omega_k (s-t))\\
-\omega_k \sin(\omega_k (s-t)) & \cos(\omega_k (s-t))
\end{pmatrix}\delta_{kq},
\end{align}
for $-n < k,q <n$.\\
Note also, that because we suppose our Klein-Gordon fields to be real valued, it is $\hat{\phi}_{\mathrm{c}}(0,t)\in \mathbb{R}^2$ for all $t\in [0,T]$.
\section{Prior knowledge}
By \citep[12.6.\, Problems,\, 2]{evans} and with the notation of \citep{enss}, the energy of a Klein-Gordon field at time $t\in [0,T]$ in our setting is given by the energy Hamiltonian
\begin{align}
\mathcal{H}(\phi)(t) :&= \dfrac{1}{2} \int\limits_{0}^{2\pi} \left[ \chi^2(x,t) + (\partial_{x}\varphi(x,t))^2 + \mu^2\varphi^2(x,t) \right] dx\\
&= \frac{1}{4\pi} \sum\limits_{k=-(n-1)}^{n-1}\left[ \left\vert \hat{\chi}_{\mathrm{c}}(k,t) \right\vert^2 +(k^2 +  \mu^2) \left\vert \hat{\varphi}_{\mathrm{c}}(k,t) \right\vert^2 \right]\\
&\overset{\eqref{exact_solution}}{=} \frac{1}{4\pi} \sum\limits_{k=-(n-1)}^{n-1}2\omega_k^2\cdot \left[ \left\vert a_k \right\vert^2 + \left\vert a_{-k} \right\vert^2\right]\\
&= \mathcal{H}(\phi)(0).
\end{align}
So the energy of a field is determined by its energy at the initial time and then conserved over the whole interval $[0,T]$.
It is further
\begin{align}
\mathcal{H}(\phi)(0) &= \frac{1}{4\pi} \sum\limits_{k=-(n-1)}^{n-1}\left[ \left\vert \hat{\chi}_{\mathrm{c}}(k,0) \right\vert^2 +(k^2 +  \mu^2) \left\vert \hat{\varphi}_{\mathrm{c}}(k,0) \right\vert^2 \right]\\
&\overset{\eqref{complex_conj}}{=} \dfrac{1}{4\pi}\left\vert \hat{\chi}_{\mathrm{c}}(0,0) \right\vert^2 + \dfrac{1}{4\pi}\mu^2 \left\vert \hat{\varphi}_{\mathrm{c}}(0,0) \right\vert^2 + \frac{1}{2\pi} \sum\limits_{k=1}^{n-1}\left[ \left\vert \hat{\chi}_{\mathrm{c}}(k,0) \right\vert^2 +\omega_k^2 \left\vert \hat{\varphi}_{\mathrm{c}}(k,0) \right\vert^2 \right].
\end{align}
Because $\hat{\phi}_{\mathrm{c}}(t)$ is fully determined by its Fourier coefficients with $k\geq 0$, we identify the whole vector with $\left(\hat{\phi}_{\mathrm{c}}(k,t)\right)_{k=0}^{n-1}\in \mathbb{C}^{2n}$, and sometimes switch between the full representation of $\hat{\phi}_{\mathrm{c}}(t)$ and the one in $\mathbb{C}^{2n}$, which gets clear from whether negative Fourier modes are used or not.\\
As in \citep[p.6]{enss}, we assume "that the field was initially in contact and equilibrium with a thermal reservoir at temperature $\beta^{-1}$ and became decoupled from it at" time $t_{0}=0$. In \citep{enss} it is stated that in this case, the prior of the signal at time zero is an exponential function where the exponent is given by the negative energy of the argument $s_{\mathrm{c}}=\begin{pmatrix}
s_{\mathrm{c}}^{(\varphi)} \\
s_{\mathrm{c}}^{(\chi)}
\end{pmatrix}\in \mathbb{C}^{2n}$, multiplied with $\beta$, namely 
\begin{align}
\mathcal{P}_{\hat{\phi}_{\mathrm{c}}(0)}(s_{\mathrm{c}}) \propto \exp\left( -\beta \cdot \mathcal{H}(\phi)(0)\left\vert_{\hat{\phi}_{\mathrm{c}}(0)=s_{\mathrm{c}}}\right. \right).
\end{align}
We also have to include our knowledge that $\hat{\phi}_{\mathrm{c}}(0,t)\in \mathbb{R}^2$. So the negative exponent reads as
\begin{align}
\beta \cdot \mathcal{H}(\phi)(0)\left\vert_{\hat{\phi}_{\mathrm{c}}(0)=s_{\mathrm{c}}}\right. &= \beta\cdot \left( \dfrac{1}{4\pi}\left\vert \left(s^{(\chi)}_{\mathrm{c}}\right)_0 \right\vert^2 + \dfrac{1}{4\pi}\mu^2 \left\vert \left(s^{(\varphi)}_{\mathrm{c}}\right)_0 \right\vert^2 + \frac{1}{2\pi} \sum\limits_{k=1}^{n-1}\left[ \left\vert \left(s^{(\chi)}_{\mathrm{c}}\right)_k \right\vert^2 +\omega_k^2 \left\vert \left(s^{(\varphi)}_{\mathrm{c}}\right)_k \right\vert^2 \right] \right)\\
&= \frac{1}{2} s_{\mathrm{c}}^{\dagger}\Phi_{\mathrm{c}}^{-1} s_{\mathrm{c}},
\end{align}
with a matrix $\Phi_{\mathrm{c}}$ in $(\mathbb{R}^2)^{n\times n}$, that has the entries
\begin{align}
(\Phi_{\mathrm{c}})_{0q}&= \frac{2\pi}{\beta} \delta_{0q} \begin{pmatrix}
\mu^{-2} & 0\\
0 & 1
\end{pmatrix},\\
(\Phi_{\mathrm{c}})_{kq} &= \frac{\pi}{\beta} \delta_{kq} \begin{pmatrix}
\omega_k^{-2} & 0\\
0 & 1
\end{pmatrix}\ \text{ for } k>0,
\end{align}
with $0\leq k,q <n$.\\
We would like to stay in a real setting within our probabilistic framework and still need to include our knowledge about the zeroth Fourier coefficients to be real valued. Thus, we identify the vector $(\hat{\phi}_{\mathrm{c}}(k,t))_{k=0}^{n-1}\in \mathbb{C}^{2n}$ at time $t\in [0,T]$ with the vector
\begin{align}
\hat{\phi}_{\mathrm{r}}(t)&:= \left(\hat{\varphi}_{\mathrm{c}}(0,t), \mathrm{Re}\left(\hat{\varphi}_{\mathrm{c}}(1,t)\right),\mathrm{Im}\left(\hat{\varphi}_{\mathrm{c}}(1,t)\right),...,\mathrm{Re}\left(\hat{\varphi}_{\mathrm{c}}(n-1,t)\right),\mathrm{Im}\left(\hat{\varphi}_{\mathrm{c}}(n-1,t)\right),\right.\\
&\left. \hat{\chi}_{\mathrm{c}}(0,t), \mathrm{Re}\left(\hat{\chi}_{\mathrm{c}}(1,t)\right),\mathrm{Im}\left(\hat{\chi}_{\mathrm{c}}(1,t)\right),...,\mathrm{Re}\left(\hat{\chi}_{\mathrm{c}}(n-1,t)\right),\mathrm{Im}\left(\hat{\chi}_{\mathrm{c}}(n-1,t)\right) \right)^T \in \mathbb{R}^{4n-2},
\end{align}
and call this our signal.\\
The exponent of our prior for the real valued signal at time $t_{0}=0$, with the argument $s\in \mathbb{R}^{4n-2}$, is then given by
\begin{align}
\beta \cdot \mathcal{H}(\phi)(0)\left\vert_{\hat{\phi}_{\mathrm{r}}(0)=s} \right. = \frac{1}{2} s^{T}\Phi_{\mathrm{r}}^{-1} s,
\end{align}
with the matrix $\Phi_{\mathrm{r}}= \begin{pmatrix}
\Phi_{\mathrm{r}}^{(\varphi)} & \Phi_{\mathrm{r}}^{(\chi,\varphi)}\\
\Phi_{\mathrm{r}}^{(\varphi,\chi)} & \Phi_{\mathrm{r}}^{(\chi)}
\end{pmatrix}\in \mathbb{R}^{4n-2\times 4n-2}$ consisting of four matrices in $\mathbb{R}^{2n-1\times 2n-1}$, for each of which the first row and column is separated and the rest is the decomposition of matrices in $\mathbb{R}^{2\times 2}$, i.e.
\begin{align}
(\Phi_{\mathrm{r}}^{(\varphi)})_{00} &= \frac{2\pi}{\beta}\mu^{-2}, \\
(\Phi_{\mathrm{r}}^{(\varphi)})_{kq} &= \frac{\pi}{\beta}\omega_k^{-2} \delta_{kq}\mathbbm{1}_{\mathbb{R}^{2}}  \text{ for } k>0,\\
(\Phi_{\mathrm{r}}^{(\chi)})_{00} &= \frac{2\pi}{\beta},\\
(\Phi_{\mathrm{r}}^{(\chi)})_{kq} &= \frac{\pi}{\beta}\delta_{kq}\mathbbm{1}_{\mathbb{R}^{2}} \text{ for } k>0 \text{ and }\\
\Phi_{\mathrm{r}}^{(\chi,\varphi)} &= \Phi_{\mathrm{r}}^{(\varphi,\chi)} = 0\in \mathbb{R}^{2n-1\times 2n-1},\label{prior_covariance}
\end{align}
with $0< k,q <n$, and all the other entries in the first row and column of each of the matrices are zero.\\
As mentioned above, by \citep{enss}, the prior knowledge for this time is then described by a probability density that is an exponential function with $\beta$ times the negative energy of the signal configuration in the exponent. Thus, the prior for $t_0=0$ is a Gaussian density on $\mathbb{R}^{4n-2}$ with zero mean and covariance matrix $\Phi_{\mathrm{r}}$,
\begin{align}
\mathcal{P}_{\hat{\phi}_{\mathrm{r}}(0)}(s)=\dfrac{1}{\left\vert 2\pi \Phi_{\mathrm{r}} \right\vert^{1/2}}\exp\left(-\dfrac{1}{2} s^{T}\Phi_{\mathrm{r}}^{-1} s\right),
\end{align}
and this prior can be assumed for each time in our simulation.

\section{Data constraints}
In this section, we specify the measurement process in our example. We assume that our measurement device averages the field values and the values of the time derivative of the field within the elements of a partition $\left([i\Delta, (i+1)\Delta)\right)_{i=0}^{Y-1}$ of $[0,2\pi)$, consisting of $Y$ intervals of lengths $\Delta=\dfrac{2\pi}{Y}$. We also suppose that the same measurement relation is valid for all times. Our data at time $t\in [0,T]$ is therefore given by a vector 
\[
d(t)= \left(d^{(\varphi)}_{0}(t),...,d^{(\varphi)}_{Y-1}(t),d^{(\chi)}_{0}(t),...,d^{(\chi)}_{Y-1}(t)\right)^T \in \mathbb{R}^{2Y}
\]
and related to the signal by
\begin{align}
d(t)=\begin{pmatrix}
R & 0\\
0 & R
\end{pmatrix} \begin{pmatrix}
\varphi(t)\\
\chi(t)
\end{pmatrix} + \begin{pmatrix}
n^{(\varphi)}(t)\\
n^{(\chi)}(t)
\end{pmatrix},\label{data_relation}
\end{align}
where $n(t)=\begin{pmatrix}
n^{(\varphi)}(t)\\
n^{(\chi)}(t)
\end{pmatrix}$ is an uncertainty vector and $R=(R_{0},...,R_{Y-1})^{T}$ consists of $Y$ integral operators, that act on functions $f$ with the properties of our Klein-Gordon fields as
\begin{align}
R_{i}f &= \int\limits_{0}^{2\pi}\dfrac{1}{\Delta}1_{[i\Delta, (i+1)\Delta)}(x)f(x) dx,
\end{align}
for $i=0,...,Y-1$.\\
In order to get a relation for the Fourier representation of our signal, we consider the discrete Fourier transform $\hat{d}_{\mathrm{c}}(t)$ of our data and of the noise respectively, which is related to $d(t)$ by
\begin{align}
\left(\hat{d}_{\mathrm{c}}^{(\varphi)}\right)_{k}(t)&= \sum\limits_{j=0}^{Y-1}\Delta e^{ikj\Delta}d^{(\varphi)}_{j}(t),\\
\left(\hat{d}_{\mathrm{c}}^{(\chi)}\right)_{k}(t)&= \sum\limits_{j=0}^{Y-1}\Delta e^{ikj\Delta}d^{(\chi)}_{j}(t),
\end{align}
and
\begin{align}
d_{j}^{(\varphi)}(t)&= \dfrac{1}{2\pi}\sum\limits_{k=0}^{Y-1} e^{-ikj\Delta} \left(\hat{d}^{(\varphi)}_{\mathrm{c}}\right)_{k}(t),\\
d_{j}^{(\chi)}(t)&= \dfrac{1}{2\pi}\sum\limits_{k=0}^{Y-1} e^{-ikj\Delta} \left(\hat{d}_{\mathrm{c}}^{(\chi)}\right)_{k}(t),
\end{align}
where $0\leq k,j<Y$.\\
Because $d_{j}^{(\varphi)}(t)\in \mathbb{R}$  for all $j$, it is further 
\begin{align}
&\left(\hat{d}_{\mathrm{c}}^{(\varphi)}\right)_{0}(t)\in \mathbb{R} \text{ and}\\
&\left(\hat{d}_{\mathrm{c}}^{(\varphi)}\right)_{k}(t)= \overline{\left(\hat{d}_{\mathrm{c}}^{(\varphi)}\right)_{Y-k}(t)} \text{ for } 0<k<Y,
\end{align}
and the same for the $(\chi)$-part of our data and also for the noise.\\
Therefore, all the information that we have about the signal is already contained in $\left( \left(\hat{d}_{\mathrm{c}}^{(\varphi)}\right)_{k}(t)\right)_{k=0}^{(Y+1)/2}$ and $\left( \left(\hat{d}_{\mathrm{c}}^{(\chi)}\right)_{k}(t)\right)_{k=0}^{(Y+1)/2}$, where we assume $Y$ to be uneven and greater than one.\\
All computations can also be done if $Y$ is even, but we decide for one case only, to keep the notation clearly arranged.\\
We only have to consider those first Fourier coefficients of the data and the noise.\\
Our first aim is now to get a matrix representation $\hat{R}_{\mathrm{c}}$ of $R$, which acts on $\hat{\varphi}_{\mathrm{c}}(t)$ or $\hat{\chi}_{\mathrm{c}}(t)$ respectively.
Because relation \eqref{data_relation} is the same for all times $t\in [0,T]$ by assumption, we drop the index $t$ in the following, to keep the notation clearly arranged. We also only consider $\hat{\varphi}_{\mathrm{c}}$, because the response acts with $\hat{R}_{\mathrm{c}}$ on both components of $\hat{\phi}_{\mathrm{c}}$.\\ 
So the $(\varphi)$-part of \eqref{data_relation} in Fourier representation reads as
\begin{align}
\left(\hat{d}_{\mathrm{c}}^{(\varphi)}\right)_{k}&= \sum\limits_{j=0}^{Y-1}\Delta e^{ikj\Delta}d^{(\varphi)}_{j} \\
&= \sum\limits_{j=0}^{Y-1}\Delta e^{ikj\Delta} R_{j}\varphi + \left(\hat{n}_{\mathrm{c}}^{(\varphi)}\right)_{k}\\
&= \sum\limits_{j=0}^{Y-1}\Delta e^{ikj\Delta} \int\limits_{0}^{2\pi}\dfrac{1}{\Delta}1_{[j\Delta, (j+1)\Delta)}(x) \dfrac{1}{2\pi} \sum\limits_{l=-(n-1)}^{n-1}e^{-ilx}\hat{\varphi}_{\mathrm{c}}(l) dx + \left(\hat{n}_{\mathrm{c}}^{(\varphi)}\right)_{k}\\
&= \sum\limits_{l=-(n-1)}^{n-1}\dfrac{\hat{\varphi}_{\mathrm{c}}(l)}{2\pi} \sum\limits_{j=0}^{Y-1}e^{ikj\Delta} \underbrace{\int\limits_{j\Delta}^{(j+1)\Delta}e^{-ilx}dx}_{\substack{=\dfrac{1}{il}\left[e^{-ilj\Delta} - e^{-il(j+1)\Delta}\right]}} + \left(\hat{n}_{\mathrm{c}}^{(\varphi)}\right)_{k}\\
&= \sum\limits_{l=-(n-1)}^{n-1}\dfrac{\hat{\varphi}_{\mathrm{c}}(l)}{2\pi} \underbrace{\left(\sum\limits_{j=0}^{Y-1}e^{ij\Delta(k-l)}\right)}_{\substack{=\theta(k =l\, \mathrm{mod}\, Y )\cdot Y}} \frac{1-e^{-il\Delta}}{il}+ \left(\hat{n}_{\mathrm{c}}^{(\varphi)}\right)_{k}\\
&=\sum\limits_{l=-(n-1)}^{n-1}\left(\hat{R}_{\mathrm{c}}\right)_{kl}\hat{\varphi}_{\mathrm{c}}(l) + \left(\hat{n}_{\mathrm{c}}^{(\varphi)}\right)_{k} =\left(\hat{R}_{\mathrm{c}}\hat{\varphi}_{\mathrm{c}}\right)_k + \left(\hat{n}_{\mathrm{c}}^{(\varphi)}\right)_{k}
\end{align}
for $0\leq k \leq \frac{Y+1}{2}$ and with $\hat{R}_{\mathrm{c}}\in \mathbb{R}^{\left(\frac{Y+1}{2}+1\right)\times 2n-1}$ given by
\begin{align}
\left(\hat{R}_{\mathrm{c}}\right)_{kl} = \theta(k =l\, \mathrm{mod}\, Y )\cdot \frac{1-e^{-il\Delta}}{il\Delta} = \theta(k =l\, \mathrm{mod}\, Y ) \cdot e^{-\frac{1}{2}il\Delta} \cdot \mathrm{sinc}\left( \frac{1}{2}l\Delta \right).
\end{align}
In this context, $\theta$ is defined as
\begin{align}
\theta(k =l\, \mathrm{mod}\, Y )&:= 1\text{ if }k =l\, \mathrm{mod}\, Y,\\
\theta(k =l \, \mathrm{mod}\,  Y )&:= 0 \text{ else}
\end{align}
and
\begin{align}
\mathrm{sinc}(x):=\dfrac{\sin(x)}{x}.
\end{align}
Because we actually want to have the relation of our data to the signal, interpreted as an $\mathbb{R}^{4n-2}$-dimensional vector, only considering non-negative Fourier coefficients, we further interpret $\hat{d}_{\mathrm{c}}$ as the vector
\begin{align}
\hat{d}_{\mathrm{r}}&=\left(\left(\hat{d}_{\mathrm{c}}^{(\varphi)}\right)_{0},\mathrm{Re}\left(\left(\hat{d}_{\mathrm{c}}^{(\varphi)}\right)_{1}\right),\mathrm{Im}\left(\left(\hat{d}_{\mathrm{c}}^{(\varphi)}\right)_{1}\right),...,\mathrm{Re}\left(\left(\hat{d}_{\mathrm{c}}^{(\varphi)}\right)_{(Y+1)/2}\right),\mathrm{Im}\left(\left(\hat{d}_{\mathrm{c}}^{(\varphi)}\right)_{(Y+1)/2}\right),\right.\\
&\left.\left(\hat{d}_{\mathrm{c}}^{(\chi)}\right)_{0},\mathrm{Re}\left(\left(\hat{d}_{\mathrm{c}}^{(\chi)}\right)_{1}\right),\mathrm{Im}\left(\left(\hat{d}_{\mathrm{c}}^{(\chi)}\right)_{1}\right),...,\mathrm{Re}\left(\left(\hat{d}_{\mathrm{c}}^{(\chi)}\right)_{(Y+1)/2}\right),\mathrm{Im}\left(\left(\hat{d}_{\mathrm{c}}^{(\chi)}\right)_{(Y+1)/2}\right)\right)^T\in \mathbb{R}^{2(Y+2)},
\end{align}
and the noise with a vector consisting of the real and imaginary parts of its components respectively.\\
$\hat{d}_{\mathrm{r}}$ is related to $\hat{\phi}_{\mathrm{r}}=\begin{pmatrix}
\hat{\varphi}_{\mathrm{r}}\\
\hat{\chi}_{\mathrm{r}}
\end{pmatrix}$ by
\begin{align}
\hat{d}_{\mathrm{r}}= \begin{pmatrix}
\hat{R}_{\mathrm{r}} & 0\\
0 & \hat{R}_{\mathrm{r}}
\end{pmatrix} \begin{pmatrix}
\hat{\varphi}_{\mathrm{r}}\\
\hat{\chi}_{\mathrm{r}}
\end{pmatrix} +
\hat{n}_{\mathrm{r}},
\end{align}
where, as we will see, $\hat{R}_{\mathrm{r}} = \begin{pmatrix}
\left(\hat{R}_{\mathrm{r}}\right)_{00} & 0 &  \cdots & 0\\
0 & \left(\hat{R}_{\mathrm{r}}\right)_{11} & \cdots & \left(\hat{R}_{\mathrm{r}}\right)_{1(n-1)}\\
\vdots & \vdots & \ddots & \vdots\\
0 & \left(\hat{R}_{\mathrm{r}}\right)_{\frac{Y+1}{2}1} & \cdots & \left(\hat{R}_{\mathrm{r}}\right)_{\frac{Y+1}{2}(n-1)}
\end{pmatrix}$ is a matrix in $\mathbb{R}^{Y+2 \times 2n-1}$, with $\left(\hat{R}_{\mathrm{r}}\right)_{00}\in \mathbb{R}$ and $\left(\hat{R}_{\mathrm{r}}\right)_{kl} \in \mathbb{R}^{2\times 2}$ for $k,l>0$, so that the relation of $\hat{d}_{\mathrm{c}}$ to $\hat{\phi}_{\mathrm{c}}$ is captured.\\
We also define $\hat{\phi}_{\mathrm{r}}$ for negative coordinates in the spirit of \eqref{complex_conj} by
\begin{align}
\begin{pmatrix}
\mathrm{Re}\left(\hat{\varphi}_{\mathrm{c}}(l)\right) \\
\mathrm{Im}\left(\hat{\varphi}_{\mathrm{c}}(l)\right)
\end{pmatrix} = \begin{pmatrix}
1 & 0\\
0 & -1
\end{pmatrix} \begin{pmatrix}
\mathrm{Re}\left(\hat{\varphi}_{\mathrm{c}}(-l)\right) \\
\mathrm{Im}\left(\hat{\varphi}_{\mathrm{c}}(-l)\right)
\end{pmatrix}
\end{align}
for $l>0$ and the same for the $\chi$-part of $\hat{\phi}_{\mathrm{r}}$.\\
Because by continuity
\begin{align}
\mathrm{sinc}(0)&=1
\end{align}
and by Euler's formula for the exponential function, the $k$-th component of the data after subtracting the noise is given by
\begin{align}
&\theta(k =0 \, \mathrm{mod}\,  Y )\cdot \hat{\varphi}_{\mathrm{c}}(0) \\
&+ \sum\limits_{0< \vert l \vert < n-1} \theta(k =l \, \mathrm{mod}\,  Y )\cdot \mathrm{sinc}\left( \frac{1}{2}l\Delta \right) \cdot \begin{pmatrix}
\cos\left(\dfrac{1}{2}l\Delta\right) & \sin\left(\dfrac{1}{2}l\Delta\right)\\
-\sin\left(\dfrac{1}{2}l\Delta\right) & \cos\left(\dfrac{1}{2}l\Delta\right)
\end{pmatrix} \begin{pmatrix}
\mathrm{Re}\left(\hat{\varphi}_{\mathrm{c}}(l)\right) \\
\mathrm{Im}\left(\hat{\varphi}_{\mathrm{c}}(l)\right)
\end{pmatrix}.
\end{align}
In case $k=0$, because
\begin{align}
\sum\limits_{0< \vert l \vert <n}    \underbrace{\theta(l \, \mathrm{mod}\,  Y = 0 ) \mathrm{sinc}\left( \frac{1}{2}l\Delta \right)}_{\substack{= 0 \text{ for } l \neq 0}} 
\left[ 
...
\right]
\begin{pmatrix}
\mathrm{Re}\left(\hat{\varphi}_{\mathrm{c}}(l)\right) \\
\mathrm{Im}\left(\hat{\varphi}_{\mathrm{c}}(l)\right)
\end{pmatrix}=0,
\end{align}
this simplifies to
\begin{align}
\left(\hat{d}^{(\varphi)}_{\mathrm{c}}\right)_{0} - \left(\hat{n}^{(\varphi)}_{\mathrm{c}}\right)_{0}
&= \hat{\varphi}_{\mathrm{c}}(0).
\end{align}
Therefore, we have
\begin{align}
\left(\hat{R}_{\mathrm{r}}\right)_{00}&=1,
\end{align}
and the rest of the first row of $\hat{R}_{\mathrm{r}}$ is indeed zero.\\
For $k>0$, with
\begin{align}
\mathrm{sinc}(-x)&=\mathrm{sinc}(x),\\
\sin(x+ \pi) &= -\sin(x)\text{ and}\\
\cos(x+ \pi) &= -\cos(x),
\end{align}
we get
\begin{align}
&\begin{pmatrix}
\mathrm{Re}\left(\left(\hat{d}^{(\varphi)}_{\mathrm{c}}\right)_{k}\right)\\
\mathrm{Im}\left(\left(\hat{d}^{(\varphi)}_{\mathrm{c}}\right)_{k}\right)
\end{pmatrix} - \begin{pmatrix}
\mathrm{Re}\left(\left(\hat{n}^{(\varphi)}_{\mathrm{c}}\right)_{k}\right)\\
\mathrm{Im}\left(\left(\hat{n}^{(\varphi)}_{\mathrm{c}}\right)_{k}\right)
\end{pmatrix}\\
&= 
\sum\limits_{l=1}^{n-1}
\mathrm{sinc}\left( \frac{1}{2}l\Delta \right)\cdot \left[
\theta(l \, \mathrm{mod}\,  Y = k ) \cdot
\begin{pmatrix}
\cos\left(\dfrac{1}{2}l\Delta\right) & \sin\left(\dfrac{1}{2}l\Delta\right)\\
-\sin\left(\dfrac{1}{2}l\Delta\right) & \cos\left(\dfrac{1}{2}l\Delta\right)
\end{pmatrix} \right. \\
&\left.  +
\theta(-l \, \mathrm{mod}\,  Y = k )\cdot
\begin{pmatrix}
\cos\left(\dfrac{1}{2}l\Delta\right) & -\sin\left(\dfrac{1}{2}l\Delta\right)\\
\sin\left(\dfrac{1}{2}l\Delta\right) & \cos\left(\dfrac{1}{2}l\Delta\right)
\end{pmatrix}\begin{pmatrix}
1 & 0\\
0 & -1
\end{pmatrix}
\right]
\begin{pmatrix}
\mathrm{Re}\left(\hat{\varphi}_{\mathrm{c}}(l)\right) \\
\mathrm{Im}\left(\hat{\varphi}_{\mathrm{c}}(l)\right)
\end{pmatrix}\\
&= \sum\limits_{l=1}^{n-1} \mathrm{sinc}\left( \frac{1}{2}l\Delta \right)\cdot
\left[ 
\theta(l \, \mathrm{mod}\,  Y = k )\cdot
\begin{pmatrix}
\cos\left(\dfrac{1}{2}l\Delta\right) & \sin\left(\dfrac{1}{2}l\Delta\right)\\
-\sin\left(\dfrac{1}{2}l\Delta\right) & \cos\left(\dfrac{1}{2}l\Delta\right)
\end{pmatrix} \right. \\
&\left.  + \theta(-l \, \mathrm{mod}\,  Y = k )\cdot
\begin{pmatrix}
\cos\left(\dfrac{1}{2}l\Delta\right) & \sin\left(\dfrac{1}{2}l\Delta\right)\\
\sin\left(\dfrac{1}{2}l\Delta\right) & -\cos\left(\dfrac{1}{2}l\Delta\right)
\end{pmatrix}
\right]
\begin{pmatrix}
\mathrm{Re}\left(\hat{\varphi}_{\mathrm{c}}(l)\right) \\
\mathrm{Im}\left(\hat{\varphi}_{\mathrm{c}}(l)\right)
\end{pmatrix}.
\end{align}
Therefore also the rest of the first column except of $\left(\hat{R}_{\mathrm{r}}\right)_{00}$ is zero and
\begin{align}
\left(\hat{R}_{\mathrm{r}}\right)_{00}&=1 \text{ as well as}\\
\left(\hat{R}_{\mathrm{r}}\right)_{kl}&= \mathrm{sinc}\left( \frac{1}{2}l\Delta \right)\cdot
\left[ 
\theta(l \, \mathrm{mod}\,  Y = k )\cdot
\begin{pmatrix}
\cos\left(\dfrac{1}{2}l\Delta\right) & \sin\left(\dfrac{1}{2}l\Delta\right)\\
-\sin\left(\dfrac{1}{2}l\Delta\right) & \cos\left(\dfrac{1}{2}l\Delta\right)
\end{pmatrix} \right. \\
&\left.  + \theta(-l \, \mathrm{mod}\,  Y = k )\cdot
\begin{pmatrix}
\cos\left(\dfrac{1}{2}l\Delta\right) & \sin\left(\dfrac{1}{2}l\Delta\right)\\
\sin\left(\dfrac{1}{2}l\Delta\right) & -\cos\left(\dfrac{1}{2}l\Delta\right)
\end{pmatrix}
\right] \text{ for } 0<k\leq \frac{Y+1}{2} \text{ and } 0<l<n.
\end{align}
Now we assume, that the noise $\hat{n}_{\mathrm{r}}$ is the realization of a random vector $\hat{n}_{\mathrm{r}}$ in $\mathbb{R}^{2(Y+2)}$ with Gaussian distribution, i.e.
\begin{align}
\mathcal{P}_{\hat{n}_{\mathrm{r}}}(u)=\mathcal{G}(u,N)
\end{align}
with $N=\sigma_{n}^{2}\cdot \mathbbm{1}_{\mathbb{R}^{2(Y+2)}}=\begin{pmatrix}
N^{(\varphi)} & 0\\
0 & N^{(\chi)}
\end{pmatrix}$ for some $\sigma_{n}^{2}>0$ and $N^{(\mathrm{part})} := N^{(\varphi)} = N^{(\chi)}= \sigma_{n}^{2}\cdot \mathbbm{1}_{\mathbb{R}^{Y+2}}$.\\
To compute the information propagator $D$ from Theorem \ref{thm:posterior} of the posterior density for our signal given the data, we need to calculate $\hat{R}_{\mathrm{r}}^T \left(N^{(\mathrm{part})}\right)^{-1} \hat{R}_{\mathrm{r}}\in \mathbb{R}^{2n-1\times 2n-1} $, which we interpret as a matrix $\left((\hat{R}_{\mathrm{r}}^T \left(N^{(\mathrm{part})}\right)^{-1} \hat{R}_{\mathrm{r}})_{kl}\right)_{0\leq k,l <n }$, with real numbers in the first row and column and components in $\mathbb{R}^{2\times 2}$ for $\vert k \vert,\vert l \vert >0$, as we already did with $\hat{R}_{\mathrm{r}}$.\\
Those components simplify to
\begin{align}
\left(\hat{R}_{\mathrm{r}}^T \left(N^{(\mathrm{part})}\right)^{-1} \hat{R}_{\mathrm{r}}\right)_{kl} &= \sigma_{n}^{-2} \left(\hat{R}_{\mathrm{r}}^T \hat{R}_{\mathrm{r}}\right)_{kl}\\
&=\sigma_{n}^{-2} \sum\limits_{q=1}^{\frac{Y+1}{2}} \left(\hat{R}_{\mathrm{r}}^{T}\right)_{kq}\left(\hat{R}_{\mathrm{r}}\right)_{ql} = \sigma_{n}^{-2} \sum\limits_{q=1}^{\frac{Y+1}{2}} \left(\left(\hat{R}_{\mathrm{r}}\right)_{qk}\right)^T\left(\hat{R}_{\mathrm{r}}\right)_{ql}.
\end{align}
Therefore, we get
\begin{align}
\left(\hat{R}_{\mathrm{r}}^T \left(N^{(\mathrm{part})}\right)^{-1} \hat{R}_{\mathrm{r}}\right)_{00} &= \sigma_{n}^{-2}
\end{align}
and for $1\leq k,l <n$ it is
\begin{align}
&\left(\hat{R}_{\mathrm{r}}^T \left(N^{(\mathrm{part})}\right)^{-1} \hat{R}_{\mathrm{r}}\right)_{kl}= \sigma_{n}^{-2} \sum\limits_{q=1}^{\frac{Y+1}{2}} \left(\left(\hat{R}_{\mathrm{r}}\right)_{qk}\right)^T\left(\hat{R}_{\mathrm{r}}\right)_{ql}\\
&= \sigma_{n}^{-2} \theta(k\, \mathrm{mod}\, Y > 0) \mathrm{sinc}\left(\frac{1}{2}k\Delta \right)\mathrm{sinc}\left(\frac{1}{2} l\Delta\right)\cdot \left[ (-1)^{\frac{k-l}{Y}}\theta((k-l)\, \mathrm{mod}\, Y = 0)\mathbbm{1}_{\mathbb{R}^{2}} \right.\\
&\left. + (-1)^{\frac{k+l}{Y}}\theta((k+l)\, \mathrm{mod}\, Y = 0)\cdot  \begin{pmatrix}
1 - 2\sin^2\left(\dfrac{1}{2}(k\, \mathrm{mod}\, Y)\Delta\right) & 0\\
0 & -1
\end{pmatrix}
\right].
\end{align}
The inverse of $D=\begin{pmatrix}
D^{(\varphi)} & D^{(\chi,\varphi)}\\
D^{(\varphi,\chi)} & D^{(\chi)}
\end{pmatrix}$, with $\Phi_{\mathrm{r}}$ from \eqref{prior_covariance}, consists of four matrices in $\mathbb{R}^{2n-1\times 2n-1}$ that are given by
\begin{align}
(D^{-1})^{(\varphi)} &= \left(\Phi_{\mathrm{r}}^{(\varphi)}\right)^{-1} + \hat{R}_{\mathrm{r}}^T \left(N^{(\mathrm{part})}\right)^{-1} \hat{R}_{\mathrm{r}},\\
(D^{-1})^{(\chi)} &= \left(\Phi_{\mathrm{r}}^{(\chi)}\right)^{-1} + \hat{R}_{\mathrm{r}}^T \left(N^{(\mathrm{part})}\right)^{-1} \hat{R}_{\mathrm{r}},\\
(D^{-1})^{(\chi,\varphi)}&= (D^{-1})^{(\varphi,\chi)} = 0\in \mathbb{R}^{2n-1\times 2n-1}.
\end{align}
That is, $D$ is given by
\begin{align}
D^{(\varphi)} &= \left[\left(\Phi_{\mathrm{r}}^{(\varphi)}\right)^{-1} + \hat{R}_{\mathrm{r}}^T \left(N^{(\mathrm{part})}\right)^{-1} \hat{R}_{\mathrm{r}}\right]^{-1},\\
D^{(\chi)} &= \left[\left(\Phi_{\mathrm{r}}^{(\chi)}\right)^{-1} + \hat{R}_{\mathrm{r}}^T \left(N^{(\mathrm{part})}\right)^{-1} \hat{R}_{\mathrm{r}}\right]^{-1},\\
D^{(\chi,\varphi)}&= D^{(\varphi,\chi)} = 0\in \mathbb{R}^{2n-1\times 2n-1}.
\end{align}
The posterior mean from Theorem \ref{thm:posterior} $m\left(\hat{d}_{\mathrm{r}}\right)$ is then given by
\begin{align}
m\left(\hat{d}_{\mathrm{r}}\right)= D \begin{pmatrix}
{R}_{\mathrm{r}} & 0\\
0 & {R}_{\mathrm{r}}
\end{pmatrix}^T N^{-1}\hat{d}_{\mathrm{r}} = \sigma_{n}^{-2}D\begin{pmatrix}
{R}_{\mathrm{r}}^T & 0\\
0 & {R}_{\mathrm{r}}^T
\end{pmatrix} \hat{d}_{\mathrm{r}},
\end{align}
and the Wiener filter from \eqref{def:Wienerfilter} by
\begin{align}
W = D \begin{pmatrix}
{R}_{\mathrm{r}} & 0\\
0 & {R}_{\mathrm{r}}
\end{pmatrix}^T N^{-1} = \sigma_{n}^{-2}D\begin{pmatrix}
{R}_{\mathrm{r}}^T & 0\\
0 & {R}_{\mathrm{r}}^T
\end{pmatrix}.
\end{align}
We now include the time index again and split $(0,T]$ into $2^N$ intervals $(t_{i},t_{i+1}]= (i\cdot \delta t,(i+1)\cdot \delta t]$ for $i=0,...,2^N-1$, each of length $\delta t = \frac{T}{2^N}$. The posterior for the time $t_{i}$ is then given by
\begin{align}
\mathcal{P}_{\hat{\phi}_{\mathrm{r}}(t_{i})\vert \hat{d}_{\mathrm{r}}(t_{i})}(s,u)=\mathcal{G}(s-m(u),D),
\end{align}
with $s\in \mathbb{R}^{4n-2}$ and $u\in \mathbb{R}^{2(Y+2)}$.

\section{Field evolution with IFD}
The evolution of a field from time $t_{i}$ to $t_{i+1}$ in IFD is approximated by a linear and time independent operator. In our case, we take the sum of the identity matrix and the time derivative of the field at time $t_{i}$, multiplied by the evolution time $\delta t$.
By \eqref{partial_derivative} we have 
\begin{align}
\partial_{t}\hat{\phi}_{\mathrm{c}}(k,t) &= \partial_{t}\begin{pmatrix}
\hat{\varphi}_{\mathrm{c}}(k,t)\\
\hat{\chi}_{\mathrm{c}}(k,t)
\end{pmatrix} = \left( L_{\mathrm{c}} \hat{\phi}_{\mathrm{c}}(t)\right)_{k},
\end{align}
with 
\begin{align}
\left(L_{\mathrm{c}}\right)_{kq}= \delta_{kq}\begin{pmatrix}
0 & 1\\
-\omega_{k}^2 & 0
\end{pmatrix}
\end{align}
for $0\leq k,q <n$.\\
This means that we approximate the evolution of $\hat{\phi}_{\mathrm{c}}$ by
\begin{align}
\hat{\phi}'_{\mathrm{c}}(t_{i+1}) = \left( \mathbbm{1}_{\mathbb{C}^{2n}} + \delta t \cdot L_{\mathrm{c}}\right)\hat{\phi}_{\mathrm{c}}(t_{i}).
\end{align}
For our representation of the signal in real space, this means
\begin{align}
\hat{\phi}'_{\mathrm{r}}(t_{i+1}) = \left( \mathbbm{1}_{\mathbb{R}^{4n-2}} + \delta t \cdot L_{\mathrm{r}}\right)\hat{\phi}_{\mathrm{r}}(t_{i})=:g\left(\hat{\phi}_{\mathrm{r}}(t_{i})\right),
\end{align}
with $L_{\mathrm{r}}\in \mathbb{R}^{4n-2}$ given by
\begin{align}
L_{\mathrm{r}}= \begin{pmatrix}
L_{\mathrm{r}}^{(\varphi)} & L_{\mathrm{r}}^{(\chi,\varphi)}\\
L_{\mathrm{r}}^{(\varphi,\chi)} & L_{\mathrm{r}}^{(\chi)}
\end{pmatrix},
\end{align}
where
\begin{align}
L_{\mathrm{r}}^{(\varphi)}&= L_{\mathrm{r}}^{(\chi)} = 0\in \mathbb{R}^{2n-1},\\
L_{\mathrm{r}}^{(\chi,\varphi)} &= \mathbbm{1}_{\mathbb{R}^{2n-1}}, \\
\left(L_{\mathrm{r}}^{(\varphi,\chi)}\right)_{00} &= -\mu^2 \text{ and}\\
\left(L_{\mathrm{r}}^{(\varphi,\chi)}\right)_{kq} &= -\omega_{k}^2 \delta_{kq}\mathbbm{1}_{\mathbb{R}^{2}} \text{ for } 0< k,q <n,
\end{align}
and the rest of the first row and column of $L_{\mathrm{r}}^{(\varphi,\chi)}$ is zero.\\
The exact evolution is given by \eqref{exact_evolution2}, with $t=t_{i}$ and $s=t_{i+1}$.
As a matrix acting on the real space signal, $A_{\mathrm{c}}(t_{i},t_{i+1})$ reads as
\begin{align}
A_{\mathrm{r}}:=A_{\mathrm{r}}(t_{i},t_{i+1}) = \begin{pmatrix}
A_{\mathrm{r}}^{(\varphi)} & A_{\mathrm{r}}^{(\chi,\varphi)}\\
A_{\mathrm{r}}^{(\varphi,\chi)} & A_{\mathrm{r}}^{(\chi)}
\end{pmatrix},
\end{align}
with
\begin{align}
\left(A_{\mathrm{r}}^{(\varphi)}\right)_{00}&= \left(A_{\mathrm{r}}^{(\chi)}\right)_{00}\cos(\mu \cdot \delta t),\\
\left(A_{\mathrm{r}}^{(\varphi)}\right)_{kq}&= \left(A_{\mathrm{r}}^{(\chi)}\right)_{kq}\delta_{kq}\cos(\omega_{k}\cdot \delta t)\mathbbm{1}_{\mathbb{R}^{2}} \text{ for } 0<k,q<n,\\
\left(A_{\mathrm{r}}^{(\chi,\varphi)}\right)_{00} &=  \frac{1}{\mu}\sin(\mu\cdot \delta t),\\
\left(A_{\mathrm{r}}^{(\chi,\varphi)}\right)_{kq} &=  \delta_{kq} \frac{1}{\omega_{k}}\sin(\omega_{k}\cdot \delta t)\mathbbm{1}_{\mathbb{R}^{2}} \text{ for }0<k,q<n,\\
\left(A_{\mathrm{r}}^{(\varphi,\chi)}\right)_{00} &= -\mu\sin(\mu\cdot \delta t)\text{ and}\\
\left(A_{\mathrm{r}}^{(\varphi,\chi)}\right)_{kq} &= -\delta_{kq} \omega_{k}\sin(\omega_{k}\cdot \delta t)\mathbbm{1}_{\mathbb{R}^{2}} \text{ for }0<k,q<n,
\end{align}
and the rest of the first row and column of each matrix is zero.\\
We further name $B_{\mathrm{c}}$ the matrix
\begin{align}
\frac{1}{(\delta t)^{2}}\left(A_c(t_{i},t_{i+1}) - \left( \mathbbm{1}_{\mathbb{C}^{2n}} + \delta t \cdot L_{\mathrm{c}}\right)\right).
\end{align}
By a Taylor expansion of the components of $A_c(t_{i},t_{i+1})$ at zero, interpreted as functions of $\delta t$, we get
\begin{align}
\left(B_{\mathrm{c}}\right)_{kq} = \delta_{kq} \begin{pmatrix}
\left(B_{\mathrm{c}}\right)^{(1)}_{k} & \left(B_{\mathrm{c}}\right)^{(2)}_{k}\\
\left(B_{\mathrm{c}}\right)^{(3)}_{k}  & \left(B_{\mathrm{c}}\right)^{(4)}_{k}
\end{pmatrix},
\end{align}
with
\begin{align}
\left(B_{\mathrm{c}}\right)^{(1)}_{k} &= \left(B_{\mathrm{c}}\right)^{(4)}_{k} = -\frac{\omega_{k}^2}{2}\cos(\omega_k\xi_k^{(1)}), \\
\left(B_{\mathrm{c}}\right)^{(2)}_{k} &= -\frac{\omega_{k}}{2}\sin(\omega_k\xi_k^{(2)}),\\
\left(B_{\mathrm{c}}\right)^{(3)}_{k} &= \frac{\omega_{k}^3}{2}\sin(\omega_k\xi_k^{(2)}),
\end{align}
for some numbers $0< \xi^{(1)}_k,\xi^{(2)}_k < \delta t$ and $0\leq k,q <n$.\\
$B_{\mathrm{c}}$ as a matrix acting on $\hat{\phi}_{\mathrm{r}}$ reads as
$B_{\mathrm{r}}= \begin{pmatrix}
B^{(\varphi)}_{\mathrm{r}} & B^{(\chi,\varphi)}_{\mathrm{r}}\\
B^{(\varphi,\chi)}_{\mathrm{r}} & B^{(\chi)}_{\mathrm{r}}\end{pmatrix}\in \mathbb{R}^{4n-2}$ with
\begin{align}
\left(B^{(\varphi)}_{\mathrm{r}}\right)_{00}&= \left(B^{(\chi)}_{\mathrm{r}}\right)_{00}= -\frac{\mu^2}{2}\cos(\mu\xi^{(1)}_0),\\
\left(B^{(\varphi)}_{\mathrm{r}}\right)_{kq}&= \left(B^{(\chi)}_{\mathrm{r}}\right)_{kq}= -\frac{\omega_{k}^2}{2}\cos(\omega_k\xi^{(1)}_k) \delta_{kq} \mathbbm{1}_{\mathbb{R}^2} \text{ for } 0<k,q<n,\\
\left(B^{(\chi,\varphi)}_{\mathrm{r}}\right)_{00} &= -\frac{\mu}{2}\sin(\mu\xi^{(2)}_0),\\
\left(B^{(\chi,\varphi)}_{\mathrm{r}}\right)_{kq} &= -\frac{\omega_{k}}{2}\sin(\omega_k\xi^{(2)}_k)\delta_{kq}\mathbbm{1}_{\mathbb{R}^2} \text{ for } 0<k,q<n,\\
\left(B^{(\varphi,\chi)}_{\mathrm{r}}\right)_{00} &= \frac{\mu^3}{2}\sin(\mu\xi^{(2)}_0) \text{ and}\\
\left(B^{(\varphi,\chi)}_{\mathrm{r}}\right)_{kq} &= \frac{\omega_{k}^3}{2}\sin(\omega_k\xi^{(2)}_k)\delta_{kq}\mathbbm{1}_{\mathbb{R}^2} \text{ for }0<k,q<n.
\end{align}
By Theorem \ref{posterior_update}, the evolved posterior with the approximated evolution is given by 
\begin{align}
\mathcal{P}_{\hat{\phi}'_{\mathrm{r}}(t_{i+1})\vert \hat{d}_{\mathrm{r}}(t_{i})}(s,u)&=\mathcal{G}(g^{-1}(s)-m(u),D)\left\vert Dg^{-1}(s) \right\vert\\
&= \mathcal{G}( s - m''(u),D'')\ \lambda_{\mathbb{R}^{4n-2}}\otimes \mathbb{P}_{\hat{d}_{\mathrm{r}}(t_{i})}\text{-a.e in } \mathbb{R}^{4n-2}\times \mathbb{R}^{2(Y+2)}, \text{ with}\\
m''(u) &= m(u) + \delta t \cdot L_{\mathrm{r}}m(u)\\
&= (\mathbbm{1}_{\mathbb{R}^{4n-2}} + \delta t\cdot L_{\mathrm{r}})m(u)\\
&= (\mathbbm{1}_{\mathbb{R}^{4n-2}} + \delta t\cdot L_{\mathrm{r}})Wu \text{ and}\\
D'' &= \left[(\mathbbm{1}_{\mathbb{R}^{4n-2}} + \delta t\cdot L_{\mathrm{r}})^{-1T} D^{-1} (\mathbbm{1}_{\mathbb{R}^{4n-2}} + \delta t\cdot L_{\mathrm{r}})^{-1}\right]^{-1}\\
&= (\mathbbm{1}_{\mathbb{R}^{4n-2}} + \delta t\cdot L_{\mathrm{r}})D(\mathbbm{1}_{\mathbb{R}^{4n-2}} + \delta t\cdot L_{\mathrm{r}})^{T}\\
&= D + \delta t (L_{\mathrm{r}}D + DL_{\mathrm{r}}^{T}) + (\delta t)^{2}L_{\mathrm{r}}DL_{\mathrm{r}}^{T}.
\end{align}
The posterior evolved according to the exact equation is similarly given by
\begin{align}
\mathcal{P}_{\hat{\phi}_{\mathrm{r}}(t_{i+1})\vert \hat{d}_{\mathrm{r}}(t_{i})}(s,u)&=
\mathcal{G}( s - m'(u),D')\ \lambda_{\mathbb{R}^{4n-2}}\otimes \mathbb{P}_{\hat{d}_{\mathrm{r}}(t_{i})}\text{-a.e in } \mathbb{R}^{4n-2}\times \mathbb{R}^{2(Y+2)}, \text{ with}\\
m'(u) &= A_{\mathrm{r}}m(t) \\
&= (\mathbbm{1}_{\mathbb{R}^{4n-2}} + \delta t\cdot L_{\mathrm{r}})m(u) + (\delta t)^2B_{\mathrm{r}}m(u)\\
D' &= A_{\mathrm{r}}DA_{\mathrm{r}}^{T}\text{ and}\\
&= D'' + (\delta t)^2(\mathbbm{1}_{\mathbb{R}^{4n-2}} + \delta t\cdot L_{\mathrm{r}})DB_{\mathrm{r}}^T + (\delta t)^{2}B_{\mathrm{r}}D(\mathbbm{1}_{\mathbb{R}^{4n-2}} + \delta t\cdot L_{\mathrm{r}})^T\\
 &+ (\delta t)^{4}B_{\mathrm{r}}DB_{\mathrm{r}}^T.
\end{align}
In the last equality we used
\begin{align}
(\mathbbm{1}_{\mathbb{R}^{4n-2}} + \delta t\cdot L_{\mathrm{r}})^{T-1}= (\mathbbm{1}_{\mathbb{R}^{4n-2}} + \delta t\cdot L_{\mathrm{r}})^{-1T}.\label{transposed_inversion}
\end{align}
For the information theoretical error coming with our approximation, we should compute the relative entropy
\begin{align}
&D\left(\mathcal{P}_{\hat{\phi}_{\mathrm{r}}(t_{i+1})\vert \hat{d}_{\mathrm{r}}(t_{i})}(.,u) \Vert \mathcal{P}_{\hat{\phi}'_{\mathrm{r}}(t_{i+1})\vert \hat{d}_{\mathrm{r}}(t_{i})}(.,u) \right)\\
&\overset{\mathrm{Cor.} \, \ref{cross_entropy}}{=} \frac{1}{2} \mathrm{Tr} \left[ ( \delta m(u) \delta m(u)^{T} + D') D''^{-1} - \mathbbm{1}_{\mathbb{R}^{4n-2}} - \log(D'D''^{-1})\right]
\end{align}
for given $\hat{d}_{\mathrm{r}}(t_{i}) = u \in \mathbb{R}^{2(Y+2)}$ and $\delta m(u)= m'(u)-m''(u)$.\\
This expression simplifies because
\begin{align}
\delta m(u)=\left[ A_{\mathrm{r}} - (\mathbbm{1}_{\mathbb{R}^{4n-2}} + \delta t\cdot L_{\mathrm{r}})\right] m(u) =  (\delta t)^2\cdot B_{\mathrm{r}}m(u)
\end{align}
and 
\begin{align}
D'D''^{-1} &=  \mathbbm{1}_{\mathbb{R}^{4n-2}} + (\delta t)^2(\mathbbm{1}_{\mathbb{R}^{4n-2}} + \delta t\cdot L_{\mathrm{r}})DB_{\mathrm{r}}^TD''^{-1}\\
 &+ (\delta t)^{2}B_{\mathrm{r}}(\mathbbm{1}_{\mathbb{R}^{4n-2}} + \delta t\cdot L_{\mathrm{r}})^{-1} + (\delta t)^{4}B_{\mathrm{r}}DB_{\mathrm{r}}^TD''^{-1}\\
&=: \mathbbm{1}_{\mathbb{R}^{4n-2}} + (\delta t)^2C,
\end{align}
with a matrix $C\in \mathbb{R}^{4n-2\times 4n-2}$ that includes all matrices with $(\delta t)^2$ or $(\delta t)^4$ in front of them.\\
Note here that for $(\delta t)^2 < \omega_{n-1}^{-2}$
\begin{align}
(\mathbbm{1}_{\mathbb{R}^{4n-2}} + \delta t\cdot L_{\mathrm{r}})^{-1} = \mathbbm{1}_{\mathbb{R}^{4n-2}} + \mathcal{O}\left( \delta t\right)\label{evolution_ocalc}
\end{align}
and the same holds then for $(\mathbbm{1}_{\mathbb{R}^{4n-2}} + \delta t\cdot L_{\mathrm{r}})^{T-1}$ by \eqref{transposed_inversion}.\\
This is the case because we have
\begin{align}
(\mathbbm{1}_{\mathbb{R}^{4n-2}} + \delta t\cdot L_{\mathrm{r}})^{-1} = \begin{pmatrix}
J^{(\varphi)} & J^{(\chi, \varphi)}\\
J^{(\varphi, \chi)} & J^{(\chi)}
\end{pmatrix},
\end{align}
with
\begin{align}
J^{(\varphi)}_{00}&=J^{(\chi)}_{00}= \frac{1}{1+\mu^{2}(\delta t)^2},\\
J^{(\varphi)}_{kq} &= J^{(\chi)}_{kq} = \frac{1}{1+\omega_{k}^{2}(\delta t)^2} \delta_{kq} \mathbbm{1}_{\mathbb{R}^2} \text{ for } 0<k,q<n,\\
J^{(\chi, \varphi)}_{00}&= \frac{-\delta t}{1+\mu^{2}(\delta t)^2},\\
J^{(\chi, \varphi)}_{kq}&= \frac{-\delta t}{1+\omega_{k}^{2}(\delta t)^2} \delta_{kq} \mathbbm{1}_{\mathbb{R}^2} \text{ for } 0<k,q<n,\\
J^{(\varphi, \chi)}_{00}&= \frac{\omega_{k}^{2}\delta t }{1+\mu^{2}(\delta t)^2} \text{ and }\\
J^{(\varphi, \chi)}_{kq}&= \frac{\omega_{k}^{2}\delta t}{1+\omega_{k}^{2}(\delta t)^2} \delta_{kq} \mathbbm{1}_{\mathbb{R}^2} \text{ for } 0<k,q<n.\label{evolution_inversion}
\end{align}
For $(\delta t)^2 < \omega_{n-1}^{-2}$, all the prefactors $\frac{1}{1+\omega_{k}^{2}(\delta t)^2}$ with $0 \leq k <n$, can be expanded into $\sum\limits_{j=0}^{\infty} \left( - \omega_{k}\delta t \right)^{2j}$ by the geometric series and this yields \eqref{evolution_ocalc}.\\
By a Taylor expansion of the logarithm in the summands of the trace of $\log(D'D''^{-1})$ around one, we get
\begin{align}
\mathrm{Tr} \left[\log\left( D'D''^{-1} \right) \right]= \sum\limits_{k=1}^{4n-2}\log(1+(\delta t)^2c_{kk})= \sum\limits_{k=1}^{4n-2}\frac{1}{1+\xi_{k}}(\delta t)^2c_{kk},
\end{align}
for a vector $\xi\in \mathbb{R}^{4n-2}$ with components that are between zero and $(\delta t)^2c_{kk}$.\\
By cyclicity of the trace operator and because $\mathrm{Tr}[A]=\mathrm{Tr}\left[A^T\right]$ for all matrices $A$, we also get
\begin{align}
&\mathrm{Tr} \left[ D'D''^{-1} - \mathbbm{1}_{\mathbb{R}^{4n-2}} \right]\\
&= (\delta t)^2 \cdot \mathrm{Tr} \left[2 \cdot B_{\mathrm{r}}(\mathbbm{1}_{\mathbb{R}^{4n-2}} + \delta t \cdot L_{\mathrm{r}})^{-1} \right.\\
&\left. + (\delta t)^4\cdot B_{\mathrm{r}}DB_{\mathrm{r}}^T(\mathbbm{1}_{\mathbb{R}^{4n-2}} + \delta t\cdot L_{\mathrm{r}})^{-1}D^{-1}(\mathbbm{1}_{\mathbb{R}^{4n-2}} + \delta t\cdot L_{\mathrm{r}})^{-1T} \right],
\end{align}
which is in $\mathcal{O}((\delta t)^2)$ for $(\delta t)^2 < \omega_{n-1}^{-2}$, again using \eqref{evolution_ocalc} for $(\mathbbm{1}_{\mathbb{R}^{4n-2}} + \delta t\cdot L_{\mathrm{r}})^{-1}$ and $(\mathbbm{1}_{\mathbb{R}^{4n-2}} + \delta t\cdot L_{\mathrm{r}})^{-1T}$.\\
Furthermore, the matrix $B_{\mathrm{r}}$ converges to a diagonal matrix where the entries have a maximal absolute value of $\frac{\omega_{n-1}^2}{2}$, as $\delta t \rightarrow 0$. The largest eigenvalue of $B_{\mathrm{r}}$ is thus bounded by some $\lambda_{\delta t}$ and also this value is bounded by its convergence to $\frac{\omega_{n-1}^2}{2}$. The matrix $B_{\mathrm{r}}$ is the same for each evolution step and only depends on $\delta t$. Therefore, we get
\begin{align}
\mathrm{Tr} \left[  \delta m(u) \delta m(u)^{T}\right] &= \delta m(u)^{T}\delta m(u)\\
&= (\delta t)^4\cdot \left\Vert B_{\mathrm{r}}m(u) \right\Vert_2^2 \leq (\delta t)^4\lambda_{\delta t}^2 \cdot \left\Vert m(u) \right\Vert_2^2\\
&\leq (\delta t)^4\lambda_{\delta t}^2 \cdot \Vert W \Vert^2 \left\Vert u \right\Vert_2^2,
\end{align}
and this bound is independent of the evolution step.\\
All in all we get for fixed $\hat{d}_{\mathrm{r}}(t_{i}) = u \in \mathbb{R}^{2(Y+2)}$, that 
\begin{align}
D\left(\mathcal{P}_{\hat{\phi}_{\mathrm{r}}(t_{i+1})\vert \hat{d}_{\mathrm{r}}(t_{i})}(.,u) \Vert \mathcal{P}_{\hat{\phi}'_{\mathrm{r}}(t_{i+1})\vert \hat{d}_{\mathrm{r}}(t_{i})}(.,u) \right) &\leq \frac{1}{2}(\delta t)^4 \lambda_{\delta t}^2 \cdot \Vert W \Vert^2 \left\Vert u \right\Vert_2^2 + c\cdot (\delta t)^2\label{relative_entropy_bound}
\end{align}
for some $c>0$.\\
If for the measured data $\hat{d}_{\mathrm{r}}(0)\in \mathbb{R}^{2(Y+2)}$ there is an upper bound for the squared norm of the simulated data vectors $\left\Vert\hat{d}_{\mathrm{r}}(t_{i})\right\Vert_2^2$ for all $0 \leq i < 2^N$, then the sum over the relative entropies in \eqref{relative_entropy_bound} over all times is still in $\mathcal{O}(\delta t)$ and thus converges to zero with growing $N$, proportional to $N^{-1}$. This is actually the case as we will see.
\newpage

\section{Data update}
Because the prior, the response and the distribution of the noise in our example are the same for all time steps, the posterior is also equal for all $t_{i}$ with $i=0,...,2^N$.
The update of $\hat{d}_{\mathrm{r}}(t_{i})=u$ to $\hat{d}_{\mathrm{r}}(t_{i+1})=u'$ with IFD, by Corollary \ref{cross_entropy}, is then determined by the minimization of
\begin{align}
&D\left(\mathcal{P}_{\hat{\phi}_{\mathrm{r}}(t_{i+1})\vert \hat{d}_{\mathrm{r}}(t_{i+1})}(.,u') \Vert \mathcal{P}_{\hat{\phi}'_{\mathrm{r}}(t_{i+1})\vert \hat{d}_{\mathrm{r}}(t_{i})}(.,u) \right)\\
&= \frac{1}{2} \mathrm{Tr} \left[ ( \delta m(u,u') \delta m(u,u')^{T} + D) D''^{-1} - \mathbbm{1}_{\mathbb{R}^{4n}} - \log(DD''^{-1})\right]\\
&= \frac{1}{2}\left( (m(u')-m''(u))^{T} D''^{-1} (m(u')-m''(u)\right) + c_{1}\\
&= \frac{1}{2}\left( (Wu'  - m''(u))^{T} D''^{-1} (Wu' - m''(u)\right) + c_{1}\\
&=  \frac{1}{2} u'^{T}(W^{T} D''^{-1}W)u' - m''(u)^{T}D''^{-1}Wu' + c_{2}\\
&= \frac{1}{2} u'^{T}(W^{T} D''^{-1}W)u' - u^{T}W^T (\mathbbm{1}_{\mathbb{R}^{4n}} + \delta t\cdot L_{\mathrm{r}})^T D''^{-1}Wu' + c_{2}\\
&= \frac{1}{2} u'^{T}(W^{T} D''^{-1}W)u' - u^{T}(W^TD''^{-1}W)u' - \delta t \cdot u^{T}W^T L_{\mathrm{r}}^T D''^{-1}Wu'+ c_{2}\label{update_rule}
\end{align}
with respect to $u'$, where $c_1$ and $c_2$ are independent of $u'$.
The update to $u'$ is therefore independent of the time step and given by a matrix which we call $M_{\mathrm{r}}$.

If $(W^{T} D''^{-1}W)$ is positive definite, then $M_{\mathrm{r}}$ is of the form $\mathbbm{1}_{\mathbb{R}^{2(Y+2)}} + \delta t\cdot E$ for some matrix $E\in \mathbb{R}^{2(Y+2)\times 2(Y+2)}$.\\
If not, it either is the zero matrix in $\mathbb{R}^{2(Y+2)\times 2(Y+2)}$, or equal to the matrix \\
$(P_{V^{\perp}_{0}})^{-1}\left(P_{V^{\perp}_{0}}(W^{T} D''^{-1}W)(P_{V^{\perp}_{0}})^{-1}\right)^{-1} P_{V^{\perp}_{0}}W^{T}D''^{-1}m''(u)$ from \eqref{g_second_case} in Definition \ref{substitution_function}, which is of the form $(P_{V^{\perp}_{0}})^{-1} \mathbbm{1}_{V^{\perp}_{0}}P_{V^{\perp}_{0}} + \delta t\cdot E$ for some matrix $E\in \mathbb{R}^{2(Y+2)\times 2(Y+2)}$.\\
In both cases one has
\begin{align}
\Vert M_{\mathrm{r}} \Vert \leq (1+ \delta t \cdot \Vert E \Vert).
\end{align}
For given $\hat{d}_{\mathrm{r}}(0)$, the simulated data in time step $t_{i}$, with $1\leq i \leq 2^N$, is then given by $M_{\mathrm{r}}^{i}\hat{d}_{\mathrm{r}}(0)$.\\
Note that $M_{\mathrm{r}}$ depends on $N$ because $\delta t = \frac{T}{2^N}$.\\
We still need the squared norm of all these vectors to have a common upper bound as seen in \eqref{relative_entropy_bound}, so that the information error coming with our approximation of the exact evolution gets arbitrarily small as $\delta t \rightarrow 0$.\\
But it is
\begin{align}
\left\Vert M_{\mathrm{r}}^{i}\hat{d}_{\mathrm{r}}(0) \right\Vert_2^2 &\leq \left\Vert M_{\mathrm{r}} \right\Vert^{2i}  \left\Vert \hat{d}_{\mathrm{r}}(0) \right\Vert_2^2\\
&\leq (1+ \delta t\cdot \Vert E \Vert)^{2i}\cdot \left\Vert \hat{d}_{\mathrm{r}}(0) \right\Vert_2^2\\
&\leq (1+ \delta t\cdot \Vert E \Vert)^{2N}\cdot \left\Vert \hat{d}_{\mathrm{r}}(0) \right\Vert_2^2 = \left(1+ \dfrac{T}{2^N}\cdot \Vert E \Vert\right)^{2N} \cdot\left\Vert \hat{d}_{\mathrm{r}}(0) \right\Vert_2^2\\
&\leq \exp(T\cdot\Vert E \Vert) \cdot\left\Vert \hat{d}_{\mathrm{r}}(0) \right\Vert_2^2.
\end{align}
Therefore, with the measured data vector $\hat{d}_{\mathrm{r}}(0)$ and the simulated vectors $\hat{d}_{\mathrm{r}}(t_{i}) = M_{\mathrm{r}}^{i}\hat{d}_{\mathrm{r}}(0)$, for $1\leq i \leq 2^N$, by \eqref{relative_entropy_bound} we get
\begin{align}
&\sum\limits_{i=0}^{2^N-1}\left[ D\left(\mathcal{P}_{\hat{\phi}_{\mathrm{r}}(t_{i+1})\vert \hat{d}_{\mathrm{r}}(t_{i})}(.,\hat{d}_{\mathrm{r}}(t_{i})) \Vert \mathcal{P}_{\hat{\phi}'_{\mathrm{r}}(t_{i+1})\vert \hat{d}_{\mathrm{r}}(t_{i})}(.,\hat{d}_{\mathrm{r}}(t_{i})) \right)\right]\\
&\leq \sum\limits_{i=0}^{2^N-1}\left[ \frac{1}{2}(\delta t)^4 \lambda_{\delta t}^2 \cdot\Vert W \Vert^2 \left\Vert \hat{d}_{\mathrm{r}}(t_{i}) \right\Vert_2^2 + c\cdot (\delta t)^2 \right]\\
&\leq \frac{T}{2}(\delta t)^3 \lambda_{\delta t}^2 \cdot \Vert W \Vert^2 \exp(T\cdot \Vert E \Vert) \cdot \left\Vert \hat{d}_{\mathrm{r}}(0) \right\Vert_2^2 + T\cdot c\cdot \delta t,
\end{align}
and this expression converges to zero as $N$ grows. This justifies the approximation of the exact field evolution by a time independent and linear one.\\
For the corresponding $n$ and $N$, $M_{\mathrm{r}}$ can be computed from the minimization problem \eqref{update_rule}. As already mentioned, the simulation of our data vector for time $T$ from $\hat{d}_{\mathrm{r}}(0)$ is then given by $M_{\mathrm{r}}^{2^N}\hat{d}_{\mathrm{r}}(0)$.\\
The data $d(0)$ at time $t_0=0$ has therefore to be transformed into its Fourier representation. One then has to identify $\hat{d}_{\mathrm{c}}(0)$ with the real vector $\hat{d}_{\mathrm{r}}(0)$ and to update it to $\hat{d}_{\mathrm{r}}(T)$. This vector has then to be identified with $\hat{d}_{\mathrm{c}}(T)$ and to be back-transformed to the simulated vector $d(T)$ by another Fourier transform.\\
Now we would like to have a more concrete idea of how \eqref{update_rule} is solved in our example and aim to get a representation of $M_{\mathrm{r}}$ which avoids matrix inversions, as those are extremely expensive to compute numerically. For this, we first show that for $\delta t$ small enough, $(W^{T} D''^{-1}W)$ is indeed invertible. We consider the representation $W=\Phi_{\mathrm{r}} \begin{pmatrix}
\hat{R}_{\mathrm{r}} & 0\\
0 & \hat{R}_{\mathrm{r}}
\end{pmatrix}^T \left[\begin{pmatrix}
\hat{R}_{\mathrm{r}} & 0\\
0 & \hat{R}_{\mathrm{r}}
\end{pmatrix} \Phi_{\mathrm{r}} \begin{pmatrix}
\hat{R}_{\mathrm{r}} & 0\\
0 & \hat{R}_{\mathrm{r}}
\end{pmatrix}^T + N\right]^{-1}$ of the Wiener filter in \eqref{Wienerfilter_representations} and first compute $\begin{pmatrix}
\hat{R}_{\mathrm{r}} & 0\\
0 & \hat{R}_{\mathrm{r}}
\end{pmatrix}\Phi_{\mathrm{r}} \begin{pmatrix}
\hat{R}_{\mathrm{r}} & 0\\
0 & \hat{R}_{\mathrm{r}}
\end{pmatrix}^T$. It turns out that we have
\begin{align}
\left( \hat{R}_{\mathrm{r}}\Phi^{(\varphi)}_{\mathrm{r}} \hat{R}_{\mathrm{r}}^T \right)_{00} &= \frac{2\pi}{\beta \mu^2}\text{ and}\\
\left( \hat{R}_{\mathrm{r}}\Phi^{(\varphi)}_{\mathrm{r}} \hat{R}_{\mathrm{r}}^T \right)_{kl} &= b^{(\varphi)}_{k} \cdot \delta_{kl}\mathbbm{1}_{\mathbb{R}^2} \text{ for } 0<k,l \leq \frac{Y+1}{2},\label{R_Phi_R_varphi}
\end{align}
with
\begin{align}
b^{(\varphi)}_{k}&= \frac{\pi}{\beta}\cdot \sum\limits_{m=1}^{n-1} \omega_{m}^{-2}\mathrm{sinc}^2\left(\frac{1}{2}m\Delta\right)\cdot \left[ \theta(k =m\, \mathrm{mod}\, Y ) + \theta(Y-k = m\, \mathrm{mod}\, Y ) \right]\\
&= \frac{\pi\sin^2\left( \frac{1}{2}k\Delta \right)}{\beta}\cdot \sum\limits_{m=1}^{n-1} \frac{1}{(m^2 + \mu^2)\cdot \frac{\pi^2m^2}{Y^2}} \left[ \theta(k =m\, \mathrm{mod}\, Y ) + \theta(Y-k = m\, \mathrm{mod}\, Y ) \right],\label{b_k_varphi}
\end{align}
which is greater than zero for all $0<k\leq \frac{Y+1}{2}$.\\
This is because the above expression would only be zero for $\sin^2\left( \frac{1}{2}k\Delta \right)=0$. The only roots of $\sin$ are numbers in $\mathbb{Z}\cdot\pi$. But for all $1\leq k \leq \frac{Y+1}{2}$, since $\Delta = \frac{2\pi}{Y}$, it is $\frac{1}{2}k\Delta = \frac{k\pi}{Y}\in \mathbb{Z}\cdot \pi$ only for $k=Y=1$ and we assumed $Y$ to be greater than one.\\
The components of $(\hat{R}_{\mathrm{r}}\Phi^{(\varphi)}_{\mathrm{r}} \hat{R}_{\mathrm{r}}^T + N^{\mathrm{(part)}})^{-1}$ are then given by
\begin{align}
(\hat{R}_{\mathrm{r}}\Phi^{(\varphi)}_{\mathrm{r}} \hat{R}_{\mathrm{r}}^T + N^{\mathrm{(part)}})^{-1}_{00} &= \left( \frac{2\pi}{\beta \mu^2} + \sigma_{n}^2\right)^{-1}\text{ and}\\
(\hat{R}_{\mathrm{r}}\Phi^{(\varphi)}_{\mathrm{r}} \hat{R}_{\mathrm{r}}^T + N^{\mathrm{(part)}})^{-1}_{kl} &=  \delta_{kl} \left(b^{(\varphi)}_k + \sigma_n^2\right)^{-1}\mathbbm{1}_{\mathbb{R}^2}
&\text{for } 0<k,l\leq \frac{Y+1}{2}.\label{varphi_part_W_inversion}
\end{align}
Similarly, one gets 
\begin{align}
\left( \hat{R}_{\mathrm{r}}\Phi^{(\chi)}_{\mathrm{r}} \hat{R}_{\mathrm{r}}^T \right)_{00} &= \frac{2\pi}{\beta}\text{ and}\\
\left( \hat{R}_{\mathrm{r}}\Phi^{(\chi)}_{\mathrm{r}} \hat{R}_{\mathrm{r}}^T \right)_{kl} &= b^{(\chi)}_{k} \cdot \delta_{kl}\mathbbm{1}_{\mathbb{R}^2}\text{ for } 0<k,l\leq \frac{Y+1}{2},\label{R_Phi_R_chi}
\end{align}
with
\begin{align}
b^{(\chi)}_{k}&= \frac{\pi}{\beta}\cdot \sum\limits_{m=1}^{n-1} \mathrm{sinc}^2\left(\frac{1}{2}m\Delta\right)\cdot \left[ \theta(k =m\, \mathrm{mod}\, Y ) + \theta(Y-k = m\, \mathrm{mod}\, Y ) \right]\\
&= \frac{\pi\sin^2\left( \frac{1}{2}k\Delta \right)}{\beta}\cdot \sum\limits_{m=1}^{n-1} \frac{Y^2}{\pi^2m^2} \left[ \theta(k =m\, \mathrm{mod}\, Y ) + \theta(Y-k = m\, \mathrm{mod}\, Y ) \right],\label{b_k_chi}
\end{align}
and also those numbers are greater than zero for all $0<k\leq \frac{Y+1}{2}$.\\
Therefore it is
\begin{align}
\left(\hat{R}_{\mathrm{r}}\Phi^{(\chi)}_{\mathrm{r}} \hat{R}_{\mathrm{r}}^T + N^{\mathrm{(part)}}\right)^{-1}_{00} &= \left( \frac{2\pi}{\beta} + \sigma_{n}^2 \right)^{-1}\text{ and}\\
\left(\hat{R}_{\mathrm{r}}\Phi^{(\chi)}_{\mathrm{r}} \hat{R}_{\mathrm{r}}^T + N^{\mathrm{(part)}}\right)^{-1}_{kl} &=   \delta_{kl} \left( b^{(\chi)}_k + \sigma_n^2 \right)^{-1}\mathbbm{1}_{\mathbb{R}^2} \text{ for } 0<k,l \leq \frac{Y+1}{2}.\label{chi_part_W_inversion}
\end{align}

So now we can compute $W$ as 
\begin{align}
W= \begin{pmatrix}
\Phi^{(\varphi)}_{\mathrm{r}} \hat{R}_{\mathrm{r}}^T (\hat{R}_{\mathrm{r}}\Phi^{(\varphi)}_{\mathrm{r}} \hat{R}_{\mathrm{r}}^T + N^{(\mathrm{part})})^{-1} & 0\\
0 & \Phi^{(\chi)}_{\mathrm{r}} \hat{R}_{\mathrm{r}}^T (\hat{R}_{\mathrm{r}}\Phi^{(\chi)}_{\mathrm{r}} \hat{R}_{\mathrm{r}}^T + N^{(\mathrm{part})})^{-1}
\end{pmatrix}.
\end{align}
We consider the other representation $W= D \begin{pmatrix}
\hat{R}_{\mathrm{r}}^T & 0\\
0 & \hat{R}_{\mathrm{r}}^T
\end{pmatrix}N^{-1}$ of the Wiener filter in \eqref{Wienerfilter_representations}.\\
This together with the fact that for $(\delta t)^2<\omega_{n-1}^{-2}$, it is
\begin{align}
(\mathbbm{1}_{\mathbb{R}^{4n-2}} + \delta t\cdot L_{\mathrm{r}})^{-1} = \mathbbm{1}_{\mathbb{R}^{4n-2}} - \delta t\cdot L_{\mathrm{r}} +\mathcal{O}\left((\delta t)^2\right),
\end{align}
which follows directly from \eqref{evolution_inversion}, and with the corresponding formula for 
\[(\mathbbm{1}_{\mathbb{R}^{4n-2}} + \delta t\cdot L_{\mathrm{r}}^T)^{-1},
\]
we then find that
\begin{align}
W^{T} D''^{-1}W &= N^{-1}\begin{pmatrix}
\hat{R}_{\mathrm{r}} & 0\\
0 & \hat{R}_{\mathrm{r}}
\end{pmatrix} D\left(\mathbbm{1}_{\mathbb{R}^{4n-2}} + \delta t\cdot L^T_{\mathrm{r}}\right)^{-1}D^{-1}\left(\mathbbm{1}_{\mathbb{R}^{4n-2}} + \delta t\cdot L_{\mathrm{r}}\right)^{-1}W\\
&= N^{-1}\begin{pmatrix}
\hat{R}_{\mathrm{r}} & 0\\
0 & \hat{R}_{\mathrm{r}}
\end{pmatrix} W - \delta t \cdot \left[N^{-1}\begin{pmatrix}
\hat{R}_{\mathrm{r}} & 0\\
0 & \hat{R}_{\mathrm{r}}
\end{pmatrix} L_{\mathrm{r}}W + W^TL_{\mathrm{r}}^T\begin{pmatrix}
\hat{R}_{\mathrm{r}} & 0\\
0 & \hat{R}_{\mathrm{r}}
\end{pmatrix}^TN^{-1} \right] + \mathcal{O}((\delta t)^2)\\
&= N^{-1}\begin{pmatrix}
\hat{R}_{\mathrm{r}} & 0\\
0 & \hat{R}_{\mathrm{r}}
\end{pmatrix} \Phi_{\mathrm{r}} \begin{pmatrix}
\hat{R}_{\mathrm{r}} & 0\\
0 & \hat{R}_{\mathrm{r}}
\end{pmatrix}^T \left(\begin{pmatrix}
\hat{R}_{\mathrm{r}} & 0\\
0 & \hat{R}_{\mathrm{r}}
\end{pmatrix}\Phi_{\mathrm{r}} \begin{pmatrix}
\hat{R}_{\mathrm{r}} & 0\\
0 & \hat{R}_{\mathrm{r}}
\end{pmatrix}^T + N\right)^{-1}\\
& - \delta t \cdot \left[N^{-1}\begin{pmatrix}
\hat{R}_{\mathrm{r}} & 0\\
0 & \hat{R}_{\mathrm{r}}
\end{pmatrix}L_{\mathrm{r}}W + W^TL_{\mathrm{r}}^T\begin{pmatrix}
\hat{R}_{\mathrm{r}} & 0\\
0 & \hat{R}_{\mathrm{r}}
\end{pmatrix}^TN^{-1} \right] + \mathcal{O}((\delta t)^2)\\
&=: N^{-1}\begin{pmatrix}
\hat{R}_{\mathrm{r}} & 0\\
0 & \hat{R}_{\mathrm{r}}
\end{pmatrix} \Phi_{\mathrm{r}} \begin{pmatrix}
\hat{R}_{\mathrm{r}} & 0\\
0 & \hat{R}_{\mathrm{r}}
\end{pmatrix}^T \left(\begin{pmatrix}
\hat{R}_{\mathrm{r}} & 0\\
0 & \hat{R}_{\mathrm{r}}
\end{pmatrix}\Phi_{\mathrm{r}} \begin{pmatrix}
\hat{R}_{\mathrm{r}} & 0\\
0 & \hat{R}_{\mathrm{r}}
\end{pmatrix}^T + N\right)^{-1} - \delta t \cdot F,
\end{align}
where $- \delta t \cdot F$ includes $- \delta t \cdot \left[N^{-1}\begin{pmatrix}
\hat{R}_{\mathrm{r}} & 0\\
0 & \hat{R}_{\mathrm{r}}
\end{pmatrix}L_{\mathrm{r}}W + W^TL_{\mathrm{r}}^T\begin{pmatrix}
\hat{R}_{\mathrm{r}} & 0\\
0 & \hat{R}_{\mathrm{r}}
\end{pmatrix}^TN^{-1} \right]$ and all the terms that are in $\mathcal{O}((\delta t)^2)$.\\
$\hat{R}_{\mathrm{r}}\Phi^{(\varphi)}_{\mathrm{r}} \hat{R}_{\mathrm{r}}^T$ and $\hat{R}_{\mathrm{r}}\Phi^{(\chi)}_{\mathrm{r}} \hat{R}_{\mathrm{r}}^T$, given by \eqref{R_Phi_R_varphi} and \eqref{R_Phi_R_chi} are invertible since all the $b^{(\varphi)}_{k}$ and $b^{(\chi)}_k$ from \eqref{b_k_varphi} and \eqref{b_k_chi} are greater than zero. $\left(\hat{R}_{\mathrm{r}}\Phi^{(\varphi)}_{\mathrm{r}} \hat{R}_{\mathrm{r}}^T\right)^{-1}$ and $\left(\hat{R}_{\mathrm{r}}\Phi^{(\chi)}_{\mathrm{r}} \hat{R}_{\mathrm{r}}^T\right)^{-1}$ are then given by \eqref{varphi_part_W_inversion} and \eqref{chi_part_W_inversion} with $\sigma_n^2=0$.\\
This way we get invertibility of $N^{-1}\begin{pmatrix}
\hat{R}_{\mathrm{r}} & 0\\
0 & \hat{R}_{\mathrm{r}}
\end{pmatrix} \Phi_{\mathrm{r}} \begin{pmatrix}
\hat{R}_{\mathrm{r}} & 0\\
0 & \hat{R}_{\mathrm{r}}
\end{pmatrix}^T \left(\begin{pmatrix}
\hat{R}_{\mathrm{r}} & 0\\
0 & \hat{R}_{\mathrm{r}}
\end{pmatrix}\Phi_{\mathrm{r}} \begin{pmatrix}
\hat{R}_{\mathrm{r}} & 0\\
0 & \hat{R}_{\mathrm{r}}
\end{pmatrix}^T + N\right)^{-1}$. By the Neumann series, if $\delta t$ is small enough so that 
\begin{align}
\delta t \cdot \left\Vert N\left[\begin{pmatrix}
\hat{R}_{\mathrm{r}} & 0\\
0 & \hat{R}_{\mathrm{r}}
\end{pmatrix} \Phi_{\mathrm{r}} \begin{pmatrix}
\hat{R}_{\mathrm{r}} & 0\\
0 & \hat{R}_{\mathrm{r}}
\end{pmatrix}^T + N\right]\left[\begin{pmatrix}
\hat{R}_{\mathrm{r}} & 0\\
0 & \hat{R}_{\mathrm{r}}
\end{pmatrix} \Phi_{\mathrm{r}} \begin{pmatrix}
\hat{R}_{\mathrm{r}} & 0\\
0 & \hat{R}_{\mathrm{r}}
\end{pmatrix}^T \right]^{-1} F \right\Vert <1,
\end{align}
we get
\begin{align}
&(W^{T} D''^{-1}W)^{-1}\\
&= \sum\limits_{j=0}^{\infty}  \left[ \delta t \cdot N\left[ \begin{pmatrix}
\hat{R}_{\mathrm{r}} & 0\\
0 & \hat{R}_{\mathrm{r}}
\end{pmatrix} \Phi_{\mathrm{r}} \begin{pmatrix}
\hat{R}_{\mathrm{r}} & 0\\
0 & \hat{R}_{\mathrm{r}}
\end{pmatrix}^T + N\right] \left[ \begin{pmatrix}
\hat{R}_{\mathrm{r}} & 0\\
0 & \hat{R}_{\mathrm{r}}
\end{pmatrix} \Phi_{\mathrm{r}} \begin{pmatrix}
\hat{R}_{\mathrm{r}} & 0\\
0 & \hat{R}_{\mathrm{r}}
\end{pmatrix}^T \right]^{-1}F \right]^j \\
& N \left[ \begin{pmatrix}
\hat{R}_{\mathrm{r}} & 0\\
0 & \hat{R}_{\mathrm{r}}
\end{pmatrix} \Phi_{\mathrm{r}} \begin{pmatrix}
\hat{R}_{\mathrm{r}} & 0\\
0 & \hat{R}_{\mathrm{r}}
\end{pmatrix}^T + N \right]\left[ \begin{pmatrix}
\hat{R}_{\mathrm{r}} & 0\\
0 & \hat{R}_{\mathrm{r}}
\end{pmatrix} \Phi_{\mathrm{r}} \begin{pmatrix}
\hat{R}_{\mathrm{r}} & 0\\
0 & \hat{R}_{\mathrm{r}}
\end{pmatrix}^T \right]^{-1}\\
&= N\left[ \begin{pmatrix}
\hat{R}_{\mathrm{r}} & 0\\
0 & \hat{R}_{\mathrm{r}}
\end{pmatrix} \Phi_{\mathrm{r}} \begin{pmatrix}
\hat{R}_{\mathrm{r}} & 0\\
0 & \hat{R}_{\mathrm{r}}
\end{pmatrix}^T + N \right]\left[ \begin{pmatrix}
\hat{R}_{\mathrm{r}} & 0\\
0 & \hat{R}_{\mathrm{r}}
\end{pmatrix} \Phi_{\mathrm{r}} \begin{pmatrix}
\hat{R}_{\mathrm{r}} & 0\\
0 & \hat{R}_{\mathrm{r}}
\end{pmatrix}^T \right]^{-1} + \mathcal{O}(\delta t).
\end{align}
For $(W^{T} D''^{-1}W)$ is positive definite, the minimization problem \eqref{update_rule} has the unique solution 
\begin{align}
u' &= \left(W^{T} D''^{-1}W\right)^{-1}\left(W^{T} D''^{-1}W\right)u + \delta t\cdot \left(W^{T} D''^{-1}W\right)^{-1}W^TD''^{-1}L_{\mathrm{r}}W u\\
&= u + \delta t \cdot N\left[ \begin{pmatrix}
\hat{R}_{\mathrm{r}} & 0\\
0 & \hat{R}_{\mathrm{r}}
\end{pmatrix}\Phi_{\mathrm{r}} \begin{pmatrix}
\hat{R}_{\mathrm{r}} & 0\\
0 & \hat{R}_{\mathrm{r}}
\end{pmatrix}^T + N \right]\left[ \begin{pmatrix}
\hat{R}_{\mathrm{r}} & 0\\
0 & \hat{R}_{\mathrm{r}}
\end{pmatrix} \Phi_{\mathrm{r}} \begin{pmatrix}
\hat{R}_{\mathrm{r}} & 0\\
0 & \hat{R}_{\mathrm{r}}
\end{pmatrix}^T \right]^{-1}N^{-1} \begin{pmatrix}
\hat{R}_{\mathrm{r}} & 0\\
0 & \hat{R}_{\mathrm{r}}
\end{pmatrix} L_{\mathrm{r}}Wu\\
& + \mathcal{O}((\delta t)^2)\\
&= u + \delta t \cdot\left[\mathbbm{1}_{\mathbb{R}^{2(Y+2)}} + N \left[ \begin{pmatrix}
\hat{R}_{\mathrm{r}} & 0\\
0 & \hat{R}_{\mathrm{r}}
\end{pmatrix} \Phi_{\mathrm{r}} \begin{pmatrix}
\hat{R}_{\mathrm{r}} & 0\\
0 & \hat{R}_{\mathrm{r}}
\end{pmatrix}^T \right]^{-1} \right] \begin{pmatrix}
\hat{R}_{\mathrm{r}} & 0\\
0 & \hat{R}_{\mathrm{r}}
\end{pmatrix} L_{\mathrm{r}} \Phi_{\mathrm{r}} \begin{pmatrix}
\hat{R}_{\mathrm{r}} & 0\\
0 & \hat{R}_{\mathrm{r}}
\end{pmatrix}^T\\
& \left[ \begin{pmatrix}
\hat{R}_{\mathrm{r}} & 0\\
0 & \hat{R}_{\mathrm{r}}
\end{pmatrix} \Phi_{\mathrm{r}} \begin{pmatrix}
\hat{R}_{\mathrm{r}} & 0\\
0 & \hat{R}_{\mathrm{r}}
\end{pmatrix}^T + N\right]^{-1}u + \mathcal{O}((\delta t)^2).\label{solution}
\end{align}
Therefore, assuming $\delta t$ small, we approximate $u'$ by $M_{\mathrm{r}}u$, with
\begin{align}
&M_{\mathrm{r}}:= \mathbbm{1}_{\mathbb{R}^{2(Y+2)}} + \delta t \cdot\left[\mathbbm{1}_{\mathbb{R}^{2(Y+2)}} + N\left[ \begin{pmatrix}
\hat{R}_{\mathrm{r}} & 0\\
0 & \hat{R}_{\mathrm{r}}
\end{pmatrix} \Phi_{\mathrm{r}} \begin{pmatrix}
\hat{R}_{\mathrm{r}} & 0\\
0 & \hat{R}_{\mathrm{r}}
\end{pmatrix}^T \right]^{-1} \right] \begin{pmatrix}
\hat{R}_{\mathrm{r}} & 0\\
0 & \hat{R}_{\mathrm{r}}
\end{pmatrix} L_{\mathrm{r}} \Phi_{\mathrm{r}} \begin{pmatrix}
\hat{R}_{\mathrm{r}} & 0\\
0 & \hat{R}_{\mathrm{r}}
\end{pmatrix}^T \\
&\left[ \begin{pmatrix}
\hat{R}_{\mathrm{r}} & 0\\
0 & \hat{R}_{\mathrm{r}}
\end{pmatrix}\Phi_{\mathrm{r}} \begin{pmatrix}
\hat{R}_{\mathrm{r}} & 0\\
0 & \hat{R}_{\mathrm{r}}
\end{pmatrix}^T + N\right]^{-1}\\
&= \mathbbm{1}_{\mathbb{R}^{2(Y+2)}} + \delta t \cdot\left[\mathbbm{1}_{\mathbb{R}^{2(Y+2)}} + \sigma_n^2 \begin{pmatrix}
\left(\hat{R}_{\mathrm{r}}\Phi^{(\varphi)}_{\mathrm{r}}\hat{R}_{\mathrm{r}}^T\right)^{-1} & 0\\
0 & \left(\hat{R}_{\mathrm{r}}\Phi^{(\chi)}_{\mathrm{r}}\hat{R}_{\mathrm{r}}^T\right)^{-1}
\end{pmatrix}  \right] \begin{pmatrix}
\hat{R}_{\mathrm{r}} & 0\\
0 & \hat{R}_{\mathrm{r}}
\end{pmatrix} L_{\mathrm{r}} \Phi_{\mathrm{r}} \begin{pmatrix}
\hat{R}_{\mathrm{r}}^T & 0\\
0 & \hat{R}_{\mathrm{r}}^T
\end{pmatrix} \\
&\begin{pmatrix}
\left(\hat{R}_{\mathrm{r}}\Phi^{(\varphi)}_{\mathrm{r}} \hat{R}_{\mathrm{r}}^T + N^{(\mathrm{part})}\right)^{-1} & 0\\
0 & \left(\hat{R}_{\mathrm{r}}\Phi^{(\chi)}_{\mathrm{r}} \hat{R}_{\mathrm{r}}^T + N^{(\mathrm{part})}\right)^{-1}
\end{pmatrix},
\end{align}
and this matrix can be computed by matrix multiplications only, as all the inversions could be solved analytically.\\
Dividing both sides in \eqref{solution} by $\delta t$ and letting $\delta t \rightarrow 0$, with $u= u(t)$ and $u'=u(t + \delta t)$, we also get a differential equation
\begin{align}
\dot{u}(t) &= \lim\limits_{\delta t \rightarrow 0} \frac{u(t+\delta t)-u(t)}{\delta t}\\
& = \left[\mathbbm{1}_{\mathbb{R}^{2(Y+2)}} + \sigma_n^2 \begin{pmatrix}
\left(\hat{R}_{\mathrm{r}}\Phi^{(\varphi)}_{\mathrm{r}}\hat{R}_{\mathrm{r}}^T\right)^{-1} & 0\\
0 & \left(\hat{R}_{\mathrm{r}}\Phi^{(\chi)}_{\mathrm{r}}\hat{R}_{\mathrm{r}}^T\right)^{-1}
\end{pmatrix} \right] \begin{pmatrix}
\hat{R}_{\mathrm{r}} & 0\\
0 & \hat{R}_{\mathrm{r}}
\end{pmatrix} L_{\mathrm{r}} \Phi_{\mathrm{r}} \begin{pmatrix}
\hat{R}_{\mathrm{r}}^T & 0\\
0 & \hat{R}_{\mathrm{r}}^T
\end{pmatrix} \\
&\begin{pmatrix}
\left(\hat{R}_{\mathrm{r}}\Phi^{(\varphi)}_{\mathrm{r}} \hat{R}_{\mathrm{r}}^T + N^{(\mathrm{part})}\right)^{-1} & 0\\
0 & \left(\hat{R}_{\mathrm{r}}\Phi^{(\chi)}_{\mathrm{r}} \hat{R}_{\mathrm{r}}^T + N^{(\mathrm{part})}\right)^{-1}
\end{pmatrix}u(t)\\
&=: M'_{\mathrm{r}}u(t).
\end{align}
The solution of this equation at time $t$, with initial condition $u(0)$, is then given by
\begin{align}
u(t) = \exp\left( t\cdot M'_{\mathrm{r}}\right)u(0),
\end{align}
which yields us the compact simulation equation
\begin{align}
\hat{d}_{\mathrm{r}}(T) = \exp\left( T\cdot M'_{\mathrm{r}}\right)\hat{d}_{\mathrm{r}}(0)
\end{align} 
for our data at time $T$, if we measured $\hat{d}_{\mathrm{r}}(0)$.\\
Because $M'_{\mathrm{r}}$ is independent of $\delta t$, this update rule is continuous in time. $M'_{\mathrm{r}}$ has to be computed only once. For the simulation of a data vector at time $T$, one then only has to calculate $\exp\left( T\cdot M'_{\mathrm{r}}\right)$ and to multiply this matrix to the measured data vector $\hat{d}_{\mathrm{r}}(0)$.
\chapter{Conclusion}
\label{chap:Conclusion}
In the following, the novelties and the content of this work are summarized. First of all, the measure theoretical fundament and the necessary probabilistic framework for IFD were introduced in Chapter \ref{chap:Background on measure theory}. This way, the fundament for a mathematical interpretation of the physical language in \citep{enss} was set.
Amongst others, this included a mathematical definition for the prior $\mathcal{P}_{\phi}$, the posterior $\mathcal{P}_{\phi \vert d}$, the evidence $\mathcal{P}_{d}$ and the likelihood $\mathcal{P}_{d \vert \phi}$ as density functions, in case $\phi$ is the unknown quantity (signal) and $d$ the known one (data). Bayes's Theorem for this mathematical setting of IFD was derived and proven.\\
In Chapter \ref{chap:setting}, after a description of the setting for IFD, a general example scenario for the approximation of a physical process by a simplified model was constructed. This illustrated the steps to be done until a suitable setup for the simulation with IFD is derived.\\
The real process has to be approximated by a step function in time which is contained in a finite dimensional space for every time step, and the real evolution has to be approximated by a linear one.\\
In Section \ref{Wiener filter}, the derivation of the posterior from a Gaussian prior and a Gaussian zero centered noise distribution with help of the Wiener filter was interpreted in terms of the mathematical language from Chapter \ref{chap:Background on measure theory}.\\
For this framework and a partition $\left((t_{i},t_{i+1}]\right)_{i=0}^{2^N}$ of the time interval $(0,T]$, where $T$ is the time for which the data is simulated, in Chapter \ref{chap:Updating the data} one updating step of IFD was described in detail, following \cite{enss} and imbedding the work by Torsten En\ss lin in a mathematical setting.\\
In this part, approximation errors were indicated and the convergence with finer and finer partitions of $[0,T]$ and a growing number of degrees of freedom of the signal was shown. This was done by proving that the information theoretical error for $\delta t = \frac{T}{2^N}$ is in $\mathcal{O}((\delta t)^2)$ in one simulation step, ensuring the sum over all approximation errors within the whole simulation to be still in $\mathcal{O}(\delta t)$. For the simulation of one data evolution step by multiplying a matrix to the data from the previous step, the convex minimization problem, which arises within entropic matching, was solved in detail. In contribution to \citep{enss}, also the case of a possibly non-positive definite Hessian matrix was taken into account. The chapter ended with a receipt for the whole algorithm produced by IFD. Also the connection between accuracy and the number of the signal's degrees of freedom as well as the number of steps within the time discretization was explained here.\\
Chapter \ref{chap:Maximum_Entropy_Principle} contains an explanation for the concept of differential entropy as a measure of information and the approach of density determination by minimizing the relative entropy between two densities.\\
In Chapter \ref{chap:Example: Klein-Gordon field}, the whole procedure of IFD was illustrated for the Klein-Gordon field in one dimension in space and periodic over $[0,2\pi)$. The data for arbitrary $t\in [0,T]$ was averaging the signal $\phi(t)$ within a partition of the interval $[0,2\pi)$. The noise was assumed to be Gaussian distributed for all times, with zero mean and covariance $N=\sigma_{n}^2\mathbbm{1}$, for $\sigma_{n}^2>0$. Instead of the assumption of no noise, as made in \citep{enss}, measurement inaccuracies where taken into account in the example of this work, to ensure the concept of Wiener filtering with its information propagator $D$ and its information source $j(d)$ to make sense at all. The vague ansatz in \citep{enss} for this example was elaborated in detail to provide a suitable illustration for the theoretical part of the work. At first, the specific solutions of the Klein-Gordon equation, considered in the example, were specified. Afterwards, a correction of the prior from \citep{enss} was made, taking into account dependencies between negative and positive Fourier coefficients of the signal and the fact that the zeroth Fourier mode has to be a real number in case the signal is real valued. Also the correct evolution of a Klein-Gordon field was elaborated and compared to the linear approximation up to first order in the step size, in terms of relative entropy between the evolving posterior densities. With the partition $\left((t_{i},t_{i+1}]\right)_{i=0}^{2^N}$ of $(0,T]$, this way, convergence of the information theoretical error to zero, as $\delta t = \frac{T}{2^N}\rightarrow 0$, was shown.\\
Not only the matrix connection between the old data $d(t_{i})$ and the new one $d(t_{i+1})$ for a single updating step from $t_{i}$ to $t_{i+1}$ was derived, but also a direct simulation equation between the data $d(T)$ at the time of interest and the measured vector $d(0)$ at the initial time could be constructed.\\
This construction was rooted in the fact that the prior, the response and the noise distribution were assumed to be equal for all the evolution steps and that the exact evolution could be approximated by a linear one sufficiently well. This way, the matrix connection between known and simulated data turned out to be the same for all time steps and to converge to an ordinary differential equation for the data for infinitesimally small intervals in the partition of $(0,T]$. The solution of this equation resulted in the direct simulation equation for the data.\\
All matrices could be expressed by simple matrix multiplications and sums over matrices, in order to avoid complicated numerical matrix inversions.\\
The derivation of the direct simulation of the data can easily be generalized for physical processes, for which a linear approximation of the real evolution produces an information theoretical error in $\mathcal{O}((\delta t)^2)$, and where the response, the Gaussian prior and the Gaussian distributed noise are assumed to be constant in time, as it is the case in this example.\\
All in all, IFD has the advantage of not assuming any explicit subgrid structure of the signal and is therefore suitable for the construction of simulation algorithms for a large number of evolution equations.

\chapter{Dictionary for physicists}
\label{chap:Dictionary for physicists}
This chapter contains a short dictionary for a better understanding of the notation in this work. On the left side of the following tabular one finds the notation used in this work, whereas on the right side the corresponding translation into an expression which physicists might be used to is given.
The signal $\phi$ and the data $d$ will always be random vectors in $\mathbb{R}^{n}$ and  $\mathbb{R}^{m}$ for $n,m\in \mathbb{N}$. $f$ is always a measurable function into $\mathbb{R}$ and $A$ a measurable set.\\

\begin{tabular}{l|l}
\underline{mathematical expression} & \underline{physical expression} \medskip \\
$\mathbb{E}[f(\phi)] = \mathbb{E}_{\phi}[f]= \int\limits_{\mathbb{R}^{n}} f(s) \mathbb{P}_{\phi}(ds)$     &     $\langle f(\phi) \rangle_{(\phi)}=\int\mathcal{D}\phi\,  f(\phi)$\\
$\mathbb{E}[\phi;\lbrace \phi \in A \rbrace] = \int\limits_{\phi^{-1}(A)} \phi d\mathbb{P} = \int\limits_{A} s \mathbb{P}_{\phi}(ds)$      &     $\int\mathcal{D}\phi\,  \phi \cdot 1_{A}(\phi)$\\
(prior) $\mathcal{P}_{\phi}:  \mathbb{R}^{n} \rightarrow \mathbb{R}$ & $\mathcal{P}(\phi)$ \\
(posterior) $\mathcal{P}_{\phi \vert d}: \mathbb{R}^{n}\times \mathbb{R}^{m} \rightarrow \mathbb{R}$ & $\mathcal{P}(\phi \vert d)$ \\
(likelihood) $\mathcal{P}_{d \vert \phi}: \mathbb{R}^{m}\times \mathbb{R}^{n} \rightarrow \mathbb{R}$ & $\mathcal{P}(d \vert \phi)$ \\
(evidence) $\mathcal{P}_{d}: \mathbb{R}^{m} \rightarrow \mathbb{R}$ & $\mathcal{P}(d)$ \\
(information Hamiltonian) $H(d,\phi):\mathbb{R}^{m}\times \mathbb{R}^{n} \rightarrow \mathbb{R}\cup \lbrace\infty \rbrace$ & $H(d,\phi)$ \\
(Gaussian prior) $\mathcal{G}(s-\psi,\Phi) = \frac{1}{\vert 2\pi \Phi \vert^{1/2}} \exp\left( -\frac{1}{2}(s-\psi)^{\dagger} \Phi^{-1} (s-\psi) \right)$ & $\mathcal{G}(\phi-\psi,\Phi)$
\end{tabular}

\end{document}